\pdfoutput=1
\documentclass[a4paper,10pt,reqno]{amsart} 
\usepackage{graphicx}
\usepackage{amssymb}
\usepackage{epstopdf}
\usepackage{mathrsfs}
\usepackage{amsmath}
\usepackage{amsthm}
\usepackage{enumerate}
\usepackage[all,cmtip,color]{xy}
\usepackage{mathtools}
\usepackage{stmaryrd}
\usepackage{bbm}
\usepackage{xcolor}
\usepackage{enumitem} 
\usepackage{tcolorbox}
\usepackage[makeroom]{cancel}
\usepackage{tikz}
\usepackage{tikz-cd}
 \usepackage{hyperref}

\newcommand{\Lmod}[1]{#1\text{-}{\mathsf{mod}}}
\newcommand{\Nmod}[1]{#1\text{-}{\mathsf{nmod}}}
\newcommand{\FPmod}[1]{#1\text{-}{\mathsf{fpmod}}}
\newcommand{\Fmod}[1]{#1\text{-}{\mathsf{fmod}}}
\newcommand{\Gmod}[1]{#1\text{-}{\mathsf{gmod}}}

\newtheoremstyle{plainx}
  {3pt} 
  {3pt} 
  {\itshape} 
  {} 
  {\bfseries} 
  {.} 
  {.5em} 
  {} 

\newtheoremstyle{definitionx}
  {3pt} 
  {3pt} 
  {} 
  {} 
  {\bfseries} 
  {.} 
  {.5em} 
  {} 

\theoremstyle{plainx}
\newtheorem{thm}{Theorem}[section]
\newtheorem{pro}[thm]{Proposition}
\newtheorem{cor}[thm]{Corollary}
\newtheorem{lem}[thm]{Lemma}

\newtheorem{con}[thm]{Conjecture} 
\newtheorem*{question*}{Question}
\newtheorem{thmm}{Theorem}[section]
\newtheorem{corr}[thmm]{Corollary}

\theoremstyle{definitionx} 
\newtheorem{defi}[thm]{Definition}
\newtheorem{rem}[thm]{Remark}
\newtheorem{exa}[thm]{Example}

\usepackage{accents}
\newlength{\dhatheight}
\newcommand{\doublehat}[1]{%
    \settoheight{\dhatheight}{\ensuremath{\hat{#1}}}%
    \addtolength{\dhatheight}{-0.35ex}%
    \hat{\vphantom{\rule{1pt}{\dhatheight}}%
    \smash{\hat{#1}}}}

\makeatletter
\DeclareRobustCommand\widecheck[1]{{\mathpalette\@widecheck{#1}}}
\def\@widecheck#1#2{%
    \setbox\z@\hbox{\m@th$#1#2$}%
    \setbox\tw@\hbox{\m@th$#1%
       \widehat{%
          \vrule\@width\z@\@height\ht\z@
          \vrule\@height\z@\@width\wd\z@}$}%
    \dp\tw@-\ht\z@
    \@tempdima\ht\z@ \advance\@tempdima2\ht\tw@ \divide\@tempdima\thr@@
    \setbox\tw@\hbox{%
       \raise\@tempdima\hbox{\scalebox{1}[-1]{\lower\@tempdima\box
\tw@}}}%
    {\ooalign{\box\tw@ \cr \box\z@}}}
\makeatother

\def \h {\mathfrak{h}}
\def \Z {\mathbb{Z}}

\def \C {\mathbb{C}}

\def \g {\mathfrak{g}}
\def \n {\mathfrak{n}}
\def \t {\mathfrak{t}}
\def \ad {\textnormal{ad}}

\def \a {\mathfrak{a}}
\def \b {\mathfrak{b}}

\def \O {\mathcal{O}}

\def \l {\mathfrak{l}}

\def \id {\mbox{\textnormal{id}}}

\def \supp {\mbox{\textnormal{supp}}}

\def \Mat {\mbox{\textnormal{Mat}}}
\def \d {\mathfrak{d}}

\def \UU {\widehat{\mathbf{U}}_\kappa}
\def \Uu {\mathbf{U}}
\def \UUc {\widehat{\mathbf{U}}_{\mathbf{c}}}

\DeclareMathOperator{\Ima}{Im}
\DeclareMathOperator{\Ann}{Ann}

\DeclareMathOperator{\End}{End}
\DeclareMathOperator{\Hom}{Hom}

\DeclareMathOperator{\Spec}{Spec}

\DeclareMathOperator{\gr}{\mathsf{gr}}

\DeclareMathOperator{\Ind}{Ind}
\DeclareMathOperator{\Res}{Res}

\DeclareMathOperator{\Tr}{Tr}
\DeclareMathOperator{\Kil}{Kil}

\DeclareMathOperator{\Op}{Op}
\DeclareMathOperator{\Loc}{Loc}

\DeclareMathOperator{\Sym}{Sym}

\DeclareMathOperator{\Ext}{Ext}
\DeclareMathOperator{\colim}{colim}

\DeclareMathOperator{\MaxSpec}{Maxspec}
\DeclareMathOperator{\coker}{coker}
\DeclareMathOperator{\Inf}{Inf}

\newcommand{\arxiv}[1]{\href{http://arxiv.org/abs/#1}{\tt arXiv:\nolinkurl{#1}}}

\newcommand{\Grad}{\nabla\!\!\!\!\nabla}

\newcommand{\Hecke}{\mathcal{H\kern-.47em H}}
\newcommand{\Heckedeg}{\mathbb{H\kern-.38em H}}

\begin{document}
\title{Suzuki functor at the critical level}
\author{Tomasz Przezdziecki}
\date{} 

\begin{abstract} 
In this paper we define and study a critical-level generalization of the Suzuki functor, relating the affine general linear Lie algebra to the rational Cherednik algebra of type A. Our main result states that this functor  induces a surjective algebra homomorphism from the centre of the completed universal enveloping algebra at the critical level to the centre of the rational Cherednik algebra at t=0. 
We use this homomorphism to obtain several results about the functor. We compute it on Verma modules, Weyl modules, and their restricted versions. We describe the maps between endomorphism rings induced by the functor and deduce that every simple module over the rational Cherednik algebra lies in its image. Our homomorphism between the two centres gives rise to a closed embedding of the Calogero-Moser space into the space of opers on the punctured disc. We give a partial geometric description of this embedding. 
\end{abstract}

\maketitle

\tableofcontents

\section{Introduction}

Arakawa and Suzuki \cite{AS} introduced a family of functors from the category $\O$ for $\mathfrak{sl}_n$ to the category of finite-dimensional representations of the degenerate affine Hecke algebra associated to the symmetric group $\mathfrak{S}_m$. These functors have been generalized in many different ways, connecting the representation theory of various Lie algebras with the representation theory of various degenerations of affine and double affine Hecke algebras. 
\begin{center}
\begin{figure}[h]
\begin{tabular}{ | c | c | c | }
\hline 
Lie algebra & ``Hecke'' algebra &  \\
\hline
$\mathfrak{sl}_n$ & degenerate affine Hecke algebra & Arakawa-Suzuki \cite{AS} \\
$\widehat{\mathfrak{sl}}_n$ & trigonometric DAHA & Arakawa-Suzuki-Tsuchiya \cite{AST} \\
$\widehat{\mathfrak{gl}}_n$ 
& rational DAHA ($t\neq0$) 
& Suzuki \cite{Suz} \\
$\widehat{\mathfrak{gl}}_n$ 
& cyclotomic rational DAHA ($t\neq0$)
& Varagnolo-Vasserot \cite{VV} \\ 
\hline
\end{tabular} 
\caption{Functors relating Lie algebras and ``Hecke'' algebras in type A}
\end{figure}
\end{center}
\vspace{-15pt} 

Other generalizations of the Arakawa-Suzuki functor may be found in, e.g., \cite{BK,CEE,ES,EFM,Jor1,Jor2,OR}. 
Here we are concerned with the third functor in the table above, introduced by Suzuki, and later studied by Varagnolo and Vasserot \cite{VV}, under the assumption that $t \neq 0$, and the level $\kappa$ is not critical. It is a functor 
\begin{equation} \label{intro Suzuki first form} \mathsf{F}_\kappa \colon \mathscr{C}_\kappa \to \Lmod{\mathcal{H}_{\kappa+n}} \end{equation} 
from the category $\mathscr{C}_\kappa$ of smooth $\widehat{\mathfrak{gl}}_n$-modules of level $\kappa$ to the category of modules over the rational Cherednik algebra $\mathcal{H}_{\kappa+n}$ (also known as the rational DAHA) associated to $\mathfrak{S}_m$ and parameters $t= \kappa+n$, $c=1$. It assigns to each $\widehat{\mathfrak{gl}}_n$-module a certain space of coinvariants:
\[ M \mapsto H_0(\mathfrak{gl}_n[z],\C[x_1, \hdots, x_m] \otimes (\mathbf{V}^*)^{\otimes m} \otimes M).\]
In this paper we study the limit of the functor $\mathsf{F}_\kappa$ as 
\[ \kappa \to \mathbf{c} = -n, \quad t \to 0.\]

The representation theory of the rational Cherednik algebra at $t=0$ differs radically from its representation theory at $t \neq 0$, mainly due to the fact that $\mathcal{H}_0$ has a large centre~$\mathcal{Z}$, whose spectrum can be identified with the classical Calogero-Moser space \cite{EG}. An analogous pattern occurs in the representation theory of $\hat{\g}:=\widehat{\mathfrak{gl}}_n$; 
the centre of the completed universal enveloping algebra $\widehat{\mathbf{U}}_\kappa$ of $\hat{\g}$ is trivial unless the level is critical. In the latter case, the centre $\mathfrak{Z}$ of $\UUc$ is a completion of a polynomial algebra in infinitely many variables, and, by a theorem of Feigin and Frenkel \cite{FF}, it can be identified with the algebra of functions on the space of opers on the punctured disc. 

The existence of an interesting connection between the two centres $\mathcal{Z}$ and $\mathfrak{Z}$, or, equivalently, between the Calogero-Moser space and opers, is suggested by the close relationship between the Calogero-Moser integrable system and the KP hierarchy. 
For example, Ben-Zvi and Nevins \cite{BN} investigated this relationship from the perspective of noncommutative geometry, identifying the Calogero-Moser space with a certain moduli space of sheaves, called micro-opers, on quantized cotangent bundles. 
There is also a more direct connection between $\mathcal{Z}$ and $\mathfrak{Z}$ via the Bethe algebra of the Gaudin model associated to $\g$. 
By the work of Chervov and Talalaev \cite{CT}, the Bethe algebra can be obtained as the image of $\mathfrak{Z}$ under the canonical projection from $\widehat{\mathbf{U}}_{\mathbf{c}}$ to $\mathbf{U}(\g[t^{-1}])$. 
A surjective homomorphism from the Bethe algebra to the centre of the rational Cherednik algebra was later constructed by Mukhin, Tarasov and Varchenko \cite{MTV}. 

Inspired by these intriguing connections, we study the relationship between the two centres from a more algebraic point of view. We consider $\mathcal{Z}$ and $\mathfrak{Z}$ as centres of the respective categories of modules and show that the functor $\mathsf{F}_{\mathbf{c}}$ induces (in a sense which will be made precise below) a surjective algebra homomorphism $\Theta \colon \mathfrak{Z} \twoheadrightarrow \mathcal{Z}$. This homomorphism encodes a lot of information about the functor, allowing us to deduce a number of interesting results (see Corollaries \ref{cor 1 intro}-\ref{cor 5 intro}). For example, we are able to prove that every simple $\mathcal{H}_0$-module is in the image of $\mathsf{F}_{\mathbf{c}}$, describe the maps between endomorphism rings induced by $\mathsf{F}_{\mathbf{c}}$, and compute the functor on Arakawa and Fiebig's restricted category~$\O$. Furthermore, we interpret $\Theta$ as an embedding of the Calogero-Moser space into the space of opers on the punctured disc and provide a partial geometric description of this embedding. We expect that there is a connection between our approach and the work of Mukhin, Tarasov and Varchenko, but we do not understand this connection precisely.

\subsection{Generalization of the Suzuki functor.} 

Our first theorem, which collects the results of Corollary \ref{cor: Suzuki all levels} and \S \ref{sec: ext to all modules} below, yields a generalization of the functor \eqref{intro Suzuki first form} originally defined by Suzuki. 

\newcounter{tmp}
\begingroup
\setcounter{tmp}{\value{thmm}}
\setcounter{thmm}{0} 
\renewcommand\thethmm{\Alph{thmm}}

\begin{thmm} \label{thm 1 intro}
For all $\kappa \in \C$, there is a colimit preserving functor 
\[ \mathsf{F}_{\kappa} \colon \Lmod{\widehat{\mathbf{U}}_\kappa} \to \Lmod{\mathcal{H}_{\kappa+n}}.\]
When $\kappa \neq \mathbf{c}$, the restriction of this functor to $\mathscr{C}_\kappa$ coincides with  \eqref{intro Suzuki first form}. 
\end{thmm}

Our next result describes the images of some important $\UU$-modules under the functor~$\mathsf{F}_{\kappa}$. Let us briefly explain the motivation for studying these modules. It comes from the representation theory of the rational Cherednik algebra. 
 
It was proven in \cite{EG} that isomorphism classes of simple $\mathcal{H}_0$-modules are in bijection with maximal ideals in $\mathcal{Z}:=Z(\mathcal{H}_0)$. 
Moreover, every simple $\mathcal{H}_0$-module occurs as a quotient of a generalized Verma module $\Delta_{0}(a,\lambda)$, introduced in \cite{Bel}. These modules can be defined for any $t \in \C$, and depend on a vector $a \in \C^m$, together with an irreducible representation $\lambda$ of a parabolic subgroup of $\mathfrak{S}_m$. When $a=0$, they are the usual Verma modules for $\mathcal{H}_t$. 
The following theorem shows that generalized Verma modules as well as the regular module are in the image of the functor $\mathsf{F}_{\kappa}$. 
\begin{thmm}[{Theorems \ref{thm: regular module}-\ref{thm: Vermas to standards}}] \label{thm 2 intro} 
Let $\kappa \in \C$. 
There exist $\widehat{\mathbf{U}}_\kappa$-modules $\mathbb{H}_\kappa$ and $\mathbb{W}_{\kappa}(a,\lambda)$ such that 
\[ \mathsf{F}_{\kappa}(\mathbb{H}_\kappa) =  \mathcal{H}_{\kappa+n}, \quad  \mathsf{F}_\kappa(\mathbb{W}_{\kappa}(a,\lambda)) = \Delta_{\kappa+n}(a,\lambda).\]
Moreover, $$\mathsf{F}_{\kappa}(\mathbb{M}_{\kappa}(\lambda)) = \Delta_{\kappa+n}(\lambda).$$ 
\end{thmm}

Here $\mathbb{M}_{\kappa}(\lambda)$ denotes the Verma module for $\hat{\g}$. 
When $a=0$, the modules $\mathbb{W}_{\kappa}(\lambda):=\mathbb{W}_{\kappa}(0,\lambda)$ coincide with the Weyl modules from \cite{KLI}. Therefore, we call $\mathbb{W}_{\kappa}(a,\lambda)$ ``generalized Weyl modules".

\subsection{Suzuki functor and the centres.}

From now on assume that $n=m$. 
One of our main goals is to understand how the centres of the categories $\widehat{\mathbf{U}}_{\mathbf{c}}\Lmod{}$ and $\mathcal{H}_{0}\Lmod{}$ behave under the functor $\mathsf{F}_{\mathbf{c}}$. This is of vital importance because the centres, to a large extent, control morphisms in these categories. For example, it was shown in \cite{FG} that the endomorphism rings of Verma and Weyl modules for $\mathbf{U}_{\mathbf{c}}(\hat{\g})$ are quotients of $\mathfrak{Z}$. 

In general, a functor of additive categories does not induce a homomorphism between their centres. We circumvent this problem by introducing the notions of an $F$-centre of a category and an $F$-central subcategory. More precisely, 
we consider the canonical maps 
\begin{equation*} \mathfrak{Z} \cong \Lmod{Z(\widehat{\mathbf{U}}_{\mathbf{c}}})  \overset{\alpha}{\longrightarrow} \End(\mathsf{F}_{\mathbf{c}}) \overset{\beta}{\longleftarrow} Z(\mathcal{H}_{0}\Lmod{}) \cong \mathcal{Z}.\end{equation*}
from the two centres to the endomorphism ring of the functor $\mathsf{F}_{\mathbf{c}}$. 
Since $\mathcal{H}_0$ lies in the image of $\mathsf{F}_{\mathbf{c}}$, the map $\beta$ is injective and $\mathcal{Z}$ can be identified with the subring $\Ima \beta$ of $\End(\mathsf{F}_{\mathbf{c}})$. 
We call $Z_{\mathsf{F}_{\mathbf{c}}}(\widehat{\mathbf{U}}_{\mathbf{c}}) := \alpha^{-1}(\mathcal{Z}) \subset \mathfrak{Z}$ the $\mathsf{F}_{\mathbf{c}}$-\emph{centre} of $\widehat{\mathbf{U}}_{\mathbf{c}}\Lmod{}$. Restricting $\alpha$ to $Z_{\mathsf{F}_{\mathbf{c}}}(\widehat{\mathbf{U}}_{\mathbf{c}})$ gives  a natural algebra homomorphism
\begin{equation*} Z(\mathsf{F}_{\mathbf{c}}) := \alpha|_{Z_{\mathsf{F}_{\mathbf{c}}}(\widehat{\mathbf{U}}_{\mathbf{c}})} \colon \  Z_{\mathsf{F}_{\mathbf{c}}}(\widehat{\mathbf{U}}_{\mathbf{c}}) \longrightarrow \mathcal{Z}\end{equation*} 
making the diagram
\begin{equation} \label{centre endo diagram intro}
\begin{tikzcd}
 Z_{\mathsf{F}_{\mathbf{c}}}(\widehat{\mathbf{U}}_{\mathbf{c}}) \arrow{r}{Z(\mathsf{F}_{\mathbf{c}})} \arrow{d}[swap]{can} & \mathcal{Z} \arrow{d}{can} \\
\End_{\UUc}(M) \arrow{r}{\mathsf{F}_{\mathbf{c}}} & \End_{\mathcal{H}_0}(\mathsf{F}_{\mathbf{c}}(M))
\end{tikzcd}
\end{equation}
commute for all $\UUc$-modules $M$. 
The homomorphism $Z(\mathsf{F}_{\mathbf{c}})$ contains partial information about all the maps between endomorphism rings induced by the functor $\mathsf{F}_{\mathbf{c}}$. 
 
Our next result gives a partial description of $Z_{\mathsf{F}_{\mathbf{c}}}(\widehat{\mathbf{U}}_{\mathbf{c}})$. 
We consider the subalgebra $\mathscr{L}_{\mathbf{c}} := \C[\id[r], {}^{\mathbf{c}}\mathbf{L}_{r+1}]_{r \leq 0} \subset \mathfrak{Z}$ consisting of certain first- and second-order Segal-Sugawara operators (see \S \ref{subsec:quadraticSSvector} for a precise definition). 

\begin{thmm}[{Theorem \ref{pro: question one}}] \label{intro thm q1}
The algebra $\mathscr{L}_{\mathbf{c}}$ lies in the $\mathsf{F}_{\mathbf{c}}$-centre of $\widehat{\mathbf{U}}_{\mathbf{c}}\Lmod{}$, i.e., $$\mathscr{L}_{\mathbf{c}} \subseteq  Z_{\mathsf{F}_{\mathbf{c}}}(\widehat{\mathbf{U}}_{\mathbf{c}}).$$
\end{thmm} 

We give an explicit description of the associated homomorphism 
\begin{equation} \label{alpha beta intro} Z(\mathsf{F}_{\mathbf{c}})|_{\mathscr{L}_{\mathbf{c}}} \colon \mathscr{L}_{\mathbf{c}} \to \mathcal{Z} \end{equation} 
in \eqref{id[r]}-\eqref{L_r complete homogeneous}. 

It is natural to ask whether $Z_{\mathsf{F}_{\mathbf{c}}}(\widehat{\mathbf{U}}_{\mathbf{c}})$ coincides with $\mathfrak{Z}$.
Unfortunately, this is far from being the case. Our solution to this problem is to relax the condition that the diagram~\eqref{centre endo diagram intro} should commute for all $\UUc$-modules $M$. 
We introduce the notion of a subcategory $\mathcal{A}$ of $\widehat{\mathbf{U}}_{c}\Lmod{}$ being $\mathsf{F}_{c}$-\emph{central} (see Definition \ref{fc central cat defi} for details), which has the consequence that there exists a unique algebra homomorphism $Z_{\mathcal{A}}(\mathsf{F}_{c}) \colon \mathfrak{Z} \to \mathcal{Z}$ making the diagram
\begin{equation} \label{centre endo diagram} 
\begin{tikzcd}
 \mathfrak{Z} \arrow{r}{Z_{\mathcal{A}}(\mathsf{F}_{\mathbf{c}})} \arrow{d}[swap]{can} & \mathcal{Z} \arrow{d}{can} \\
\End_{\UUc}(M) \arrow{r}{\mathsf{F}_{\mathbf{c}}} & \End_{\mathcal{H}_0}(\mathsf{F}_{\mathbf{c}}(M))
\end{tikzcd}
\end{equation}
commute for all $M \in \mathcal{A}$. 
Our next result identifies an important $\mathsf{F}_{\mathbf{c}}$-central subcategory of $\widehat{\mathbf{U}}_{\mathbf{c}}\Lmod{}$. 

\begin{thmm}[{Theorem \ref{thm restricted centres}}] \label{intro thm q2}
The full subcategory $\mathscr{C}_{\mathbb{H}}$ of $\widehat{\mathbf{U}}_{\mathbf{c}}\Lmod{}$ projectively generated by $\mathbb{H}_{\mathbf{c}}$ is $\mathsf{F}_{\mathbf{c}}$-central. 
\end{thmm} 

The category $\mathscr{C}_{\mathbb{H}}$ contains all the Verma and generalized Weyl modules which are not annihilated by $\mathsf{F}_{\mathbf{c}}$. 
The associated homomorphism
\[ \Theta = Z_{\mathscr{C}_{\mathbb{H}}}(\mathsf{F}_{\mathbf{c}}) \colon \ \mathfrak{Z} \longrightarrow \mathcal{Z} \]
plays a key role in our study of the functor~$\mathsf{F}_{\mathbf{c}}$. 
The following theorem, whose representation theoretic and geometric consequences are discussed in the next subsection, 
is the main result of this paper. 


\begin{thmm}[{Theorem \ref{thm: Psi is surj}}] \label{main theorem intro}
The homomorphism $\Theta \colon \mathfrak{Z} \to \mathcal{Z}$ is surjective. 
\end{thmm}
  
Let us briefly comment on the proof of Theorem \ref{main theorem intro}. We first show that $\Theta$ factors through $\mathfrak{Z}^{\leqslant 2}(\hat{\g})$ (see \S \ref{section: opers} for the definition), and that the homomorphism $\Theta \colon \mathfrak{Z}^{\leqslant 2}(\hat{\g}) \to \mathcal{Z}$ is filtered with respect to the standard filtration on $\mathcal{Z}$ and a certain ``height" filtration on $\mathfrak{Z}^{\leqslant 2}(\hat{\g})$ (see \S \ref{subsec: U filtr} and  \S \ref{assoc graded sec})
We compute the associated graded homomorphism $\gr \Theta$ and use it to deduce the surjectivity of $\Theta$. In our calculations, we rely heavily on the explicit construction of Segal-Sugawara operators due to Chervov and Molev \cite{CM}. 

We also consider the Poisson algebra structures on $\mathfrak{Z}$ and $\mathcal{Z}$ given by the Hayashi bracket \cite{Hay}. The map $\Theta$ is not a Poisson homomorphism. However, the following is true. 

\begin{thmm}[{Theorem \ref{pro: Poisson homo}}] 
The restriction of $\Theta$ to $\mathscr{L}_{\mathbf{c}}$ is a homomorphism of Poisson algebras. 
\end{thmm}

The partial compatibility of the Poisson structures on $\mathfrak{Z}$ and $\mathcal{Z}$ is a shadow of the fact that the functor $\mathsf{F}_{\kappa}$ is defined for all levels $\kappa$. 
We remark that the Poisson subalgebra $\mathscr{L}_{\mathbf{c}} \subset \mathfrak{Z}$ can be described quite explicitly. It is isomorphic to a certain subalgebra of $S(\mathsf{Heis} \rtimes \mathsf{Vir})$, the symmetric algebra on the semi-direct product of the Heisenberg and the Virasoro Lie algebras.

\setcounter{thmm}{\thetmp} 

\subsection{Applications.} 

Our main result (Theorem \ref{main theorem intro}) has several applications. First of all, we can use it to gain more information about the homomorphisms between endomorphism rings induced by $\mathsf{F}_{\mathbf{c}}$. 

\begin{corr}[{Corollary \ref{cor: surjective on endo rings}}] \label{cor 1 intro}
The ring homomorphisms 
\[  \End_{\UUc} (\mathbb{W}_{\mathbf{c}}(a,\lambda)) \twoheadrightarrow \End_{\mathcal{H}_{0}}(\Delta_{0}(a,\lambda)), \quad  \End_{\UUc} (\mathbb{M}_{\mathbf{c}}(\lambda)) \twoheadrightarrow \End_{\mathcal{H}_{0}}(\Delta_{0}(\lambda)).\]
induced by $\mathsf{F}_{\mathbf{c}}$ are surjective. 
\end{corr} 

Secondly, we are able to deduce from Corollary \ref{cor 1 intro} that every simple $\mathcal{H}_0$-module lies in the image of $\mathsf{F}_{\mathbf{c}}$. This result is, on the one hand, analogous to similar results \cite{Suz, VV} in the $\kappa \neq \mathbf{c}$ case. On the other hand, the situation at the critical level is very different because there are uncountably many non-isomorphic simple $\mathcal{H}_0$-modules. This is reflected by the fact that our proof relies on completely different techniques from those used in \cite{Suz, VV}. 

\begin{corr}[{Corollary \ref{all simples}}]
Every simple $\mathcal{H}_0$-module is in the image of the functor~$\mathsf{F}_{\mathbf{c}}$. 
\end{corr}

We next connect the functor $\mathsf{F}_{\mathbf{c}}$ with the work of Arakawa and Fiebig. 
In \cite{AF1, AF2}, they studied a restricted version of category $\O$, obtained by ``killing" the action of the centre $\mathfrak{Z}$. This category contains restricted Verma modules $\overline{\mathbb{M}}_{\mathbf{c}}(\lambda)$ as well as, analogously defined, restricted versions of Weyl modules $\overline{\mathbb{W}}_{\mathbf{c}}(\lambda)$. In our third corollary, we describe the image of these modules under~$\mathsf{F}_{\mathbf{c}}$. 

\begin{corr}[{Corollaries \ref{cor: res Vermas}-\ref{cor: restricted Weyl calc}}] 
We have 
\[ \mathsf{F}_{\mathbf{c}}(\overline{\mathbb{M}}_{\mathbf{c}}(\lambda)) = \mathsf{F}_{\mathbf{c}}(\overline{\mathbb{W}}_{\mathbf{c}}(\lambda)) = \mathsf{F}_{\mathbf{c}}(\mathbb{L}(\lambda)) = L_\lambda, \] 
where $\mathbb{L}(\lambda)$ (resp.\ $L_\lambda$) is the unique graded simple quotient of $\mathbb{M}_{\mathbf{c}}(\lambda)$ 
(resp.\ $\Delta_0(\lambda)$). 
\end{corr} 

Fourthly, we give a partial geometric description of the homomorphism $\Theta \colon \mathfrak{Z} \to \mathcal{Z}$ in terms of opers. By a theorem of Feigin and Frenkel \cite{FF}, $\mathfrak{Z}$ is canonically isomorphic to the algebra of functions on the space $\Op_{\check{G}}(\mathbb{D}^\times)$ of opers on the punctured disc. Therefore, $\Theta$ induces a closed embedding $\Theta^* \colon \Spec \mathcal{Z} \hookrightarrow \Op_{\check{G}}(\mathbb{D}^\times)$. We show that the image of this embedding lies in the space $\Op_{\check{G}}(\mathbb{D})^{\leq 2}$ of opers with singularities of order at most two. 

We are also able to obtain some information about the residue and monodromy of the opers in the image of $\Theta^*$. To state our results, we first need to recall some facts about the affine variety $\Spec \mathcal{Z}$ and a canonical map $\pi \colon \Spec \mathcal{Z} \to \C^n/\mathfrak{S}_n$ (see \eqref{definition of pi}). Bellamy showed in \cite{Bel, Bel2} that each fibre of $\pi$ decomposes as a disjoint union of 
subvarieties $\Omega_{\mathbf{a},\lambda}$, which can be identified with supports of the generalized Verma modules $\Delta_{0}(a,\lambda)$. Moreover, $\mathcal{Z}$ surjects onto the endomorphism rings $\End_{\mathcal{H}_{0}}(\Delta_{0}(a,\lambda))$, and $\Spec \End_{\mathcal{H}_{0}}(\Delta_{0}(a,\lambda)) \cong \Omega_{\mathbf{a},\lambda} $. 

Endomorphism rings of the Weyl modules $\mathbb{W}_{\mathbf{c}}(\lambda)$ also admit a geometric interpretation.  
Frenkel and Gaitsgory \cite{FG} showed that $\mathfrak{Z}$ surjects onto $\End_{\UUc}(\mathbb{W}_{\mathbf{c}}(\lambda))$, and identified the latter with the algebra of functions on the space $\Op_{\check{G}}^\lambda(\mathbb{D})$ of opers with residue $\varpi(-\lambda - \rho)$ and trivial monodromy. 

Using the results of \cite{FFTL}, we show that the image of $\Omega_{\mathbf{a},\lambda}$ under $\Theta^*$ is contained in the space $\Op_{\check{G}}^{\leqslant 2}(\mathbb{D})_{\mathbf{a}}$ of opers with singularities of order at most two and $2$-residue $\mathbf{a}$. Moreover, we show that the image of $\Omega_\lambda$ is contained in $\Op_{\check{G}}^\lambda(\mathbb{D})$.

\begin{corr}[{Corollary \ref{big cor on opers}}] The following hold. 
\begin{enumerate}[label=\alph*), font=\textnormal,noitemsep,topsep=3pt,leftmargin=1cm]
\item The map $\Theta \colon \mathfrak{Z} \to \mathcal{Z}$ induces a closed embedding 
\[ \Theta^* \colon \Spec \mathcal{Z} \hookrightarrow \Op_{\check{G}}(\mathbb{D})^{\leqslant 2}. \] 
\item We have
\begin{equation*} \Theta^*(\Omega_{\mathbf{a},\lambda}) \subseteq \Op_{\check{G}}^{\leqslant 2}(\mathbb{D})_{\mathbf{a}}.\end{equation*} 
Hence the following diagram commutes: 
\begin{equation*} 
\begin{tikzcd}
\Spec \mathcal{Z} \arrow[swap, "\pi"]{d} \arrow[hookrightarrow, "\Theta^*"]{r} & \Op_{\check{G}}(\mathbb{D})^{\leqslant 2} \arrow[ "\Res_2"]{d} \\ 
\C^n/\mathfrak{S}_n \arrow["\sim"]{r} & \t^*/\mathfrak{S}_n
\end{tikzcd}
\end{equation*} 
\item If $\mathbf{a} = 0$ then 
\begin{equation*} \Theta^*(\Omega_{\lambda}) \subseteq \Op_{\check{G}}^\lambda(\mathbb{D}). \end{equation*} 
\end{enumerate} 
\end{corr}

Finally, we study the behaviour of self-extensions under $\mathsf{F}_{\mathbf{c}}$. 

\begin{corr} \label{cor 5 intro}
Suppose that $M$ is a $\UUc$-module with a filtration by Weyl modules. Then $\mathsf{F}_{\mathbf{c}}$ induces a linear map 
\begin{equation} \label{intro exts} \Ext_{\UUc}^1(M,M) \to \Ext_{\mathcal{H}_{0}}^1(\mathsf{F}_{\mathbf{c}}(M),\mathsf{F}_{\mathbf{c}}(M)). \end{equation}
\end{corr}

We conjecture (see Conjecture \ref{the conjecture}) that \eqref{intro exts} extends to a surjective homomorphism between extension algebras, and that it admits an interpretation in terms of differential forms on opers and the Calogero-Moser space.

\subsection{Structure of the paper.} 

Let us finish by summarizing the contents of the paper. In sections 2-4 we recall the relevant definitions and facts concerning affine Lie algebras, rational Cherednik algebras and vertex algebras. These sections contain no new results. In Section 5 we recall Suzuki's construction of the functor $\mathsf{F}_\kappa$ and generalize it to the critical level. In section 6 we further generalize the functor $\mathsf{F}_\kappa$ to the category of all $\UU$-modules, proving Theorem \ref{thm 1 intro}. Section 7 is devoted to the proof of Theorem \ref{thm 2 intro}. In Section 8 we study the relationship between the two centres $\mathfrak{Z}$ and $\mathcal{Z}$ via the functor $\mathsf{F}_{\mathbf{c}}$. Section 8 contains the proofs of Theorems \ref{intro thm q1}-\ref{intro thm q2}. In Section 9 we define graded and filtered analogues of the Suzuki functor, which are later used in Section 10 to set up our ``associated graded'' argument. All of section 10 is devoted to the proof of Theorem \ref{main theorem intro}. 
In Section 11 we study the applications of Theorem \ref{main theorem intro}, proving Corollaries \ref{cor 1 intro}-\ref{cor 5 intro}. 

\subsection*{Acknowledgements.} 
The research for this paper was carried out with the financial support of the College of Science \& Engineering at the University of Glasgow and the Max Planck Institute for Mathematics in Bonn. The material will form part of the author's PhD thesis. 
I would like to thank G. Bellamy for recommending the problem to me, many useful suggestions and comments, as well as his unwavering support and encouragement throughout the time in which this paper was written. I am also grateful to C. Stroppel for stimulating discussions as well as numerous and detailed comments on draft versions of this paper. Finally, I would like to thank A. Molev for discussing his paper \cite{CM} with~me, and the two anonymous referees for extremely detailed and insightful comments. 

\section{Preliminaries} \label{sec: affine Lie algebras}

\subsection{General conventions.} \label{general conventions} 
Fix once and for all two positive integers $n$ and $m$. The parameter $n$ refers to the Lie algebra $\g = \mathfrak{gl}_n$ while $m$ refers to the rational Cherednik algebra $\mathcal{H}_{t,c}$ associated to the symmetric group $\mathfrak{S}_m$. 
We work over the field of complex numbers throughout. 
If $V$ is a vector space, let $T(V)$ denote the tensor algebra and $S(V)$ the symmetric algebra on $V$. 

For a unital associative algebra $A$, with unit $1_A$, we denote by $A\Lmod{}$ the category of left $A$-modules. 
Given a left $A$-module $M$ and a left ideal $I$ in $A$, let $M^I := \{ m \in M \mid I \cdot M = 0\}$ be the set of $I$-invariants. 
We will also work with the full subcategory $A\FPmod{}$ of $A\Lmod{}$ consisting of finitely presented modules, i.e., modules $M$ such that there exists a short exact sequence $A^k \to A^l \to M \to 0$ for some $k,l \geq 0$. 
If $B$ is another algebra, let $(A,B)\Nmod{}$ be the full subcategory of $A \otimes B\Lmod{}$ consisting of modules $M$ with the property that the action of $B$ normalizes the action of $A$, i.e., $[A,B] \subseteq A$ in the endomorphism ring of $M$. 

Given a subalgebra $B \subset A$, let $Z_A(B)$ denote the centralizer of $B$ in $A$. In particular, $Z(A) := Z_A(A)$ is the centre of $A$. Recall that the centre $Z(\mathcal{C})$ of an additive category $\mathcal{C}$ is the endomorphism ring of the identity functor $\id_{\mathcal{C}}$. We can naturally identify $Z(A) \cong Z(A\Lmod{})$, $z \mapsto \{ z_M \mid M \in A\Lmod{}\}$, where $z_M$ is the endomorphism of $M$ given by the left action of $z$. 

Suppose that $A$ is a commutative algebra and $M$ is an $A$-module. Let $\Ann_A(M) := \{ a \in A \mid a \cdot M = 0\}$ be the annihilator of $M$ in $A$. The affine variety $\supp_A(M) :=\Spec A/\Ann_A(M)$ is called the support of $M$ in $\Spec A$.

\subsection{Combinatorics.} \label{weights and partitions}
Let $l \geq 1$.  
We say that $\nu=(\nu_1,\hdots,\nu_l) \in \Z_{+}^l$ is a \emph{composition} of $m$ of \emph{length} $l$ if $\nu_1 + \hdots + \nu_l = m$. Let $\mathcal{C}_l(m)$ denote the set of all such compositions.~Set 
$\nu_{\leq i} = \nu_1 + \hdots + \nu_i$ for each $1 \leq i \leq l$ 
with $\nu_{\leq 0} = 0$ by convention.

The symmetric group $\mathfrak{S}_m$ on $m$ letters acts naturally on $\mathfrak{h} = \C^m$ by permuting the coordinates. If $a \in \h$, let $\mathfrak{S}_m(a)$ denote its stabilizer in $\mathfrak{S}_m$. For $1 \leq i,j \leq m$, let $s_{i,j}$ be the simple transposition swapping $i$ and $j$. We abbreviate $s_i := s_{i,i+1}$. Given $\nu \in \mathcal{C}_l(m)$, let $\mathfrak{S}_\nu :=\mathfrak{S}_{\nu_1} \times \hdots \times \mathfrak{S}_{\nu_l}$ denote the parabolic subgroup of $\mathfrak{S}_m$ generated by the simple transpositions $s_1, \hdots , s_{m-1}$ with the omission of $s_{\nu_{\leq 1}}, s_{\nu_{\leq 2}}, \hdots, s_{\nu_{\leq l-1}}$. 

A sequence $\lambda = (\lambda_1, \hdots, \lambda_n) \in \Z_{\geq 0}^n$ is a \emph{partition} if $\lambda_1 \geq \hdots \geq \lambda_n$. Let $\mathcal{P}_n(m)$ denote the set of all partitions of $m$ of length $n$. 
We call $\lambda = (\lambda^1, \hdots, \lambda^l) \in \prod_{i=1}^l \mathcal{P}_{n_i}(m_i)$ an $l$-\emph{multipartition} of $m$ if $\sum_{i=1}^l m_i = m$ and each $m_i \neq 0$. 
We say that $\lambda$ has \emph{length} $n$ if $\sum_{i=1}^l n_i = n$, and \emph{length type} $\mu$ if $(n_1, \hdots, n_l) = \mu \in \mathcal{C}_l(n)$. We say that $\lambda$ is of \emph{size type} $\nu$ if $(m_1, \hdots, m_l) = \nu \in \mathcal{C}_l(m)$. 
Let $\mathcal{P}_\mu(m)$ denote the set of multipartitions of $m$ of length type $\mu$ and let $\mathcal{P}_n(\nu)$ denote the set of all multipartitions of length $n$ of size type $\nu$ (where we let $l$ vary over all positive integers). 
Set 
\[ \mathcal{P}_\mu(\nu) := \mathcal{P}_\mu(m) \cap \mathcal{P}_n(\nu), \quad \mathcal{P}_\mu := \bigsqcup_{m \geq 0} \mathcal{P}_\mu(m), \quad \mathcal{P}(\nu) := \bigcup_{n \geq 0} \mathcal{P}_n(\nu).\]
In the union on the RHS we identify $l$-multipartitions $\lambda$ and $\chi$ whenever each pair of partitions $\lambda^i$ and $\chi^i$ differ only by the number of parts equal to zero.

If $\lambda \in \mathcal{P}_n(m)$, let $\mathsf{Sp}(\lambda)$ denote the corresponding Specht module. 
Given $\nu \in \mathcal{C}_l(m)$ and $\lambda \in \mathcal{P}_n(\nu)$, set  
$\mathsf{Sp}(\lambda) := \mathsf{Sp}(\lambda^1) \otimes \hdots \otimes \mathsf{Sp}(\lambda^l)$. It is a $\mathfrak{S}_\nu$-module. Let $\mathsf{Sp}_\nu(\lambda) := \C \mathfrak{S}_m \otimes_{\C\mathfrak{S}_\nu}\mathsf{Sp}(\lambda)$ be the corresponding $\mathfrak{S}_m$-module obtained by induction.

\subsection{Lie algebras.}
Given a Lie algebra $\mathfrak{a}$, 
let $\mathbf{U}(\mathfrak{a})$ denote its universal enveloping algebra, with unit $1_{\mathfrak{a}}:=1_{\mathbf{U}(\mathfrak{a})}$ and  augmentation ideal $\mathbf{U}_+(\mathfrak{a})$. If $M$ is an $\mathfrak{a}$-module and $k \geq 0$, let $H_k(\mathfrak{a},M)$ denote the $k$-th homology group of $\mathfrak{a}$ with coefficients in $M$. In particular, $H_0(\mathfrak{a},M) = M/\mathbf{U}_+(\mathfrak{a}).M = M/\mathfrak{a}.M$. 
Given a Lie subalgebra $\mathfrak{c}\subset \a$ and a $\mathfrak{c}$-module $N$, let $\Ind_{\mathfrak{c}}^{\mathfrak{a}} N := \mathbf{U}(\mathfrak{a}) \otimes_{\mathbf{U}(\mathfrak{c})} N$ be the induced module. For a surjective Lie algebra homomorphism $\mathfrak{d} \twoheadrightarrow \mathfrak{c}$, let $\Inf_{\mathfrak{c}}^{\mathfrak{d}} N$ denote $N$ regarded as a $\mathfrak{d}$-module. 

Let $G = GL_n(\C)$ be the general linear group and $\g = \mathfrak{gl}_n(\C)$ its Lie algebra. 
Let $e_{kl}$ be the $(k,l)$-matrix unit and let $\id$ denote the identity matrix. 
We use the standard triangular decomposition $\g = \n_- \oplus \t \oplus \n_+$  with respect to the strictly lower triangular, diagonal and strictly upper triangular matrices, and abbreviate $\b_+ := \t \oplus \n_+$. 
For $1 \leq k \leq n$, let $\epsilon_k \in \t^*$ be the function defined by $\epsilon_k(e_{ll}) = \delta_{k,l}$. 

Given $\mu \in \mathcal{C}_l(n)$, 
let $\mathfrak{l}_{\mu} := \prod_{i=1}^l \mathfrak{gl}_{\mu_i} \subseteq \g$ be the corresponding standard Levi subalgebra. 
We next recall the connection between multipartitions and weights. A weight $\lambda = \sum_i \lambda_i \epsilon_i \in \t^*$ is called $\mu$\emph{-dominant} and integral if each $\lambda_i \in \Z$ and $\lambda_i - \lambda_{w(i)} \in \Z_{\geq 0}$ whenever $w(i) > i$, for all $w \in \mathfrak{S}_\mu$. 
Let $\Pi_\mu^+$ denote the set of $\mu$-dominant integral weights with the property that each $\lambda_i \in \Z_{\geq 0}$. 
If $\mu = (n)$, we abbreviate $\Pi_\mu^+ = \Pi^+$. 
There is a natural bijection 
\begin{equation} \label{multipart vs part} 
\Pi_\mu^+ \cong \mathcal{P}_\mu, \quad 
\lambda \mapsto (\lambda^1,\hdots,\lambda^l),\end{equation}
where $\lambda^i:=( \lambda_{\mu_{\leq i-1}+1},\hdots, \lambda_{\mu_{\leq i}})$. 
From now on we will implicitly identify weights with partitions using this bijection. 

\subsection{Schur-Weyl duality.} \label{subsec: SW duality}

Given $\lambda \in \Pi_\mu^+$, let $L(\lambda)$ be the corresponding simple $\mathfrak{l}_\mu$-module of highest weight $\lambda$. Let $\mathbf{V} \cong L(\epsilon_1)$ 
be the standard representation of $\g$, with standard basis $\{e_i\mid 1 \leq i \leq n\}$ and the corresponding dual basis $\{e_i^*\mid 1 \leq i \leq n\}$ of $\mathbf{V}^*$. 
If $n=m$, set $e_{\mathsf{id}}^*:= e_1^* \otimes \hdots \otimes e_n^*$ and, for $w \in \mathfrak{S}_n$, 
\begin{equation} e_{w}^* := e_{w^{-1}(1)}^* \otimes \hdots \otimes e_{w^{-1}(n)}^* \in (\mathbf{V}^*)^{\otimes n}. \end{equation}
Given $\mu \in \mathcal{C}_l(n)$ and $\nu \in \mathcal{C}_l(m)$, let $\mathbf{V}_{i}^*$ be the subspace of $\mathbf{V}$ spanned by $e_{\mu_{\leq i-1}+1}^*, \hdots, e_{\mu_{\leq i}}^*$ and $(\mathbf{V}^*)^{\otimes m}_{(\mu,\nu)}:= \bigotimes_{i=1}^l(\mathbf{V}_i^*)^{\otimes \nu_i} \subseteq (\mathbf{V}^*)^{\otimes m}$. 

There is an analogue of classical Schur-Weyl duality (see, e.g., \cite[Proposition 9.1.2]{Proc}) for $\mathfrak{l}_\mu$ and $\mathfrak{S}_m \ltimes \Z_l^m$ - their actions on $\mathbf{V}^{\otimes m}$ centralize each other (see, e.g., \cite[Theorem 6.1]{KP}). We will need the following application, whose proof can be found in \cite[Proposition 3.8(a)]{VV}. 

\begin{pro} \label{Schur-Weyl pro}

Let $\lambda \in \t^*$. Then 
\begin{enumerate}[label=\alph*), font=\textnormal,noitemsep,topsep=3pt,leftmargin=1cm]
\itemsep0em
\item $H_0(\l_{\mu}, (\mathbf{V}^*)^{\otimes m} \otimes L(\lambda)) = 0$ unless $\lambda \in \mathcal{P}_\mu(m)$. 
\item If $\nu \in \mathcal{C}_l(m)$ and $\lambda \in \mathcal{P}_\mu(\nu)$ then 
\begin{equation} \label{SW duality iso} H_0(\l_{\mu}, (\mathbf{V}^*)_{(\mu,\nu)}^{\otimes m} \otimes L(\lambda)) \cong \mathsf{Sp}(\lambda), \quad H_0(\l_{\mu}, (\mathbf{V}^*)^{\otimes m} \otimes L(\lambda)) \cong \mathsf{Sp}_\nu(\lambda)\end{equation} 
as $\C\mathfrak{S}_\nu$- resp.\ $\C\mathfrak{S}_m$-modules. 
\end{enumerate} 
\end{pro}

In the case $\mu = (n)$, classical Schur-Weyl duality also implies the following. 

\begin{cor} \label{lem: SW duality form 2} 
Let $\lambda \in \t^*$. 
Then:
\begin{enumerate}[label=\alph*), font=\textnormal,noitemsep,topsep=3pt,leftmargin=1cm]
\itemsep0em
\item $H_0(\b_+, (\mathbf{V}^*)^{\otimes m} \otimes \C_{\lambda}) = 0$ unless $\lambda \in \mathcal{P}_n(m)$. 
\item If $\lambda \in \mathcal{P}_n(m)$ then there is a natural $\C \mathfrak{S}_m$-module isomorphism 
\begin{equation} \label{SW duality Borel} H_0(\b_+, (\mathbf{V}^*)^{\otimes m} \otimes \C_{\lambda}) \cong \mathsf{Sp}(\lambda).\end{equation} 
\end{enumerate} 
\end{cor}

\begin{proof}
The space $H_0(\b_+, (\mathbf{V}^*)^{\otimes m} \otimes \C_{\lambda})$ can be identified with the space of lowest weight vectors of weight $-\lambda$ in $(\mathbf{V}^*)^{\otimes m}$. It follows from Schur-Weyl duality that $(\mathbf{V}^*)^{\otimes m} = \bigoplus_{\xi \in \mathcal{P}_n(m)} L(\xi)^* \otimes \mathsf{Sp}(\xi)$. Since the lowest weight in each $L(\xi)^*$ is equal to $-\xi$,  the space of lowest weight vectors of weight $-\lambda$ in $(\mathbf{V}^*)^{\otimes m}$ is isomorphic to $ \mathsf{Sp}(\lambda)$ if $\lambda \in \mathcal{P}_n(m)$ and is zero otherwise. 
\end{proof}

\subsection{The affine Lie algebra.} \label{Lie g notation} 
\label{subsec: affine Lie algebras, level} 

We recall the definition of the affine Lie algebra associated to $\g$.  

\begin{defi} \label{defi:themixedform} 
Let $\kappa \in \C$. The \emph{affine Lie algebra} $\hat{\g}_\kappa$ is the central extension
\begin{equation} \label{Central extension} 0 \to \C \mathbf{1} \to \hat{\g}_\kappa \to \g((t)) \to 0\end{equation}
associated to the cocycle $(X\otimes f, Y \otimes g) \mapsto \langle X, Y \rangle_\kappa \Res_{t=0}(g\partial_tf)$, where 
\begin{equation*} \label{Bilinear forms} 
\langle - , - \rangle_{\kappa} := 
\kappa \Tr(XY) + \Tr(X)\Tr(Y). 
\end{equation*} 
Note that $\langle - , - \rangle_{-n} = -\frac{1}{2}\Kil_{\g}$, where $\Kil_{\g}$ is the Killing form on $\g$. 
Explicitly, the Lie bracket in $\hat{\mathfrak{g}}_\kappa$ is given by: 
\begin{equation*} [X\otimes f, Y \otimes g] = [X,Y] \otimes fg + \langle X,Y\rangle_\kappa \Res_{t=0}(g\partial_tf) \mathbf{1}, \quad [X\otimes f,\mathbf{1}]=[\mathbf{1},\mathbf{1}]=0 \end{equation*}
for $X,Y \in \g$ and $f,g \in \C((t))$. 
\end{defi}

We will also use the central extension $\tilde{\g}_\kappa$ obtained by replacing $\g((t))$ with $\g[t^{\pm1}]$ in~\eqref{Central extension}. 
Given $X \in \g$ and $k \in \Z$, set $$X[k] := X \otimes t^k \in \hat{\g}_\kappa, \quad \g[k] := \g \otimes t^k \subset \hat{\g}_\kappa.$$ 
We next introduce notation for the following Lie subalgebras of $\hat{\g}_\kappa$: \[ \hat{\g}_{-} := \g \otimes t^{-1}\C[t^{-1}], \quad \hat{\g}_+ :=\g[[t]] \oplus \C\mathbf{1}, \quad \hat{\g}_{\geq r} := \g\otimes t^r\C[[t]], \quad \hat{\g}_{\leq -r} := \g \otimes t^{-r}\C[t^{-1}],\] where $r \geq 0$. 
Moreover, we abbreviate  
\[ \hat{\n}_+ := \n_+ \oplus \hat{\g}_{\geq 1}, \quad 
 \hat{\b}_+ := \hat{\n}_+ \oplus \t \oplus \C\mathbf{1}, \quad \hat{\t}_+ := \t \oplus \hat{\g}_{\geq 1} \oplus \C\mathbf{1}.\] 
Let $\tilde{\g}_+, \tilde{\g}_{\geq r}$, etc., denote the corresponding Lie subalgebras of $\tilde{\g}_\kappa$.

\subsection{The completed universal enveloping algebra.} \label{sec: cuea}

We are interested in modules on which $\mathbf{1}$ acts as the identity endomorphism. Therefore we consider the quotient algebra
$$\mathbf{U}_\kappa(\hat{\g}) := \mathbf{U}(\hat{\g}_\kappa)/ \langle\mathbf{1} - 1_{\hat{\g}_\kappa}\rangle.$$ 
\begin{defi}
The parameter $\kappa$ is called the \emph{level}. The value $\mathbf{c} := -n$ is called the \emph{critical level}. 
\end{defi} 

We next recall the definition of a certain completion of $\mathbf{U}_\kappa(\hat{\g})$ (see, e.g., \cite[\S 2.1.2]{Fre}). 
There is a topology on $\mathbf{U}_\kappa(\hat{\g})$ defined by declaring the left ideals $I_r:= \mathbf{U}_\kappa(\hat{\g}).\hat{\g}_{\geq r}$ ($r \geq 0$) to be a basis of open neighbourhoods of zero. Let $\widehat{\mathbf{U}}_\kappa$ be the completion of $\mathbf{U}_\kappa(\hat{\g})$ with respect to this topology. Equivalently, we can write 
\begin{equation} \label{completion cuea} \widehat{\mathbf{U}}_\kappa = \varprojlim \mathbf{U}_\kappa(\hat{\g})/I_r. \end{equation}
It is a complete topological algebra with a basis of open neighbourhoods of zero given by the left ideals $\hat{I}_r:=\widehat{\mathbf{U}}_\kappa.\hat{\g}_{\geq r}$. The following proposition illustrates the special nature of the critical level. 

\begin{pro}[{\cite[Proposition 4.3.9]{Fre}}] \label{pro: Z crit level}
$Z(\widehat{\mathbf{U}}_\kappa) = \C$ if and only if $\kappa \neq \mathbf{c}$.
\end{pro}

We abbreviate  
$$\mathfrak{Z}:=Z(\widehat{\mathbf{U}}_{\mathbf{c}}).$$ 

\subsection{Smooth modules.} 

Throughout the paper we will mostly deal with smooth $\widehat{\mathbf{U}}_\kappa$-modules. Let us recall their definition (see, e.g., \cite[\S 1.3.6]{Fre} or \cite[\S 1.9]{KLI}). 
\begin{defi}
A $\widehat{\mathbf{U}}_\kappa$-module $M$ 
is called \emph{smooth} if $M = \bigcup_{r\geq 0} M^{\hat{I}_r}$. Let $\mathscr{C}_\kappa$ denote the full subcategory of $\widehat{\mathbf{U}}_\kappa\Lmod{}$ whose objects are smooth modules. Let  $\mathscr{C}_\kappa(r)$ denote the full subcategory of $\mathscr{C}_\kappa$ consisting of all modules $M$ generated by $M^{\hat{I}_r}$. 
\end{defi}

One can analogously define smooth $\mathbf{U}_\kappa(\hat{\g})$- and $\mathbf{U}_\kappa(\tilde{\g})$-modules. 
It is easy to see that the corresponding categories of smooth modules coincide with $\mathscr{C}_\kappa$. 
The following lemma, whose proof is standard, shows that the concept of smoothness defined above is analogous to that familiar from the representation theory of $p$-adic groups.

\begin{lem} \label{pro: smooth mod equivs}
Let $M$ be a $\widehat{\mathbf{U}}_\kappa$-module. The following are equivalent: 
\begin{enumerate}[label=\alph*), font=\textnormal,noitemsep,topsep=3pt,leftmargin=1cm]
\itemsep0em
\item $M$ is smooth, 
\item $M$, endowed with the discrete topology, is a topological $\widehat{\mathbf{U}}_\kappa$-module, 
\item $\Ann_{\widehat{\mathbf{U}}_\kappa}(v)$ is an open left ideal in $\widehat{\mathbf{U}}_\kappa$ for all $v \in M$. 
\end{enumerate}
\end{lem}

\section{Rational Cherednik algebras}

In this section we recall the definition and the main properties of rational Cherednik algebras of type A. Rational Cherednik algebras were introduced by Etingof and Ginzburg in \cite{EG}. In type A they can also be regarded as degenerations of the double affine Hecke algebras defined by Cherednik in \cite{Che}. 

\subsection{Rational Cherednik algebras.} \label{subsec:RCAtypeA}

\label{Subsec: conf sp 1}

Recall that $\h$ denotes the permutation representation of $\mathfrak{S}_m$. 
Let $\h_{\mathsf{reg}} \subset \h$ be the subvariety on which $\mathfrak{S}_m$ acts freely. 
We fix a basis $y_1, \hdots, y_m$ of $\h$ with the property that $w.y_i = y_{w(i)}$ for any $w \in \mathfrak{S}_m$ and $1 \leq i \leq m$. Let $x_1, \hdots, x_m$ be the dual basis of $\h^*$ so that $\C[\h] = \C[x_1,\hdots,x_m]$ and $\C[\h^*] = \C[y_1,\hdots,y_m]$. Define
\[ \C[\h]^{\rtimes}:=\C[\h] \rtimes \C \mathfrak{S}_m, \quad \C[\h^*]^{\rtimes}:= \C \mathfrak{S}_m \ltimes \C[\h^*].\] 
Set $\delta := \prod_{1 \leq i < j \leq m} (x_i - x_j)$ and $\delta_z = \prod_{j=1}^m(z-x_j)$. Define
\[ \mathcal{R}:=\C[\h_{\mathsf{reg}}] = \C[x_1,\hdots,x_m][\delta^{-1}], \quad \mathcal{R}^\rtimes := \mathcal{R} \rtimes \C \mathfrak{S}_m, \quad \mathcal{R}_z := \mathcal{R}[z][\delta_z^{-1}]\] 

\begin{defi}[{\cite[\S 4]{EG}}]
The \emph{rational Cherednik algebra} $\mathcal{H}_{t,c}$  associated to the complex reflection group $\mathfrak{S}_m$ and parameters $t,c \in \C$ is the quotient of the tensor algebra $T(\h \oplus \h^*) \rtimes \C \mathfrak{S}_m$ by the relations: 
\begin{itemize}
\item $[x_i,x_j]= [y_i,y_j] = 0$ \ $(1 \leq i,j \leq m)$,
\item $[x_i,y_j] = cs_{i,j}$ \ $(1 \leq i \neq j \leq m)$, 
\item $[x_i,y_i] = t - c\sum_{j \neq i} s_{i,j}$ \ $(1 \leq i \leq m)$.
\end{itemize}
Let $1_{\mathcal{H}}$ denote the unit in $\mathcal{H}_{t,c}$.
\end{defi}
It follows directly from the relations that if $\xi \in \C^*$ then $\mathcal{H}_{t,c}\cong\mathcal{H}_{\xi t,\xi c}$. From now on we will assume that $c=1$ and abbreviate $\mathcal{H}_t:=\mathcal{H}_{t,1}$.  
Setting $\deg x_i = \deg y_i = 1$ and $\deg s_i = 0$ defines a filtration on $\mathcal{H}_{t}$. Let $\gr\mathcal{H}_{t}$ be the associated graded algebra. 

\begin{thm}[{\cite[Theorem 1.3]{EG}}] 
The tautological embedding $(\h \oplus \h^*) \hookrightarrow \gr\mathcal{H}_{t}$ extends to a graded algebra isomorphism 
\begin{equation} \label{RCA PBW} \C[\h \oplus \h^*] \rtimes \C\mathfrak{S}_m \xrightarrow{\sim} \gr\mathcal{H}_{t} \end{equation} 
called the PBW isomorphism. 
\end{thm}

The following result shows that there is an analogy between the centres of $\mathcal{H}_{t}$ and~$\widehat{\mathbf{U}}_\kappa$.  

\begin{pro}[{\cite[Proposition 7.2]{GorBro}}]
$Z(\mathcal{H}_{t}) = \C$ if and only if $t \neq 0$.
\end{pro}

The next theorem summarizes the main properties of the centre of $\mathcal{H}_0$, which we abbreviate as
$$\mathcal{Z} :=Z(\mathcal{H}_{0}).$$ 

\begin{thm} \label{pro: centre facts}
The following hold. 
\begin{enumerate}[label=\alph*), font=\textnormal,noitemsep,topsep=3pt,leftmargin=1cm]
\itemsep0em
\item We have $\C[\h]^{\mathfrak{S}_m} \otimes \C[\h^*]^{\mathfrak{S}_m} \subset \mathcal{Z}$. The algebra $\mathcal{Z}$ is a free $\C[\h]^{\mathfrak{S}_m} \otimes \C[\h^*]^{\mathfrak{S}_m}$-module of rank $m!$. 
\item The PBW isomorphism restricts to an isomorphism $\C[\h \oplus \h^*]^{\mathfrak{S}_m} \xrightarrow{\sim} \gr\mathcal{Z}$. 
\item The affine variety $\Spec \mathcal{Z}$ is isomorphic to the Calogero-Moser space 
\[ \{ (X,Y,u,v) \in \Mat_{m \times m}(\C)^{\oplus 2} \times \C^m \times (\C^m)^* \mid [X,Y] + I_m = v \cdot u\} \sslash GL_m(\C).\]
\item We have $\mathcal{Z} \cong \mathbf{e}\mathcal{H}_0\mathbf{e}$. Moreover, the functor
\begin{equation} \label{functor e} \mathcal{H}_0\Lmod{} \to \mathbf{e}\mathcal{H}_0\mathbf{e}\Lmod{}, \quad M \mapsto \mathbf{e} \cdot M, \end{equation} 
where $\mathbf{e} = \frac{1}{m!} \sum_{w \in \mathfrak{S}_m} w$ is the trivial idempotent, is an equivalence of categories. 
\item Every simple $\mathcal{H}_0$-module has dimension $m!$ and is isomorphic to $\C \mathfrak{S}_m$ as an $\mathfrak{S}_m$-module. Moreover, there is a bijection
\[ \{\mbox{isoclasses of simple } \mathcal{H}_0 \mbox{-modules}\} \longleftrightarrow \MaxSpec \mathcal{Z}. \] 
\end{enumerate}
\end{thm}

\begin{proof}
Part a) is \cite[Proposition 4.15]{EG}, part b) is \cite[Theorem 3.3]{EG} and part c) is \cite[Theorem 11.16]{EG}. For part d) see the proof of  \cite[Proposition 3.8]{EG} and the remark following~it. Part e) is \cite[Theorem 1.7]{EG}. 
\end{proof} 

\subsection{Generalized Verma modules.}

Let us recall the definition of generalized Verma modules for $\mathcal{H}_t$. 
Let $l \geq 1$, $\nu \in \mathcal{C}_l(m)$, $\lambda \in \mathcal{P}_m(\nu)$ and $a \in \h^*$ with $\mathfrak{S}_m(a) = \mathfrak{S}_\nu$. Extend the $\C \mathfrak{S}_\nu$-module $\mathsf{Sp}(\lambda)$ to a $\C \mathfrak{S}_\nu \ltimes \C[\mathbf{\h^*}]$-module $\mathsf{Sp}(a,\lambda)$ by letting each $y_i$ act on $\mathsf{Sp}(\lambda)$ by the scalar $a_i:=a(y_i)$. 

\begin{defi}[{\cite[\S 1.3]{Bel}}] \label{def: gen Vermas} 
The \emph{generalized Verma module} of type $(a,\lambda)$ is
\[ \Delta_{t}(a,\lambda) := \mathcal{H}_{t} \otimes_{\C\mathfrak{S}_\nu \ltimes \C[\mathbf{\h^*}]} \mathsf{Sp}(a,\lambda).\]  
We abbreviate $\Delta_{t}(\lambda):=\Delta_{t}(0,\lambda)$. 
\end{defi}

\begin{rem}
When $t \neq 0$, 
the modules $\Delta_{t}(\lambda)$ play the role of standard modules in the category $\O(\mathcal{H}_{t})$ defined in \cite{GGOR}. Using the results of \cite{BT}, Bonnaf\'{e} and Rouquier \cite{BR} also defined a highest weight category for $\mathcal{H}_{0}$ with graded shifts of $\Delta_{0}(\lambda)$ as the standard modules. 
\end{rem} 

\begin{thm}[{\cite[Theorem 2]{Bel}}] \label{thm 2 bellamy} 
The following hold.
\begin{enumerate}[label=\alph*), font=\textnormal,noitemsep,topsep=3pt,leftmargin=1cm]
\itemsep0em
\item The canonical map $\mathcal{Z} \to \End_{\mathcal{H}_{0}}(\Delta_{0}(a,\lambda))$ is surjective. 
\item The ring $\End_{\mathcal{H}_{0}}(\Delta_{0}(a,\lambda))$ is isomorphic to a polynomial ring in $m$ variables.
\item The $\End_{\mathcal{H}_{0}}(\Delta_{0}(a,\lambda))$-module $\mathbf{e}\Delta(a,\lambda)$ is free   of rank one. 
\end{enumerate} 
\end{thm}

Theorem \ref{thm 2 bellamy} allows us to construct simple $\mathcal{H}_0$-modules as quotients of generalized Verma modules. 

\begin{lem} \label{lem: quotients of vermas} 
Let $L$ be a simple $\mathcal{H}_0$-module. Then there exist $l \geq 1$, $\nu \in \mathcal{C}_l(m)$, $\lambda \in \mathcal{P}_m(\nu)$ and $a \in \h^*$ with $\mathfrak{S}_m(a) = \mathfrak{S}_\nu$ such that
$L \cong \Delta_0(a,\lambda)/I\cdot \Delta_0(a,\lambda)$ for some maximal ideal $I \lhd \End_{\mathcal{H}_0}(\Delta_0(a,\lambda))$. 
\end{lem}

\begin{proof}
The commuting operators $y_1, \hdots, y_m$ have a simultaneous eigenvector $v \in L$. Let $a \in \h^*$ be the corresponding eigenvalue. Without loss of generality, we may assume that $\mathfrak{S}_m(a) = \mathfrak{S}_\nu$ for some $\nu \in \mathcal{C}_l(m)$. The subspace $\mathfrak{S}_\nu \cdot v \subset L$ is $\C[\h^*]$-stable and decomposes as a sum of simple $\mathfrak{S}_\nu$-modules. Suppose that this sum contains a simple module isomorphic to $\mathsf{Sp}(\lambda)$. Then there is a surjective homomorphism $\Delta_0(a,\lambda) \twoheadrightarrow L$. Let $K$ denote its kernel. 

We abbreviate $E(a,\lambda) := \End_{\mathcal{H}_0}(\Delta_0(a,\lambda))$. Since, by part a) of Theorem \ref{thm 2 bellamy}, \linebreak $\mathcal{Z}$ surjects onto $E(a,\lambda)$, all endomorphisms in $E(a,\lambda)$ preserve $\mathbf{e}K$. Hence $\mathbf{e}K$ is an $E(a,\lambda)$-submodule of $\mathbf{e}\Delta_0(a,\lambda)$. But, by part c) of Theorem \ref{thm 2 bellamy}, $\mathbf{e}\Delta_0(a,\lambda)$ is a free $E(a,\lambda)$-module of rank one. Hence $\mathbf{e}K = I \cdot \mathbf{e}\Delta_0(a,\lambda) = \mathbf{e} I \cdot \Delta_0(a,\lambda)$ for some ideal $I \lhd E(a,\lambda)$. 

By the definition of $K$ and part d) of Theorem \ref{pro: centre facts}, there is a short exact sequence $0 \to \mathbf{e} I \cdot \Delta_0(a,\lambda) \to \mathbf{e} \Delta_0(a,\lambda) \to \mathbf{e} L \to 0$. Since, by part e) of Theorem \ref{pro: centre facts}, $\mathbf{e}L \cong \C$, it follows that $I$ is a maximal ideal. The fact that \eqref{functor e} is an equivalence implies that the sequence $0 \to I \cdot \Delta_0(a,\lambda) \to \Delta_0(a,\lambda) \to L \to 0$ is exact as well. Hence $K = I \cdot \Delta_0(a,\lambda)$. 
\end{proof}

\subsection{Supports of Verma modules.} \label{sec: supp Verma modules}

By \cite[\S 1.1]{Bel2}, the support of the module $\Delta_{0}(a,\lambda)$ only depends on $\mathbf{a}:=\varpi(a)$, where $\varpi \colon \h^* \to \h^*/\mathfrak{S}_m$ is the canonical map. Therefore we can define 
\[ \Omega_{\mathbf{a},\lambda} := \supp_{\mathcal{Z}}(\Delta_{0}(a,\lambda)).\] 
We abbreviate $\Omega_\lambda := \Omega_{0,\lambda}$. 
Let 
\begin{equation} \label{definition of pi} \pi \colon \Spec \mathcal{Z} \to \h^*/\mathfrak{S}_m \end{equation} 
be the morphism of affine varieties induced by the inclusion $\C[\h^*]^{\mathfrak{S}_m} \hookrightarrow \mathcal{Z}$. 

\begin{pro} \label{pro pi fibre Omega cells}
We have 
\[ \pi^{-1}(\mathbf{a})_{\mathsf{red}} = \bigsqcup_{\lambda \in \mathcal{P}(\nu)} \Omega_{\mathbf{a},\lambda} \]
with $\Omega_{\mathbf{a},\lambda} \cong \Spec \End_{\mathcal{H}_{0}}(\Delta_{0}(a,\lambda)) \cong \mathbb{A}^m.$ 
\end{pro}

\begin{proof}
The first statement follows from \cite[Proposition 4.9]{Bel2} and the second statement from Theorem \ref{thm 2 bellamy}.b). 
\end{proof}

\section{Recollections on vertex algebras}
In this section we recall the definition of the vertex algebra associated to the vacuum module $\mathsf{Vac}_\kappa :=\Uu(\hat{\g}_\kappa)/\Uu(\hat{\g}_\kappa).\hat{\g}_+$. We also recall the main results about the centre of this vertex algebra and its connection to $\mathfrak{Z}$. 

\subsection{Vertex algebras.}

Let $R$ be an algebra and let $f(z) = \sum_{r \in \Z} f_{(-r-1)}z^{r}$ and $g(z) = \sum_{r \in \Z} g_{(-r-1)}z^{r}$ be formal power series in $R[[z,z^{-1}]]$. Their \emph{normally ordered product} :$f(z)g(z)$: is defined to be the formal power series 
\[ \mbox{:}f(z)g(z)\mbox{:} = f_+(z)g(z) + g(z)f_-(z), \quad f_+(z) = \sum_{r \geq 0} f_{(-r-1)} z^r, \quad  f_-(z) = \sum_{r < 0} f_{(-r-1)} z^r. \]
Given $f_1(z),\hdots,f_l(z) \in R[[z,z^{-1}]]$, set 
$$\mbox{:}f_1(z) \dotsm f_l(z)\mbox{:} = \mbox{:}f_1(z) \dotsm (\mbox{:}f_{l-2}(z)(\mbox{:}f_{l-1}(z)f_l(z)\mbox{:})\mbox{:})\mbox{:}$$

Let $W$ be a vector space. A series $f(z) = \sum_{r \in \Z} f_{(-r-1)}z^{r} \in (\End_{\C} W) [[z,z^{-1}]]$ is called a \emph{field} on $W$ if for every $v \in W$ there exists an integer $k\geq0$ such that $f_{(r)}.v=0$ for all $r \geq k$. Fields are preserved by the normally ordered product.

A \emph{vertex algebra} is a quadruple $(W,|0\rangle,\mathbb{Y}, T)$ consisting of a complex vector space $W$, 
a distinguished element $|0\rangle \in W$, called the vacuum vector, a linear map
\[ \mathbb{Y} \colon W \to (\End_{\C} W) [[z,z^{-1}]], \quad a \mapsto \mathbb{Y}(a,z) = \sum_{r \in \Z} a_{(-r-1)} z^{r}\]
sending vectors to fields on $W$,  
called the \emph{state-field correspondence}, and a linear map
$T \colon W \to W$ called the \emph{translation operator}. These data must satisfy a list of axioms, see, e.g., \cite[Definition 1.3.1]{BF}. 

Let us briefly recall the construction of a functor 
\[ \widetilde{U} \colon \{ \Z\mbox{-graded vertex algebras} \} \to \{ \mbox{complete topological associative algebras} \}. \] 
Given a $\Z$-graded vertex algebra $W$, one considers a completion of the Lie algebra of Fourier coefficients associated to $W$, and  takes its universal enveloping algebra. To obtain $\widetilde{U}(W)$, one again needs to form a completion and take a quotient by certain relations. The precise definition can be found in \cite[\S 4.3.1]{BF}.

\subsection{The affine vertex algebra.} 
Let $\kappa \in \C$. 
The vacuum module $\mathsf{Vac}_\kappa$ can be endowed with the structure of a vertex algebra, as in  \cite[\S 2.4]{BF}. Let us explicitly recall the state-field correspondence. 
Let $\rho : \Uu_\kappa(\hat{\g}) \to \End_{\C}(\mathsf{Vac}_\kappa)$ be the representation of $\hat{\g}_\kappa$ on $\mathsf{Vac}_\kappa$. 
The state-field correspondence $\mathbb{Y}$ is given by $\mathbb{Y}(|0\rangle,z) = \id$ and 
\begin{equation} X(z):=\mathbb{Y}(X[-1],z) = \sum_{r \in \Z} \rho(X[r])z^{-r-1}, \end{equation} 
\begin{equation} \label{statefieldcorrespondence} \mathbb{Y}(X_1[k_1]\hdots X_l[k_l],z) = \frac{1}{(-k_1-1)!} \hdots \frac{1}{(-k_l-1)!}\mbox{:}\partial_z^{-k_1-1}X_1(z) \hdots \partial_z^{-k_l-1}X_l(z) \mbox{:}\end{equation} 
for $X,X_1, \hdots, X_l \in \g$ and $k_1, \hdots, k_l \leq -1$. 
Given $X \in \g$ we also define a power series 
\[ X\langle z \rangle := \mathbb{Y}\langle X[-1],z\rangle := \sum_{r\in \Z} X[r] z^{-r-1}.\]
Applying formula \eqref{statefieldcorrespondence} with each $X_i(z)$ replaced by $X_i \langle z \rangle$ 
we can associate a power series $\mathbb{Y}\langle A,z\rangle = \sum_{r \in \Z} A_{\langle-r-1\rangle} z^r\in \widehat{\Uu}_\kappa[[z,z^{-1}]]$ to an arbitrary element $A \in \mathsf{Vac}_\kappa$.

\subsection{The Feigin-Frenkel centre.} \label{explicit set}

Let $Z(\mathsf{Vac}_\kappa)$ denote the centre of the vertex algebra $\mathsf{Vac}_\kappa$. It is a commutative vertex algebra, which is also a commutative ring. A precise definition can be found in \cite[\S 3.3.1]{Fre}. 

\begin{pro}[{\cite[Proposition 3.3.3]{Fre}}] 
$Z(\mathsf{Vac}_\kappa) = \C |0\rangle$ if and only if $\kappa \neq \mathbf{c}$. 
\end{pro} 

The commutative vertex algebra $\mathfrak{z}(\hat{\g}):=Z(\mathsf{Vac}_{\mathbf{c}})$ is known as the \emph{Feigin-Frenkel centre}. 
Elements of $\mathfrak{z}(\hat{\g})$ are called \emph{Segal-Sugawara vectors}. 
We are now going to recall an explicit description of $\mathfrak{z}(\hat{\g})$ due to Chervov and Molev. 
Identify $\Uu(\hat{\g}_-) \xrightarrow{\sim} \mathsf{Vac}_{\mathbf{c}}, \ X \mapsto X\cdot|0\rangle$ as vector spaces and 
consider the maps
\[
S(\g) \stackrel{i}{\hookrightarrow} S(\hat{\g}_-) \stackrel{\sigma}{\leftarrow} \Uu(\hat{\g}_-),
\]
where $i(X) = X[-1]$ for $X \in \g$ and $\sigma$ is the principal symbol map with respect to the PBW filtration. 

\begin{defi}[{\cite[\S 2.2]{CM}}]
One calls $A_1, \hdots, A_n \in \mathfrak{z}(\hat{\g}) \subset \Uu(\hat{\g}_-)$ a \emph{complete set of Segal-Sugawara vectors} if there exist algebraically independent generators $B_1, \hdots, B_n$ of the algebra $S(\g)^{\g}$ such that $i(B_1) = \sigma(A_1), \hdots, i(B_n) = \sigma(A_n)$. 
\end{defi}

\begin{thm}[{\cite[Theorem 9.6]{Fre2}}]  \label{FF small theorem thm}
If $A_1, \hdots, A_n$ are a complete set of Segal-Sugawara vectors then 
\begin{equation} \label{FF small theorem} \mathfrak{z}(\hat{\g}) = \C[T^{k}A_r \mid r=1,\hdots,n, \ k \geq 0], \end{equation}
where $T$ is the translation operator. 
\end{thm}

\begin{exa} \label{exa: complete set T} 
Let $\doublehat{\g}_\kappa$ be the extension 
$0 \to \hat{\g}_\kappa \to \doublehat{\g}_\kappa \to \C \tau \to 0$ 
defined by the relations
$[\tau, X\otimes f] = - X \otimes \partial_t f$ and $[\tau,\mathbf{1}] = [\tau,\tau]= 0$. 
The subspace $\doublehat{\g}_{-} := \hat{\g}_- \oplus \C \tau$ is a Lie subalgebra of $\doublehat{\g}_\kappa$. 
Consider the matrix $E_\tau \in \Mat_{n \times n}(\Uu(\doublehat{\g}_-))$ defined as
\[ E_\tau := 
\left( \begin{array}{cccc}
\tau + e_{11}[-1] & e_{12}[-1] & \cdots & e_{1n}[-1] \\
e_{21}[-1] & \tau + e_{22}[-1] & \cdots & e_{2n}[-1] \\
\vdots & \vdots & \ddots & \vdots \\
e_{n1}[-1] & e_{n2}[-1] & \cdots & \tau + e_{nn}[-1]
 \end{array} \right).\] 
The traces $\Tr (E_\tau^k)$ are elements of $\Uu(\doublehat{\g}_-)$. In light of the canonical vector space isomorphism $\Uu(\doublehat{\g}_-) \cong \Uu(\hat{\g}_-) \otimes \C[\tau]$, we can regard $\Tr (E_\tau^k)$ as polynomials in $\tau$ with coefficients in $\Uu(\hat{\g}_-) \cong \mathsf{Vac}_{\mathbf{c}}$. 
Define $\mathbf{T}_{k;l}$ $(0 \leq l \leq k \leq n)$ to be the coefficients of the polynomial
\[ \Tr(E^k_\tau ) = \mathbf{T}_{k;0} \tau^k + \mathbf{T}_{k;1} \tau^{k-1} + \hdots + \mathbf{T}_{k;k-1} \tau + \mathbf{T}_{k;k} \]
and set $\mathbf{T}_k := \mathbf{T}_{k;k}$. 
By \cite[Theorem 3.1]{CM}, the set $\{ \mathbf{T}_{k} \mid 1 \leq k \leq n\}$ is a complete set of Segal-Sugawara vectors in $\mathfrak{z}(\hat{\g})$. 
\end{exa}

\subsection{The centre of the enveloping algebra.} \label{sec:va-centre of env} 
If $A$ is a Segal-Sugawara vector, 
the coefficients $A_{\langle r \rangle}$ of the power series $\mathbb{Y} \langle A,z \rangle$ are called \emph{Segal-Sugawara operators}. 
Given a complete set of Segal-Sugawara vectors $A_1, \hdots ,A_n$ such that $\deg A_i =~-i$, let  
\begin{equation} \label{centre A} \mathscr{Z}:=\C[A_{i,\langle l \rangle}]_{i=1,...,n}^{l \in \Z}. \end{equation}
be the free polynomial algebra generated by the corresponding Segal-Sugawara operators. 
For $k>0$, let $J_k$ be the ideal in $\mathscr{Z}$ generated by the $A_{i,\langle l \rangle}$ with $l \geq ik$. 
\begin{thm}
There exist natural algebra isomorphisms 
\begin{equation} \label{small and big centres} \widetilde{U}(\mathsf{Vac}_{\mathbf{c}}) \cong \widehat{\Uu}_{\mathbf{c}}, \quad \widetilde{U}(\mathfrak{z}(\hat{\g})) \cong \mathfrak{Z}.\end{equation} 
Moreover, $\mathfrak{Z} = \varprojlim \left( \mathscr{Z} / J_k\right).$ 
\end{thm} 

\begin{proof}
For the isomorphisms \eqref{small and big centres}, see \cite[Lemma 3.2.2, Proposition 4.3.4]{Fre}. For the second statement, see \cite[\S 4.3.2]{Fre} or \cite[\S 12.2]{Fre2}. 
\end{proof}

\subsection{Quadratic Segal-Sugawara operators.} \label{subsec:quadraticSSvector} 
Let $\kappa \in \C$. An important role is played by the vector 
\begin{equation} \label{kappa-L vector} {}^\kappa\mathbf{L} = \frac{1}{2} \sum_{1 \leq k,l \leq n} e_{kl}[-1]e_{lk}[-1] \in \mathsf{Vac}_\kappa. \end{equation} 
Writing 
$\mathbb{Y} \langle {}^{\kappa}\mathbf{L},z\rangle = \sum_{r \in \Z} {}^{\kappa}\mathbf{L}_{\langle r \rangle} z^{-r-1}$, we have the formula 
\begin{equation} \label{kappa-L coeff} {}^{\kappa}\mathbf{L}_r :={}^{\kappa}\mathbf{L}_{\langle r+1 \rangle} = \frac{1}{2} \sum_{1 \leq k,l \leq n} \left( \sum_{i \leq - 1} e_{kl}[i]e_{lk}[r-i] + \sum_{i \geq 0} e_{lk}[r-i]e_{kl}[i] \right) \in \widehat{\Uu}_{\kappa}(\hat{\g}). \end{equation}  
\begin{pro}
If $\kappa = \mathbf{c}$ then ${}^{\mathbf{c}}\mathbf{L} \in \mathfrak{z}(\hat{\g})$ and ${}^{\mathbf{c}}\mathbf{L}_{r} \in \mathfrak{Z}$ for each $r \in \Z$. 
If $\kappa \neq \mathbf{c}$, then
$ [\frac{1}{\kappa+n}{}^\kappa\mathbf{L}_{-1},X \otimes f ] = -X \otimes \partial_t f$ for all $X\in \g$ and $f \in \C((t)).$
\end{pro} 

\begin{proof}
The proposition follows from a direct calculation using operator product expansions. This calculation can be found in, e.g., \cite[\S 3.1.1]{Fre}. 
\end{proof}

\section{Suzuki functor for all levels} \label{sec: conformal coinvariants} 

In \cite{Suz}, Suzuki defined a functor $\mathsf{F}_\kappa \colon \mathscr{C}_\kappa \to \mathcal{H}_{\kappa+n}\Lmod{}$ for $\kappa \neq \mathbf{c}$. In this section we generalize his construction to the $\kappa = \mathbf{c}$ case. 
Throughout this section assume that $m,n$ are any positive integers and $\kappa \in \C$ unless stated otherwise.

\subsection{Simultaneous affinization.} 
Let $\mathbb{V}_\kappa^* := \Ind^{\hat{\g}_\kappa}_{\hat{\g}_+} \circ \Inf_{\g \oplus \C \mathbf{1}}^{\hat{\g}_+} \mathbf{V}^*$, where $\mathbf{1}$ acts on $\mathbf{V}^*$ as the identity endomorphism. 
We start by recalling (see e.g.\ \cite[\S 9.9, 9.11]{KLII}) the construction of a $\g \otimes \mathcal{R}_z$-action on 
\begin{equation} \label{Tkappa tensor} \mathbb{T}_{\kappa}(M):= \mathcal{R} \otimes (\mathbb{V}_\kappa^*)^{\otimes m} \otimes M,  \end{equation} 
for any module $M$ in $\mathscr{C}_\kappa$. For that purpose we first recall the definition of an auxiliary Lie algebra $\mathfrak{G}_R$. 

Let $R$ be a commutative unital algebra. 
We fix formal variables $t_1, \hdots, t_m, t_\infty$. Set $\g(i)_R := \g \otimes R((t_i))$, $\g(i):=\g(i)_{\C}$. 
Consider the $R$-Lie algebra 
\begin{equation} \label{GR tensor} \mathfrak{G}_R:=\bigoplus_{i=1}^m \g(i)_R \oplus \g(\infty)_R = \g \otimes (\bigoplus_{i=1}^m R((t_i)) \oplus R((t_\infty))).\end{equation}
We denote a pure tensor on the RHS of \eqref{GR tensor} by $X \otimes (f_i)$, where $X \in \g$ and $f_i \in R((t_i))$ for $i=1,\hdots,m,\infty$.  
Define $\hat{\mathfrak{G}}_{R,\kappa}$ to be the central extension
\begin{equation} \label{multiloop} 0 \to R \mathbf{1} \to \hat{\mathfrak{G}}_{R,\kappa} \to \mathfrak{G}_R \to 0\end{equation}
associated to the cocycle $(X\otimes (f_i), Y\otimes (g_i)) \mapsto  \langle X, Y \rangle_\kappa \sum_{i \in \{1,...,m,\infty\}}\Res_{t_i=0}(g_idf_i)$. 
Set $$\mathbf{U}_\kappa(\hat{\mathfrak{G}}_{R}) := \mathbf{U}(\hat{\mathfrak{G}}_{R,\kappa})/\langle\mathbf{1} - 1_{\hat{\mathfrak{G}}_{R,\kappa}}\rangle.$$ 
If $R=\C$, we abbreviate $\hat{\mathfrak{G}}_\kappa := \hat{\mathfrak{G}}_{\C,\kappa}$ and $\mathbf{U}_\kappa(\hat{\mathfrak{G}}) = \mathbf{U}_\kappa(\hat{\mathfrak{G}}_{\C})$. 

A $\mathbf{U}_\kappa(\hat{\mathfrak{G}}_{R})$-module $M$ is called \emph{smooth} if for every vector $v \in M$ there exists a positive integer $k$ such that $\g \otimes (\bigoplus_{i=1}^m t_i^kR((t_i)) \oplus t_\infty^kR((t_\infty))).v=0$. 
Suppose that $M_1,\hdots,M_m,M_\infty$ are smooth $\mathbf{U}_\kappa(\hat{\g})$-modules. 
Then $R \otimes\bigotimes_{i=1}^mM_i \otimes M_\infty$ is a smooth $\mathbf{U}_\kappa(\hat{\mathfrak{G}}_{R})$-module with the action of the dense subalgebra $R \otimes \mathbf{U}_\kappa(\hat{\mathfrak{G}})$ given by the formula
\begin{equation} \label{multiG-action} r \otimes X \otimes (f_i) \mapsto \sum_{i=1,...,m,\infty}r \otimes (X \otimes f_i)^{(i)},\end{equation} 
where $(X \otimes f_i)^{(i)} := \id^{i-1} \otimes (X \otimes f_i) \otimes \id^{m-i}$. 
Note that if $R$ were an infinite-dimensional algebra and the modules $M_i$ were not smooth, the action of $R \otimes \mathbf{U}_\kappa(\hat{\mathfrak{G}})$ would not necessarily extend to an action of  $\mathbf{U}_\kappa(\hat{\mathfrak{G}}_{R})$.

\subsection{Conformal coinvariants.} \label{subsec: conf coinv global version} 
We next recall the connection between the Lie algebras $\mathfrak{G}_R$ and $\g \otimes \mathcal{R}_z$. 
Consider $\mathcal{R}_z$ as an $\mathcal{R}$-subalgebra of $\mathcal{R}(z)$. We thus view elements of $\mathcal{R}_z$ as rational functions which may have poles at $x_1,\hdots,x_m$ and $\infty$. Set $z_i := z - x_i$. 
\begin{defi} 
For $1 \leq i \leq m$, 
let $\iota_{\mathcal{R},i} \colon \mathcal{R}_z \to \mathcal{R}((z_i))$ (resp.\ $\iota_{\mathcal{R},\infty} \colon \mathcal{R}_z \to \mathcal{R}((z^{-1}))$) be the $\mathcal{R}$-algebra homomorphism sending a function in $\mathcal{R}_z$ to its Laurent series expansion at $x_i$ (resp.\ $\infty$). Let
\begin{equation} \label{R-iota-x} \iota_{\mathcal{R}} : \mathcal{R}_z \hookrightarrow \bigoplus_{i=1}^m \mathcal{R}((t_i)) \oplus \mathcal{R}((t_\infty))\end{equation}
be the injective $\mathcal{R}$-algebra homomorphism 
given by $(\iota_{\mathcal{R},1},...,\iota_{\mathcal{R},m},\iota_{\mathcal{R},\infty})$ followed by the assignment $z_i \mapsto t_i, z^{-1} \mapsto t_\infty$. 
\end{defi}

The map \eqref{R-iota-x} induces the Lie algebra homomorphism 
\begin{equation} \label{R-g-iota} \g \otimes \mathcal{R}_z \hookrightarrow \mathfrak{G}_{\mathcal{R}}, \quad X \otimes f \mapsto X \otimes \iota_{\mathcal{R}}(f),\end{equation}
which, by the residue theorem, lifts to an injective Lie algebra homomorphism
\begin{equation} \label{R-g-iota2} \g \otimes \mathcal{R}_z \hookrightarrow \hat{\mathfrak{G}}_{\mathcal{R},\kappa}.\end{equation}

Let $M$ be a smooth $\mathbf{U}_\kappa(\hat{\g})$-module. 
The vector space $\mathbb{T}_{\kappa}(M)$ is a smooth $\mathbf{U}_\kappa(\hat{\mathfrak{G}}_{\mathcal{R}})$-module (with the action given by \eqref{multiG-action}). 
We consider it as a $\mathbf{U}(\g \otimes \mathcal{R}_z)$-module via \eqref{R-g-iota2}. It also carries a natural $\mathcal{R}^\rtimes$-action: $\mathcal{R}$ acts by multiplication and $\mathfrak{S}_m$ acts by permuting the factors of the tensor product $(\mathbb{V}_\kappa^*)^{\otimes m}$ and the $x_i$'s. 
The next lemma follows directly from the definitions.  

\begin{lem} The $\mathcal{R}^\rtimes$-action on $\mathbb{T}_{\kappa}(M)$ normalizes the $\mathbf{U}(\g\otimes \mathcal{R}_z)$-action. Therefore we have functors
\begin{alignat}{7} \label{Tk functor}
\mathbb{T}_\kappa \colon& \mathscr{C}_\kappa \to (\mathbf{U}(\g\otimes \mathcal{R}_z),\mathcal{R}^\rtimes)\Nmod{},& \quad& M \mapsto \mathbb{T}_\kappa(M),\\ \label{Fk functor}
\mathbb{F}_\kappa \colon& \mathscr{C}_\kappa \to \mathcal{R}^\rtimes\Lmod{},& \quad&  M \mapsto H_0(\g \otimes \mathcal{R}_z, \mathbb{T}_{\kappa}(M)). 
\end{alignat} 
\end{lem}

\subsection{The Knizhnik-Zamolodchikov connection.} \label{subsec:KZ connection}

We are going to extend the $\mathcal{R}^\rtimes$-action on $\mathbb{T}_\kappa(M)$ and $\mathbb{F}_\kappa(M)$ to an action of $\mathcal{H}_{\kappa+n}$. 

\begin{defi} 
Let $\kappa \in \C$. The \emph{deformed Weyl algebra} $\mathcal{D}_{\kappa}$ is the algebra generated by $x_1, \hdots, x_m$ and $q_1, \hdots, q_m$ subject to the relations
\[ [x_i,x_j] = [q_i,q_j] =0, \quad [x_i,q_j]=(\kappa + n)\delta_{ij} \quad (1 \leq i,j \leq m).\]
Note that $\mathcal{D}_{\mathbf{c}} = \C[x_1,\hdots,x_m,q_1,\hdots,q_m]$. Set \[\mathcal{D}_{\kappa}^\rtimes := \mathcal{D}_{\kappa}\rtimes \C \mathfrak{S}_m, \quad \mathcal{D}^{\rtimes}_{\kappa,\mathsf{reg}} := \mathcal{D}^\rtimes_\kappa[\delta^{-1}].\] 
Suppose that $M$ is a $\C[\h]^\rtimes$- (resp.\ $\mathcal{R}^\rtimes$-) module. 
A 
\emph{good connection} on $M$ is a representation of $\mathcal{D}^\rtimes_{\kappa}$ (resp.\ $\mathcal{D}^{\rtimes}_{\kappa,\mathsf{reg}}$) on $M$ extending the given $\C[\h]^\rtimes$- (resp.\ $\mathcal{R}^\rtimes$-) module structure. 
\end{defi}

\begin{lem} \label{lem: mod connection}
Let $M$ be a $\C[\h]^\rtimes$-module. If $\rho \colon \mathcal{D}_{\kappa}^\rtimes \to \End_{\C}(M)$ is a good connection on $M$, then $\rho'$, defined as 
\[ \rho'(q_i) := \rho(q_i) + \sum_{j \neq i} \frac{1}{x_i - x_j}, \] 
is a good connection on the $\mathcal{R}^\rtimes$-module $M_{\mathsf{reg}} := \mathcal{R} \otimes_{\C[\h]} M$. 
\end{lem}

\begin{proof}
The lemma follows by a direct calculation, as in \cite[Proposition 1.8]{VV}.  
\end{proof}

Let $M$ be a smooth $\mathbf{U}_\kappa(\hat{\g})$-module. Consider the $\mathcal{R}^\rtimes$-module $\mathbb{T}_{\kappa}(M)$ and the operators
\[ {}^\kappa\Grad_i := -(\kappa + n)\partial_{x_i} + {}^\kappa\mathbf{L}_{-1}^{(i)} \quad (1 \leq i \leq m) \]
on $\mathbb{T}_{\kappa}(M)$. The following proposition extends \cite[Lemma 13.3.7]{BF} to the critical level case.

\begin{pro} \label{pro: nabla descends} 
Let $\kappa \in \C$. 
\begin{enumerate}[label=\alph*), font=\textnormal,noitemsep,topsep=3pt,leftmargin=1cm]
\itemsep0em
\item The assignment
\[ {}^\kappa\Grad \colon \mathcal{D}^\rtimes_{\kappa, \mathsf{reg}} \to \End_{\C}(\mathbb{T}_{\kappa}(M)), \quad q_i \mapsto {}^\kappa\Grad_i \]
defines a good connection (known as the Knizhnik-Zamolodchikov connection) on $\mathbb{T}_{\kappa}(M)$. 
\item The operators ${}^\kappa\Grad_i$ normalize the $\g \otimes \mathcal{R}_z$-action on $\mathbb{T}_{\kappa}(M)$, i.e., $[{}^\kappa\Grad_i,\g \otimes \mathcal{R}_z] \subset \g \otimes \mathcal{R}_z$. Hence ${}^\kappa\Grad$ descends to a good connection on $\mathbb{F}_{\kappa}(M)$. 
\end{enumerate} 
\end{pro}

\begin{proof}
It suffices to consider the case $\kappa = \mathbf{c}$. The operators ${}^{\mathbf{c}}\Grad_i = {}^{\mathbf{c}}\mathbf{L}_{-1}^{(i)}$ act on different factors $\mathbb{V}_{\mathbf{c}}^*$ of the tensor product $\mathbb{T}_{\mathbf{c}}(M) = \mathcal{R} \otimes (\mathbb{V}_{\mathbf{c}}^*)^{\otimes m} \otimes M$, so they commute. Moreover, the operators $x_j$ act only on the first factor $\mathcal{R}$ and so they commute with the operators ${}^{\mathbf{c}}\Grad_i$ as well. Hence ${}^{\mathbf{c}}\Grad$ is a representation of $\mathcal{D}_{\mathbf{c}}$, which clearly extends to a representation of $\mathcal{D}^\rtimes_{\mathbf{c},\mathsf{reg}}$. 
The second statement follows immediately from the fact that ${}^{\mathbf{c}}\mathbf{L}_{-1} \in \mathfrak{Z}$. 
\end{proof}

To obtain representations of the rational Cherednik algebra on $\mathbb{T}_\kappa(M)$ and $\mathbb{F}_\kappa(M)$, we are going to compose the connection ${}^\kappa\Grad'$ with the Dunkl embedding, whose definition we now recall.

\begin{pro}[{\cite[Proposition 4.5]{EG}}]
There is an injective algebra homomorphism, called the \emph{Dunkl embedding}, 
\begin{equation} \label{Dunkl embedding} \mathcal{H}_{\kappa + n} \hookrightarrow \mathcal{D}^{\rtimes}_{\kappa,\mathsf{reg}}, \quad x_i \mapsto x_i, \ w \mapsto w, \ y_i \mapsto D_i := q_i + \sum_{j \neq i} \frac{1}{x_i - x_j}(s_{i,j} - 1),\end{equation}
with $1 \leq i \leq m$ and $w \in \mathfrak{S}_m$. 
\end{pro}

\begin{pro}
Composing \eqref{Dunkl embedding} with ${}^\kappa\Grad'$ yields representations of $\mathcal{H}_{\kappa+n}$ on $\mathbb{T}_\kappa(M)$ and $\mathbb{F}_\kappa(M)$. Moreover, the functors \eqref{Tk functor} and \eqref{Fk functor} extend to functors
\[ \mathbb{T}_\kappa \colon \mathscr{C}_\kappa \to (\mathbf{U}(\g\otimes\mathcal{R}_z),\mathcal{H}_{\kappa+n})\Nmod{}, \quad \mathbb{F}_\kappa \colon \mathscr{C}_\kappa \to \mathcal{H}_{\kappa+n}\Lmod{}.\]
\end{pro}

\begin{proof} 
By Lemma \ref{lem: mod connection} and Proposition \ref{pro: nabla descends}, ${}^\kappa\Grad'$ is a good connection on $\mathbb{T}_\kappa(M)$, which descends to a good connection on $\mathbb{F}_\kappa(M)$. It therefore yields representations of $\mathcal{D}^{\rtimes}_{\kappa,\mathsf{reg}}$ on $\mathbb{T}_\kappa(M)$ and $\mathbb{F}_\kappa(M)$, which become representations of $\mathcal{H}_{\kappa+n}$ via the Dunkl embedding. 

Let us check that $\mathbb{T}_\kappa$ and $\mathbb{F}_\kappa$ are functors. 
Let $f \colon M \to N$ be a morphism in $\mathscr{C}_\kappa$. It induces a map $\mathbb{T}_\kappa(f) \colon \mathbb{T}_{\kappa}(M) \to \mathbb{T}_{\kappa}(N)$. Since the $\mathcal{H}_{\kappa+n}$-action doesn't affect the last factor (as in \eqref{Tkappa tensor}) in these tensor products, $\mathbb{T}_\kappa(f)$ commutes with the $\mathcal{H}_{\kappa+n}$-action. The fact that $f$ is a $\hat{\g}_\kappa$-module homomorphism also implies that $\mathbb{T}_\kappa(f)$ commutes with the $\g \otimes \mathcal{R}_z$-action on $\mathbb{T}_{\kappa}(M)$ and $\mathbb{T}_{\kappa}(N)$. Hence $\mathbb{T}_\kappa(f)$ descends to a $\mathcal{H}_{\kappa+n}$-module homomorphism $\mathbb{F}_\kappa(f) \colon \mathbb{F}_{\kappa}(M) \to \mathbb{F}_{\kappa}(N)$. 
\end{proof}

\subsection{The current Lie algebra action.} \label{subsec: current Lie}
Given a smooth $\mathbf{U}_\kappa(\hat{\g})$-module $M$, set 
\[ \mathsf{T}_\kappa(M) := \C[\h] \otimes (\mathbf{V}^*)^{\otimes m} \otimes M, \quad \mathsf{T}_\kappa^{\mathsf{loc}}(M) := \mathcal{R} \otimes (\mathbf{V}^*)^{\otimes m} \otimes M. \] 
We will show that the functors $\mathsf{T}_\kappa$ and $\mathsf{T}_\kappa^{\mathsf{loc}}$ fit into the following commutative diagram
\[
\begin{tikzcd}[row sep=large, column sep = huge]
& (\mathbf{U}(\g[t]),\mathcal{H}_{\kappa+n})\Nmod{} \arrow[r,"{H_0(\g[t], - )}"] \arrow[d,"{\mathsf{loc}}"] & \mathcal{H}_{\kappa+n}\Lmod{} \arrow[d,"{\mathsf{loc}}"] \\
\mathscr{C}_\kappa \arrow[ru,"{\mathsf{T}_\kappa}",bend left = 15] \arrow[rd,"{\mathbb{T}_\kappa}",swap,bend right = 15] \arrow[r,"{\mathsf{T}_\kappa^{\mathsf{loc}}}"] & (\mathbf{U}(\g[t]),\mathcal{H}_{\kappa+n})\Nmod{} \arrow[r,"{H_0(\g[t], - )}"] & \Lmod{\mathcal{H}_{\kappa+n}} \\
 & (\mathbf{U}(\g\otimes\mathcal{R}_z),\mathcal{H}_{\kappa+n})\Nmod{} \arrow[ru,"{H_0(\g \otimes \mathcal{R}_z, - )}",swap,bend right = 15]
\end{tikzcd}
\]
where $\mathsf{loc}$ is the localization functor sending $N$ to $N_{\mathsf{reg}} := \mathcal{R} \otimes_{\C[\h]} N$. The Suzuki functor is the composition of $\mathsf{T}_{\kappa}$ with $H_0(\g[t], - )$. Let us explain this diagram in more detail. 
The current Lie algebra $\g[t]$ acts on $\mathsf{T}_\kappa^{\mathsf{loc}}(M)$ by the rule 
\begin{equation} \label{currentaction} Y[k] \mapsto \sum_{i=1}^m x_i^k \otimes Y^{(i)} + 1 \otimes (Y[-k])^{(\infty)} \quad (Y \in \g, \ k \geq 0). \end{equation} 
The $\mathcal{R}^\rtimes$-action on $\mathsf{T}_\kappa^{\mathsf{loc}}(M)$ is analogous to that on $\mathbb{T}_{\kappa}(M)$. 
It follows directly from the definitions that the $\g[t]$-action and the $\mathcal{R}^\rtimes$-action commute. We next recall how the $\mathcal{R}^\rtimes$-action can be extended to an $\mathcal{H}_{\kappa+n}$-action on $\mathsf{T}_\kappa^{\mathsf{loc}}(M)$.

\begin{defi}
Let $1 \leq i, j \leq m$ and $p \geq 0$. Consider  
\begin{alignat*}{7} \Omega^{(i,j)} &:= \sum_{1 \leq k,l \leq n} e_{kl}^{(i)} e_{lk}^{(j)},& \quad \Omega^{(i,\infty)}_{[p+1]} &:= \sum_{1 \leq k,l \leq n} e_{kl}^{(i)}e_{lk}[p+1]^{(\infty)},\\ 
\mathfrak{L}^{(i)} &:= -\sum_{1 \leq j \neq i \leq m} \frac{\Omega^{(i,j)}}{x_i - x_j} + \sum_{p \geq 0} x_i^p \Omega^{(i,\infty)}_{[p+1]},& \quad {}^\kappa\nabla_i &:=  -(\kappa+n)\partial_{x_i} +  \mathfrak{L}^{(i)}. \end{alignat*} 
as operators on $\mathsf{T}_\kappa^{\mathsf{loc}}(M)$. 
They are well-defined because $M$ is smooth. 
\end{defi}

\begin{lem} \label{lem: y-ops formula} 
The assignment 
\[ {}^\kappa\nabla \colon \mathcal{D}^\rtimes_{\kappa} \to \End_{\C}(\mathsf{T}_\kappa^{\mathsf{loc}}(M)), \quad q_i \mapsto {}^\kappa\nabla_i \]
defines a good connection on $\mathsf{T}_\kappa^{\mathsf{loc}}(M)$. 
\end{lem}

\begin{proof}
One needs to check that ${}^\kappa\nabla$ is a well-defined ring homomorphisms, i.e., show that $[{}^\kappa\nabla_i,{}^\kappa\nabla_j] = 0$ and $[{}^\kappa\nabla_i,x_j]=-(\kappa+n)\delta_{ij}$. These commutation relations are calculated in \cite[Lemma 3.2-3.3]{KN}. 
\end{proof}

\begin{pro} \label{pro:y preserves subspace}
Composing \eqref{Dunkl embedding} with ${}^\kappa\nabla'$ yields a representation of $\mathcal{H}_{\kappa+n}$ on $\mathsf{T}_\kappa^{\mathsf{loc}}(M)$. The element $y_i$ acts as the operator 
\begin{equation} \label{y-formula} {}^\kappa\bar{y}_i= -(\kappa+n)\partial_{x_i} + \sum_{1 \leq j \neq i \leq m} \frac{\Omega^{(i,j)}}{x_i - x_j}(\underline{s_{i,j}}-1) + \sum_{p \geq 0} x_i^p \Omega^{(i,\infty)}_{[p+1]},
\end{equation}
where $\underline{s_{i,j}}$ acts by permuting the $x_i$'s but not the factors of the tensor product. Moreover, $\mathsf{T}_\kappa(M)$ is a subrepresentation.  
\end{pro} 

\begin{proof} 
By Lemma \ref{lem: mod connection} and Lemma \ref{lem: y-ops formula}, ${}^\kappa\nabla'$ is a good connection, which implies the first statement. For the second part, observe that
\[{}^\kappa\nabla'(D_i) = {}^\kappa\nabla_i + \sum_{j \neq i} \frac{1}{x_i - x_j}s_{i,j} = -(\kappa+n)\partial_{x_i} + \sum_{j \neq i} \frac{1}{x_i - x_j}(s_{i,j}-\Omega^{(i,j)}) + \sum_{p \geq 0} x_i^p \Omega^{(i,\infty)}_{[p+1]}.\]
The equality of operators $s_{i,j} = \Omega^{(i,j)}\underline{s_{i,j}}$ implies \eqref{y-formula}. The third statement follows from the fact that the operators $\frac{-1+\underline{s_{i,j}}}{x_i - x_j}$ and $\partial_{x_i}$ preserve $\C[\h] \subset \mathcal{R}$. 
\end{proof}

\subsection{Suzuki functor.} 
\label{subsec: coinv iso} 
We next consider the relationship between the functors $\mathsf{T}_\kappa^{\mathsf{loc}}$ and~$\mathbb{T}_\kappa$. Both $\mathsf{T}_\kappa^{\mathsf{loc}}(M)$ and $\mathbb{T}_\kappa(M)$ carry representations of $\mathcal{D}^\rtimes_\kappa$ given by ${}^\kappa\nabla'$ and ${}^\kappa\Grad'$, respectively. 
The following result is well known (see, e.g., \cite[Proposition 2.18]{VV}).

\begin{pro} \label{pro: coinv iso} 
Let $\kappa \in \C$. 
\begin{enumerate}[label=\alph*), font=\textnormal,noitemsep,topsep=3pt,leftmargin=1cm]
\itemsep0em
\item 
The connection ${}^\kappa\nabla'$ normalizes the $\g[t]$-action on $\mathsf{T}_\kappa^{\mathsf{loc}}(M)$ and descends to a good connection on $H_0(\g[t], \mathsf{T}_\kappa^{\mathsf{loc}}(M))$. 
\item There is a $\mathcal{D}^\rtimes_{\kappa,\mathsf{reg}}$-module isomorphism 
\begin{equation} \label{global-current coinv iso} H_0(\g[t], \mathsf{T}_\kappa^{\mathsf{loc}}(M)) \cong H_0(\g \otimes \mathcal{R}_z, \mathbb{T}_{\kappa}(M)),\end{equation} 
intertwining the connections ${}^\kappa\nabla'$ and ${}^\kappa\Grad'$. 
\end{enumerate} 
\end{pro} 

\begin{proof}
Let us prove part a). Set $t=\kappa+n$, and let $X \in \g$ and $r \geq 0$. We need to compute the commutator $[{}^\kappa\nabla'(q_i), a(X[r])]$. We have 
\[ [{}^\kappa\nabla'(q_i), a(X[r])] = A + \sum_{j \neq i} B_j + C \]
as linear operators on $\mathsf{T}_{\kappa}^{\mathsf{loc}}(M)$, where
\begin{align*} A = -[t \partial_{x_i}, a(X[r])], \quad B_j &= \frac{1}{x_i - x_j} \left[1-\Omega^{(i,j)}, a(X[r]) \right] \\
C &= \sum_{p \geq 0}\left[ x_i^p \Omega^{(i,\infty)}_{[p+1]},a(X[r])\right]. \end{align*}
We compute:
\begin{align*}
A &= -[t \partial_{x_i}, x_i^rX^{(i)}] = -rtx_i^{r-1}X^{(i)}, \\
B 
&= \frac{-1}{x_i - x_j} \left[\Omega^{(i,j)}, x_i^rX^{(i)} + x_j^rX^{(j)} \right]\\
&= \frac{x_i^r - x_j^r}{x_i - x_j}\left[\Omega^{(i,j)}, X^{(j)} \right] \\
&= \sum_{1 \leq p \leq r} \sum_{k,l} ([e_{kl},X][r-p])^{(j)}(e_{lk}[p-1])^{(i)}, \\
C &= \sum_{p \geq 0} \sum_{k,l} x_i^p \left[e_{kl}^{(i)}e_{lk}[p+1]^{(\infty)}, x_i^rX^{(i)} + X[-r]^{(\infty)}\right] \\
&= \sum_{p \geq 0} \sum_{k,l} \left( x_i^{p+r} [e_{kl}, X]^{(i)}e_{lk}[p+1]^{(\infty)} + x_i^pe_{kl}^{(i)}\left[e_{kl}[p+1]^{(\infty)}, X[-r]^{(\infty)}\right]\right) \\
&= \sum_{p \geq 0} \sum_{k,l} \left( -x_i^{p+r} e_{kl}^{(i)}([e_{lk},X][p+1])^{(\infty)} + x_i^pe_{kl}^{(i)}\left[e_{kl}[p+1]^{(\infty)}, X[-r]^{(\infty)}\right]\right) \\
& \quad + \sum_{k,l} r x_i^{r-1}\langle e_{lk},X\rangle_\kappa e_{kl}^{(i)} \\
&= \sum_{1 \leq p \leq r} \sum_{k,l} ([e_{kl},X][p-r])^{(\infty)}(e_{lk}[p-1])^{(i)} + \sum_{k,l} r x_i^{r-1}\langle e_{lk},X\rangle_\kappa e_{kl}^{(i)}.
\end{align*}
Therefore, we have 
\begin{align*}
[{}^\kappa\nabla'(q_i), a(X[r])] &= -rtx_i^{r-1}X^{(i)} + \sum_{j \neq i} \sum_{1 \leq p \leq r} \sum_{k,l} ([e_{kl},X][r-p])^{(j)}(e_{lk}[p-1])^{(i)} \\
& \quad + \sum_{1 \leq p \leq r} \sum_{k,l} ([e_{kl},X][p-r])^{(\infty)}(e_{lk}[p-1])^{(i)} + \sum_{k,l} r x_i^{r-1}\langle e_{lk},X\rangle_\kappa e_{kl}^{(i)}\\
&= -rtx_i^{r-1}X^{(i)} + \sum_{1 \leq p \leq r} \sum_{k,l}a([e_{kl},X][r-p])(e_{lk}[p-1])^{(i)} - D \\
& \quad  + rx_i^{r-1}(\kappa X^{(i)} - \Tr(X)), 
\end{align*}
where
\begin{align*} 
D &= \sum_{1 \leq p \leq r} \sum_{k,l} ([e_{kl},X][r-p])^{(i)}(e_{lk}[p-1])^{(i)} \\
&= \sum_{1 \leq p \leq r} \sum_{k,l} (x_i^{r-p}[e_{kl},X]^{(i)})(x_i^{p-1}e_{lk}^{(i)}) \\
&= rx_i^{r-1} \sum_{k,l} \left( e_{kl}^{(i)} X^{(i)}e_{lk}^{(i)} - X^{(i)}e_{kl}^{(i)}e_{lk}^{(i)}\right) \\
&= - rx_i^{r-1}(\Tr(X) + nX^{(i)}).
\end{align*}
Hence, we have 
\[
[{}^\kappa\nabla'(q_i), a(X[r])] \equiv r x_i^{r-1} \left(-tX^{(i)} + \Tr(X) + nX^{(i)} + \kappa X^{(i)} - \Tr(X) \right) = 0 
\]
modulo $\g[t] \cdot \mathsf{T}_{\kappa}^{\mathsf{loc}}(M)$. This shows that the connection ${}^\kappa\nabla'$ normalizes the $\g[t]$-action on $\mathsf{T}_\kappa^{\mathsf{loc}}(M)$. Part b) is standard - detailed proofs can be found in, e.g., \cite[Proposition 3.6]{VV} or \cite[Lemma 2.1, Proposition 2.6]{FFTL}. 
\end{proof}

\begin{cor} \label{cor: Suzuki all levels} Let $\kappa \in \C$. 
\begin{enumerate}[label=\alph*), font=\textnormal,noitemsep,topsep=3pt,leftmargin=1cm]
\itemsep0em
\item 
We have functors
\begin{alignat}{7} \label{T functor first appearance} 
\mathsf{T}_\kappa \colon& \mathscr{C}_\kappa \to (\mathbf{U}(\g[t]),\mathcal{H}_{\kappa+n})\Nmod{},& \quad& M \mapsto \mathsf{T}_\kappa(M),\\ \label{Suzuki functor first appearance} 
\mathsf{F}_\kappa \colon& \mathscr{C}_\kappa \to \mathcal{H}_{\kappa+n}\Lmod{},& \quad&  M \mapsto  H_0(\g[t], \mathsf{T}_{\kappa}(M)).
\end{alignat}
\item The map \eqref{global-current coinv iso} is an $\mathcal{H}_{\kappa+n}$-module isomorphism and the functors $\mathbb{F}_\kappa$ and $\mathsf{loc} \circ \mathsf{F}_\kappa$ are naturally isomorphic. 
\item The functor \eqref{Suzuki functor first appearance} is right-exact and commutes with direct sums.  
\end{enumerate}
\end{cor}

\begin{proof}
By Proposition \ref{pro:y preserves subspace}, composing the Dunkl embedding \eqref{Dunkl embedding} with the connection ${}^\kappa\nabla'$ yields a representation of $\mathcal{H}_{\kappa+n}$ on $\mathsf{T}_\kappa(M)$. Proposition \ref{pro: coinv iso}.a) implies that this representation descends to a representation on $H_0(\g[t], \mathsf{T}_\kappa(M))$. 
This proves part a). Part b) follows directly from \ref{pro: coinv iso}.b). Part c) follows from the fact that $\mathsf{T}_\kappa$ is exact and taking coinvariants is right exact and commutes with direct sums. 
\end{proof} 

\begin{defi} \label{Suzuki first definition} 
Given $\kappa \in \C$, we call 
\begin{equation} \label{definition of Suzuki displayed} \mathsf{F}_\kappa \colon \mathscr{C}_\kappa \to \mathcal{H}_{\kappa+n}\Lmod{}, \quad
M \mapsto H_0(\g[t],  \C[\h] \otimes (\mathbf{V}^*)^{\otimes m} \otimes M) 
\end{equation} the \emph{Suzuki functor} (of level $\kappa$). 
\end{defi} 

The functor \eqref{definition of Suzuki displayed} extends Suzuki's construction from \cite{Suz} to the critical level case. 
Indeed, setting $\kappa = \mathbf{c}$, we get the functor
\[ \mathsf{F}_{\mathbf{c}} \colon \mathscr{C}_{\mathbf{c}} \to \Lmod{\mathcal{H}_{0}} \]
relating the affine Lie algebra at the critical level to the rational Cherednik algebra at $t=0$. 

\begin{rem}
In \cite{VV} Varagnolo and Vasserot constructed functors from $\mathscr{C}_\kappa$ ($\kappa \neq \mathbf{c}$) to the category 
of modules over the rational Cherednik algebra ($t\neq 0$) associated to the wreath product $(\Z/l\Z) \wr \mathfrak{S}_m$. We expect that our approach to extending the Suzuki functor to the $\kappa = \mathbf{c}$, $t=0$ case can also be applied to their functors. 
\end{rem}

\section{Suzuki functor - further generalizations} \label{almost smooth}

The Suzuki functor has so far been defined on smooth $\widehat{\mathbf{U}}_\kappa$-modules. We now extend its definition to all $\widehat{\mathbf{U}}_\kappa$-modules in several steps. We first extend it to finitely presented modules using a certain inverse limit construction. We then introduce an even more general definition which applies to all modules. Let $\kappa \in \C$ and $t = \kappa+n$ throughout this section. 

\subsection{Pro-smooth modules.} 

We are going to define the category of pro-smooth modules and the pro-smooth completion functor. If $\mathscr{I}$ is an inverse system in some category, we write $\lim \mathscr{I}$ or $\lim_{M_i \in \mathscr{I}} M_i$, where the $M_i$ run over the objects in $\mathscr{I}$, for its inverse limit. We start with the following auxiliary  lemma. 

\begin{lem} \label{kernel cosmooth} 
Let $M$ be any $\widehat{\mathbf{U}}_\kappa$-module, $N$ a smooth module and $f \colon M \to N$ a $\widehat{\mathbf{U}}_\kappa$-module homomorphism. Then $M/\ker f$ is a smooth module.
\end{lem} 

\begin{proof}
Let $v \in M$ and let $\bar{v}$ be the image of $v$ in $M/\ker f$. Since $N$ is smooth, there exists $r \geq 0$ such that $\hat{I}_r \cdot f(v) =0$. Hence $f(\hat{I}_r \cdot v) = 0$, $\hat{I}_r \cdot v \subseteq \ker f$ and so $\hat{I}_r \cdot \bar{v} = 0$. 
\end{proof}

\begin{defi}
A $\widehat{\mathbf{U}}_\kappa$-module $M$ is called \emph{pro-smooth} if $M$ is the inverse limit of an inverse system of smooth $\widehat{\mathbf{U}}_\kappa$-modules. Let $\widetilde{\mathscr{C}}_\kappa$ denote the full subcategory of $\widehat{\mathbf{U}}_\kappa\Lmod{}$ whose objects are pro-smooth modules. 
\end{defi}

\begin{defi}
Let $M$ be a $\widehat{\mathbf{U}}_\kappa$-module. The smooth quotients of $M$ form an inverse system $\mathscr{I}_M$ partially ordered by projections. 
Let
\[ \widetilde{M} := \lim \mathscr{I}_M. \] 
\end{defi}

\begin{pro} \label{pro: uni prop adjunction EW thm}
There exists a ``pro-smooth completion" functor
\begin{equation}\label{completion functor} \widehat{\mathbf{U}}_\kappa\Lmod{} \to \widetilde{\mathscr{C}}_\kappa, \quad M \mapsto \widetilde{M}, \ f \mapsto \tilde{f} \end{equation} 
left adjoint to the inclusion functor $\widetilde{\mathscr{C}}_\kappa \hookrightarrow  \widehat{\mathbf{U}}_\kappa\Lmod{}$. 
\end{pro} 

\begin{proof}
We first construct $\tilde{f}$ explicitly. 
Let $f \colon M \to N$ be a homomorphism of $\widehat{\mathbf{U}}_\kappa$-modules.
Given a smooth quotient $N_i$ of $N$, let $f_i$ be the map $M \xrightarrow{f} N \twoheadrightarrow N_i$. By Lemma \ref{kernel cosmooth}, $M_i := M/\ker f_i$ is a smooth module. Hence, there is a canonical map $\widetilde{M} \to M_i$ as part of the inverse limit data. Consider the diagram on the LHS below, where $N_j$ is another smooth quotient of $N$ and all the unnamed maps are part of the inverse system or inverse limit data. Since the outer pentagon commutes, the universal property of the inverse limit $\widetilde{N}$ implies that there exists a unique map $\tilde{f}$ making the diagram commute.  
\[
\begin{tikzcd}[column sep=small, row sep=small]
& \widetilde{M} \arrow["\tilde{f}", dashed]{d} \arrow[bend right]{ld} \arrow[bend left]{rd} &  \\
M_i \arrow["f_i", swap]{d} & \widetilde{N} \arrow{dl} \arrow{dr} & M_j \arrow["f_j"]{d} \\
N_i \arrow{rr} & & N_j
\end{tikzcd} \quad \quad \quad \begin{tikzcd}[column sep=small, row sep=small]
& \widetilde{M} \arrow["g'", dashed]{d} \arrow[bend right]{ld} \arrow[bend left]{rd} &  \\
M_i \arrow["g_i", swap]{d} & K \arrow{dl} \arrow{dr} & M_j \arrow["g_j"]{d} \\
K_i \arrow{rr} & & K_j
\end{tikzcd} \quad \quad \quad
\begin{tikzcd}[column sep=small, row sep=small]
 & M \arrow["\iota_M", dashed]{d}\arrow[bend right, twoheadrightarrow]{ldd}\arrow[bend left, twoheadrightarrow]{rdd}  & \\ 
 & \widetilde{M} \arrow{ld} \arrow{rd} &  \\
 M_i \arrow[rr] & & M_j 
\end{tikzcd}
\]

Next we construct the adjunction. Let $g \colon M \to K$ be a homomorphism of $\widehat{\mathbf{U}}_\kappa$-modules, and assume that $K$ is the inverse limit of an inverse system of smooth modules. Given such a smooth module $K_i$, let $g_i$ be the composition of $g$ with the canonical map $K \to K_i$. By Lemma \ref{kernel cosmooth}, $g_i$ factors through the smooth module $M_i := M/\ker g_i$. An analogous argument to the one above shows that there exists a unique map $g'$ making the middle diagram above commute. 
The universal property of the inverse limit $\widetilde{M}$ also yields a unique  map $\iota_M$ making the diagram on the RHS above commute.

It is easy to check that the maps  
\begin{equation} \label{adjunction pro smooth} \Hom_{\widetilde{\mathscr{C}}_\kappa}(\widetilde{M},K) \cong \Hom_{\widehat{\mathbf{U}}_\kappa}(M,K), \quad h \mapsto h \circ \iota_M, \ g' \mapsfrom g \end{equation}  
are mutually inverse bijections. This gives the adjunction. 
\end{proof}

\begin{pro} \label{pro: equiv: prosmooth and cUEA-mod} 
The restriction of \eqref{completion functor} to $\mathscr{C}_\kappa$ or $\widehat{\mathbf{U}}_\kappa\FPmod{}$ is naturally isomorphic to the identity functor. 
\end{pro} 

\begin{proof}
If $M$ is smooth then $M$ is the greatest element in the inverse system $\mathscr{I}_M$, so $\widetilde{M} = M$. Next suppose that $M$ is finitely presented with presentation
\[ (\widehat{\mathbf{U}}_\kappa)^{\oplus a} \xrightarrow{f} (\widehat{\mathbf{U}}_\kappa)^{\oplus b} \to M \to 0.\] 
We first show that $(\widehat{\mathbf{U}}_\kappa)^{\widetilde{}} = \widehat{\mathbf{U}}_\kappa$. 
The inverse system $\mathscr{I}':=\{ \widehat{\mathbf{U}}_\kappa/\hat{I}_r \mid r \geq 0 \}$ is a subsystem of $\mathscr{I}:=\mathscr{I}_{\widehat{\mathbf{U}}_\kappa}$. Suppose that $N=\widehat{\mathbf{U}}_\kappa/J$ is a smooth quotient and let $\bar{1}$ be the image of $1$ in $N$. Then, by smoothness, $\hat{I}_r.\bar{1}=0$ for some $r \geq 0$. Hence $\hat{I}_r \subseteq J$ and $N$ is a quotient of $\widehat{\mathbf{U}}_\kappa/\hat{I}_r$. Therefore $\mathscr{I}'$ is a cofinal subsystem of $\mathscr{I}$ and 
\[ (\widehat{\mathbf{U}}_\kappa)^{\widetilde{}} := \lim \mathscr{I} = \lim \mathscr{I}' = \widehat{\mathbf{U}}_\kappa.\] 
The fact that limits commute with finite direct sums implies that the pro-smooth completion functor sends $(\widehat{\mathbf{U}}_\kappa)^{\oplus a}$ to itself. Hence $\iota_{(\widehat{\mathbf{U}}_\kappa)^{\oplus a}} = \id$ and $\tilde{f} = f'$, using the notation from \eqref{adjunction pro smooth}. The adjunction \eqref{adjunction pro smooth}, therefore, implies that $\tilde{f} = f$. 
By Proposition \ref{pro: uni prop adjunction EW thm}, the pro-smooth completion functor is left adjoint and, hence, right exact. Hence  $(\coker f)^{\widetilde{}} = \coker \tilde{f} = \coker f =~M$. 
\end{proof}

We will also need the following lemma.

\begin{lem} \label{lem: pro homos 2} 
Let $M$ and $N$ be $\widehat{\mathbf{U}}_\kappa$-modules. Then 
\[ \Hom_{\widetilde{\mathscr{C}}_\kappa}(\widetilde{M},\widetilde{N}) = \lim_{N_i \in \mathscr{I}_N} \underset{M_j \in \mathscr{I}_M}{\colim} \Hom_{\widehat{\mathbf{U}}_\kappa}(M_j,N_i).\]
\end{lem}

\begin{proof}
The equality $\Hom_{\widetilde{\mathscr{C}}_\kappa}(\widetilde{M},\widetilde{N}) = \lim_{N_i \in \mathscr{I}_N} \Hom_{\widehat{\mathbf{U}}_\kappa}(\widetilde{M},N_i)$ follows from the general properties of limits. Therefore it suffices to show that, for each $N_i \in \mathscr{I}_N$,  
\begin{equation} \label{Hom colimit} \Hom_{\widehat{\mathbf{U}}_\kappa}(\widetilde{M},N_i) = \underset{M_j \in \mathscr{I}_M}{\colim} \Hom_{\widehat{\mathbf{U}}_\kappa}(M_j,N_i).\end{equation}  
Let us check that the LHS of \eqref{Hom colimit} satisfies the universal property of the colimit. Suppose that we are given a vector space $X$ and linear functions $\chi_{M_j} \colon \Hom_{\widehat{\mathbf{U}}_\kappa}(M_j,N_i) \to X$, for each $M_j \in \mathscr{I}_M$, which commute with the natural inclusions between the Hom-spaces. 
We are now going to define a map $\chi \colon \Hom_{\widehat{\mathbf{U}}_\kappa}(\widetilde{M},N_i) \to X$. If $f \in \Hom_{\widehat{\mathbf{U}}_\kappa}(\widetilde{M},N_i)$, then, by Lemma \ref{kernel cosmooth}, the module $\overline{M} := \widetilde{M} / \ker f$ is smooth. Let $\bar{f} \colon \overline{M} \to N_i$ be the homomorphism induced by $f$. We define $\chi$ by setting $\chi(f) := \chi_{\overline{M}}(\bar{f})$. One can easily see that $\chi$ is the unique map making the diagram below commute (where $M_k$ is another smooth quotient of $M$ and all the unnamed maps are the canonical ones). 
\[ 
\begin{gathered}[b]
\begin{tikzcd}[column sep=tiny, row sep=small]
 \Hom_{\widehat{\mathbf{U}}_\kappa}(M_j,N_i)\arrow[rr] \arrow["\chi_{M_j}",swap]{rd} \arrow[bend right]{rdd} & & \Hom_{\widehat{\mathbf{U}}_\kappa}(M_k,N_i) \arrow["\chi_{M_k}"]{ld} \arrow[bend left]{ldd}\\
 & X \arrow["\chi"]{d} & \\
& \Hom_{\widehat{\mathbf{U}}_\kappa}(\widetilde{M},N_i)   & 
\end{tikzcd} \\[-\dp\strutbox]
\end{gathered}\qedhere
\] 
\end{proof} 

\subsection{Extension to all modules.} \label{sec: ext to all modules}
We start by extending the Suzuki functor from Definition \eqref{Suzuki first definition} to the category $\widehat{\mathbf{U}}_\kappa\FPmod{}$. Suppose that $M$ is a finitely presented $\widehat{\mathbf{U}}_\kappa$-module.  
By Proposition \ref{pro: equiv: prosmooth and cUEA-mod}, we have $M= \widetilde{M} := \lim \mathscr{J}_M$. Set 
\begin{equation} \label{commutes with limits} \mathsf{F}_\kappa(M) := \lim_{M_i \in \mathscr{I}_M} \mathsf{F}_\kappa(M_i),\end{equation} 
where the limit is taken in the category $\mathcal{H}_{t}\Lmod{}$. If $M$ is smooth then $M$ is the maximal element in the inverse system $\mathscr{I}_M$, so \eqref{commutes with limits} is compatible with the previous definition of~$\mathsf{F}_\kappa$.  

\begin{pro} \label{pro: Suzuki limit ext}
The functor \eqref{Suzuki functor first appearance} extends to a right exact functor \begin{equation} \label{extended functor} \mathsf{F}_{\kappa} \colon \widehat{\mathbf{U}}_\kappa\FPmod{} \to \mathcal{H}_{t}\Lmod{}, 
\end{equation}
which preserves finite direct sums. 
\end{pro}

\begin{proof}
We need to construct maps between Hom-sets. 
Suppose that $N = \widetilde{N}$ is another finitely presented $\widehat{\mathbf{U}}_\kappa$-module. 
Let $N_i \in \mathscr{I}_N$. 
For all $M_j \in \mathscr{J}_M$, we have maps
\[ \phi_j \colon \Hom_{\widehat{\mathbf{U}}_\kappa}(M_j,N_i) \xrightarrow{\mathsf{F}_\kappa} \Hom_{\mathcal{H}_{t}}(\mathsf{F}_\kappa(M_j),\mathsf{F}_\kappa(N_i)) \to \Hom_{\mathcal{H}_{t}}(\mathsf{F}_\kappa(M),\mathsf{F}_\kappa(N_i))\]
compatible with the transition maps of the direct system $\{ \Hom_{\widehat{\mathbf{U}}_\kappa}(M_j,N_i) \mid M_j \in \mathscr{J}_M\}$. 
The universal property of the colimit and \eqref{Hom colimit} yield a canonical map
\[ \psi_i \colon \Hom_{\widehat{\mathbf{U}}_\kappa}(M,N_i) = \underset{M_j \in \mathscr{I}_M}{\colim} \Hom_{\widehat{\mathbf{U}}_\kappa}(M_j,N_i) \to \Hom_{\mathcal{H}_{t}}(\mathsf{F}_{\kappa}(M),\mathsf{F}_{\kappa}(N_i)).\] 
The maps $\psi_i$ are compatible with the transition maps of the inverse system \linebreak $\{ \Hom_{\mathcal{H}_{t}}(\mathsf{F}_{\kappa}(M),\mathsf{F}_{\kappa}(N_i)) \mid N_i \in \mathscr{I}_N\}$. 
Hence the universal property of the limit 
yields a canonical map 
\begin{equation} \label{fpmod functoriality} \Hom_{\widehat{\mathbf{U}}_\kappa}(M,N) =\lim_{N_i \in \mathscr{I}_N} \Hom_{\widehat{\mathbf{U}}_\kappa}(M,N_i) \to \Hom_{\mathcal{H}_{t}}(\mathsf{F}_{\kappa}(M),\mathsf{F}_{\kappa}(N)).\end{equation} 
Therefore \eqref{extended functor} is in fact a functor. 

Since limits commute with finite direct sums, \eqref{extended functor} must preserve finite direct sums. 
We now prove right exactness. Suppose that we have a short exact sequence 
\begin{equation} \label{ses completion} 0 \to A \to B \to C \to 0 \end{equation}
in $\widehat{\mathbf{U}}_\kappa\FPmod{}$. By Proposition \ref{pro: equiv: prosmooth and cUEA-mod}, these modules are pro-smooth, and there exists a short exact sequence of inverse systems of smooth quotients 
\[ \{ 0 \to A_i \to B_i \to C_i \to 0 \mid i \in \Z_{\geq 0}\} \]
whose limit is \eqref{ses completion}. Since we are dealing with inverse systems of smooth quotients, the structure maps are all epimorphisms. Next, note that the functor $\mathsf{F}_\kappa$ is right exact on smooth modules by Corollary \ref{cor: Suzuki all levels}.c). Hence, after applying $\mathsf{F}_\kappa$, we get a short exact sequence of inverse systems of $\mathcal{H}_{t}$-modules 
\[ \{ \mathsf{F}_\kappa(A_i) \to \mathsf{F}_\kappa(B_i) \to \mathsf{F}_\kappa(C_i) \to 0 \mid i \in \Z_{\geq 0}\}, \]
where the structure maps are still epimorphisms. By \cite[Lemma 10.86.1]{stacks}, after taking the inverse limit, we get the sequence
\[ \mathsf{F}_\kappa(A) \to \mathsf{F}_\kappa(B) \to \mathsf{F}_\kappa(C) \to 0, \]
proving right-exactness. 
\end{proof}

\begin{cor} \label{cor: EW Suzuki tensor} 
The space $\mathsf{F}_\kappa(\widehat{\mathbf{U}}_{\kappa})$ is a $(\mathcal{H}_{t},\widehat{\mathbf{U}}_{\kappa})$-bimodule. There exists a natural isomorphism of functors 
\begin{equation} \label{EW Suzuki tensor} \mathsf{F}_\kappa(-) \cong \mathsf{F}_\kappa(\widehat{\mathbf{U}}_{\kappa}) \otimes_{\widehat{\mathbf{U}}_{\kappa}}- \colon \widehat{\mathbf{U}}_\kappa\FPmod{} \to \mathcal{H}_{t}\Lmod{}.\end{equation}  
\end{cor}

\begin{proof}
If we take $M=N=\widehat{\mathbf{U}}_{\kappa}$ then \eqref{fpmod functoriality} is an algebra homomorphism  $$\widehat{\mathbf{U}}_{\kappa}^{op} \to \End_{\mathcal{H}_{t}}(\mathsf{F}_{\kappa}(\widehat{\mathbf{U}}_{\kappa}),\mathsf{F}_{\kappa}(\widehat{\mathbf{U}}_{\kappa}))$$
giving the right $\widehat{\mathbf{U}}_{\kappa}$-module structure. 

The second statement is proven in the same way as the Eilenberg-Watts theorem (see, e.g., \cite[Theorem 5.45]{Rot}). Let us briefly summarize the argument. One first uses the fact that $\mathsf{F}_\kappa$ preserves finite direct sums to show that the isomorphism \eqref{EW Suzuki tensor} holds for the category of finitely generated free $\widehat{\mathbf{U}}_{\kappa}$-modules. One then concludes that \eqref{EW Suzuki tensor} holds for arbitrary finitely presented modules by using the right exactness of $\mathsf{F}_\kappa$ together with the five lemma. 
\end{proof}

We now introduce the final and most general definition of the Suzuki functor. 

\begin{defi} \label{general definition of Suzuki} 
The functor \eqref{extended functor}, in the realization \eqref{EW Suzuki tensor}, extends to the colimit-preserving functor 
\begin{equation} \label{Suzuki functor - final} \mathsf{F}_\kappa(-) := \mathsf{F}_\kappa(\widehat{\mathbf{U}}_{\kappa}) \otimes_{\widehat{\mathbf{U}}_{\kappa}}- \ \colon \ \widehat{\mathbf{U}}_\kappa\Lmod{} \to \mathcal{H}_{t}\Lmod{}.\end{equation} 
From now on we will refer to \eqref{Suzuki functor - final} as \emph{the Suzuki functor}. 
\end{defi} 

\begin{rem} \label{remark colimits Suzuki} 
Let us make several remarks about the definition above.  
\begin{enumerate}[label=\alph*), font=\textnormal,noitemsep,topsep=3pt,leftmargin=1cm]
\itemsep0em
\item In Corollary \ref{cor: EW Suzuki tensor} we had to restrict ourselves to the category $\widehat{\mathbf{U}}_\kappa\FPmod{}$ because inverse limits do not commute with infinite coproducts. However, the functor  \eqref{Suzuki functor - final} preserves all colimits since it is left adjoint to the functor $N \mapsto \Hom_{\mathcal{H}_{t}}(\mathsf{F}_{\kappa}(\widehat{\mathbf{U}}_{\kappa}),N)$. 
\item We now have three definitions of the Suzuki functor:
\begin{itemize}
\item the ``coinvariants definition'' for smooth modules: $\mathsf{F}_\kappa(M) = H_0(\g[t], \mathsf{T}_{\kappa}(M))$, 
\item the ``limit definition'' for finitely presented modules: $\mathsf{F}_\kappa(M) = \lim_{M_i \in \mathscr{I}_M} \mathsf{F}_\kappa(M_i)$, 
\item the ``tensor product'' definition for all modules: $\mathsf{F}_\kappa(M) = \mathsf{F}_\kappa(\widehat{\mathbf{U}}_{\kappa}) \otimes_{\widehat{\mathbf{U}}_{\kappa}}M$. 
\end{itemize}
The limit definition agrees with the coinvariants definition, when restricted to smooth modules, by the comments preceding Proposition \ref{pro: Suzuki limit ext}. The tensor product definition agrees with the limit definition by Corollary \ref{cor: EW Suzuki tensor}. 
\end{enumerate} 
\end{rem}

\subsection{A generic functor.} 

Considering $t$ as an indeterminate, one obtains flat $\C[t]$-algebras $\widehat{\mathbf{U}}_{\C[t]}$ and $\mathcal{H}_{\C[t]}$ such that 
$$\widehat{\mathbf{U}}_{\C[t]}/(t-\xi)\widehat{\mathbf{U}}_{\C[t]} \cong \widehat{\mathbf{U}}_{\xi-n}, \quad \mathcal{H}_{\C[t]}/(t-\xi)\mathcal{H}_{\C[t]} \cong \mathcal{H}_{\xi}$$ for all $\xi \in \C$. More details on the algebra $\mathcal{H}_{\C[t]}$, often called the \emph{generic} rational Cherednik algebra, can be found in \cite[\S3]{BR}. We  have specialization functors
\begin{alignat*}{7}
\mathsf{spec}_{t=\xi} \colon& \ \widehat{\mathbf{U}}_{\C[t]}\Lmod{} \to \widehat{\mathbf{U}}_{\xi-n}\Lmod{}, \quad& \ M \mapsto& \ M/(t-\xi)\cdot M, \\
\mathsf{spec}_{t=\xi} \colon& \ \mathcal{H}_{\C[t]}\Lmod{} \to \mathcal{H}_{\xi}\Lmod{}, \quad& \ M \mapsto& \ M/(t-\xi)\cdot M.
\end{alignat*}
One can easily verify that our construction of the functor $\mathsf{F}_\kappa$ still makes sense if we treat $t$ as a variable throughout. Therefore, we obtain the \emph{generic Suzuki functor}
\[ \mathsf{F}_{\C[t]} \colon \widehat{\mathbf{U}}_{\C[t]}\Lmod{} \to \mathcal{H}_{\C[t]}\Lmod{},\]
which commutes with the specialization functors, i.e., $\mathsf{spec}_{t=\xi} \circ \mathsf{F}_{\C[t]} = \mathsf{F}_{\xi -n} \circ \mathsf{spec}_{t=\xi}$. 

\section{Computation of the Suzuki functor}
\label{sec: computation of Suzuki functor}

In this section we compute the Suzuki functor on certain induced $\mathbf{U}_\kappa(\hat{\g})$-modules, showing that the generalized Verma modules from Definition \ref{def: gen Vermas} as well as the regular module $\mathcal{H}_{t}$ are in the image of $\mathsf{F}_\kappa$. Let $\kappa \in \C$ and $t= \kappa+n$ throughout. 

\subsection{Induced modules} \label{weyl modules section} 

We start by recalling the definition of Verma modules. 
\begin{defi} 
Let $\lambda \in \t^*$ and 
let $\C_{\lambda,1}$ be the one-dimensional $\t\oplus\C\mathbf{1}$-module of weight $(\lambda,1)$. 
The corresponding \emph{Verma module} is 
\[ \mathbb{M}_{\kappa}(\lambda) := \Ind^{\hat{\g}_{\kappa}}_{\hat{\b}_+}\circ \Inf_{\t\oplus\C\mathbf{1}}^{\hat{\b}_+}\C_{\lambda,1}. \]
\end{defi} 

We next define certain induced modules which generalize the Weyl modules from \cite[\S 2.4]{KLI} (see also \cite[\S 9.6]{Fre}). 
Given $l\geq 1$ and $\mu \in \mathcal{C}_l(n)$, define 
\begin{equation} \label{def: l-mu bar} \hat{\mathfrak{l}}^+_{\mu} := \mathfrak{l}_{\mu} \oplus \hat{\g}_{\geq 1} \oplus \C\mathbf{1} \subseteq \hat{\g}_+, \quad \bar{\mathfrak{l}}_\mu := \hat{\mathfrak{l}}_\mu^+/\mathfrak{j}_\mu, \end{equation}
where $\mathfrak{j}_\mu :=\n_-[1] \oplus \n_+[1] \oplus (\t[1] \cap [\mathfrak{l}_\mu,\mathfrak{l}_\mu][1]) \oplus \hat{\g}_{\geq 2}$. 
\begin{lem} \label{lem:leviideal} 
The subspace $\mathfrak{j}_\mu$ 
is an ideal in the Lie algebra $\hat{\mathfrak{l}}_\mu^+$. Moreover, there is a Lie algebra isomorphism $\bar{\mathfrak{l}}_\mu \cong \mathfrak{l}_\mu \oplus \mathfrak{z}_\mu[1] \oplus \C \mathbf{1}$, 
where $\mathfrak{z}_\mu$ denotes the centre of $\l_\mu$. 
\end{lem}

\begin{proof}
Since $\mathfrak{j}_\mu \subset \hat{\g}_{\geq 1}$, we have $[\hat{\g}_{\geq 1}, \mathfrak{j}_\mu] \subseteq \hat{\g}_{\geq 2} \subset \mathfrak{j}_\mu$. Therefore it suffices to show that $[\mathfrak{l}_{\mu}, \mathfrak{j}_\mu] \subseteq \mathfrak{j}_\mu$. This follows from the fact that $\mathfrak{j}_\mu = [\mathfrak{l}_\mu,\mathfrak{l}_\mu][1] \oplus \mathfrak{r}[1] \oplus  \hat{\g}_{\geq 2}$, where $\mathfrak{r}$ is the direct sum of the nilradical of the standard parabolic containing $\l_\mu$ and the nilradical of the opposite parabolic,  together with the following three inclusions. Firstly, we have $[\mathfrak{l}_{\mu}, \hat{\g}_{\geq 2}] \subseteq \hat{\g}_{\geq 2} \subset \mathfrak{j}_\mu$. Secondly, $[\mathfrak{l}_{\mu},\mathfrak{l}_{\mu}[1]] \subseteq [\mathfrak{l}_{\mu},\mathfrak{l}_{\mu}][1] \subset \mathfrak{j}_\mu$. Thirdly, $[\mathfrak{l}_{\mu},\mathfrak{r}[1]] \subseteq \mathfrak{r}[1] \subset~\mathfrak{j}_\mu$. The second statement of the lemma follows immediately. 
\end{proof}

Let $\Uu_1(\bar{\l}_\mu) := \Uu(\bar{\l}_\mu)/\langle \mathbf{1} - 1 \rangle$. Consider the functor
\begin{equation} \label{Ind a definition} \mathsf{Ind}_{\mu,\kappa} = \Ind^{\hat{\g}_\kappa}_{\hat{\l}_\mu^+} \circ \Inf_{\bar{\l}_\mu}^{\hat{\l}_\mu^+} \colon \Uu_1(\bar{\l}_\mu)\Lmod{} \to \mathscr{C}_\kappa.\end{equation} 
In the case $\mu = (1^n)$ we abbreviate $\mathsf{Ind}_{\kappa}:=\mathsf{Ind}_{(1^n),\kappa}$. Note that $\hat{\mathfrak{l}}^+_{(1^n)} = \hat{\t}_+$. Set $\mathfrak{i} := \mathfrak{j}_{(1^n)} = \n_-[1] \oplus \n_+[1] \oplus \hat{\g}_{\geq 2}$ and $\bar{\t} := \hat{\t}_+/\mathfrak{i}$. By Lemma \ref{lem:leviideal}, we have $\bar{\t} \cong \t \oplus \t[1]$.

\begin{defi}  
Let $\mu \in \mathcal{C}_l(n)$, $\lambda \in \Pi_\mu^+$ and $a \in (\t[1])^*$ with $\mathfrak{S}_n(a) = \mathfrak{S}_\mu$ (with respect to the usual Weyl group action). Extend $L(\lambda)$ to an $\Uu_1(\bar{\l}_\mu)$-module $L(a,\lambda)$ by letting $\mathfrak{z}_\mu[1]$ act via the weight $a$. 
We define the \emph{Weyl module} of type $(a,\lambda,\kappa)$ to be 
\[ \mathbb{W}_{\kappa}(a,\lambda) := \mathsf{Ind}_{\mu,\kappa}(L(a,\lambda)). \]
\end{defi} 

\begin{rem}
As a special case, when $a=0$, we obtain modules $\mathbb{W}_{\kappa}(\lambda):=\mathbb{W}_{\kappa}(0,\lambda)$ which coincide with the Weyl modules from \cite[\S 2.4]{KLI}.  
\end{rem}

\begin{defi} \label{HVM defi} Assume that $n=m$. 
Let $\mathfrak{I}_\kappa$ be the left ideal in $\Uu(\hat{\g}_\kappa)$ generated by $e_{ii} - 1_{\hat{\g}_\kappa} \ (1 \leq i \leq n)$, $\mathbf{1} - 1_{\hat{\g}_\kappa}$ and $\mathfrak{i} := \n_-[1] \oplus \n_+[1] \oplus \hat{\g}_{\geq 2}$. 
Define
\[ \mathbb{H}_{\kappa} := \Uu(\hat{\g}_\kappa)/\mathfrak{I}_\kappa = \mathsf{Ind}_{\kappa}(\mathcal{I}),\] 
where
$\mathcal{I} := \Ind_{\t\oplus \C \mathbf{1}}^{\bar{\t}}\C_{(1^n,1)} \cong S(\t[1]).$ 
\end{defi}

The module $\mathbb{H}_{\kappa}$ 
is cyclic, generated by the image $1_{\mathbb{H}}$ of $1_{\hat{\g}_\kappa} \in \Uu(\hat{\g}_\kappa)$. 
From now on, whenever $n=m$, let us identify 
\begin{equation} \label{symT-Ch-iso} \mathcal{I} \cong S(\t[1]) \to \C[\h^*], \quad e_{ii}[1] \mapsto - y_i.\end{equation} 

\subsection{Statement of the results.} \label{sec: statements comp}

We state the three main results of this section. The first one implies that the regular module appears in the image of the Suzuki functor.

\begin{thm} \label{thm: regular module} \label{Regular module thm} Let $n=m$. The map 
\begin{align} \label{Regular module iso}
\Upsilon \colon \mathcal{H}_{t} \xrightarrow{\sim}& \ \mathsf{F}_{\kappa}(\mathbb{H}_\kappa), \\ \label{Regular module iso 2} f(x_1,\hdots,x_n) w g(y_1,\hdots,y_n) \mapsto& \ [f(x_1,\hdots,x_n) \otimes e_{w}^* \otimes g(-e_{11}[1],\hdots,-e_{nn}[1])1_{\mathbb{H}}] 
\end{align} 
is an isomorphism of $\mathcal{H}_{t}$-modules. 
\end{thm}

The next theorem states that the Suzuki functor sends generalized Weyl modules to generalized Verma modules.  

\begin{thm} \label{thm: standard to Weyl} Let $n=m$. Take $l \geq 1$, $\mu \in \mathcal{C}_l(n)$, $\lambda \in \mathcal{P}_\mu(\mu)$ and $a \in \h^*\cong(\t[1])^*$ with $\mathfrak{S}_n(a) = \mathfrak{S}_\mu$.
There is an $\mathcal{H}_{t}$-module isomorphism $$\mathsf{F}_\kappa(\mathbb{W}_{\kappa}(a,\lambda)) \cong \Delta_{t}(a,\lambda).$$ 
\end{thm}

We remark that the $a=0$ case of the preceding theorem also follows from \cite[Proposition 6.3]{VV}. 
Our third theorem shows that the Suzuki functor sends Verma modules to Verma modules. 

\begin{thm} \label{thm: Vermas to standards}
Let $m,n \in \Z_{\geq0}$ and $\lambda \in \t^*$. 
Then $\mathsf{F}_{\kappa}(\mathbb{M}_{\kappa}(\lambda))\neq 0$ if and only if $\lambda \in \mathcal{P}_n(m)$. If $\lambda \in \mathcal{P}_n(m)$ then there is an $\mathcal{H}_{t}$-module isomorphism 
\[ \mathsf{F}_{\kappa}(\mathbb{M}_{\kappa}(\lambda)) \cong \Delta_{t}(\lambda).\]  
\end{thm} 

\subsection{Partial Suzuki functors.} 

The proof of Theorems \ref{thm: regular module}-\ref{thm: Vermas to standards} requires some preparation. 
We start by recalling a few facts about induction.

\begin{lem} \label{lemma: tower of ind data}
Let $\d \subset \a$ be Lie algebras, $M$ a $\d$-module and $N$ an $\a$-module. 
\begin{enumerate}[label=\alph*), font=\textnormal,noitemsep,topsep=3pt,leftmargin=1cm]
\itemsep0em
\item \label{lemma: tower of ind data a} There exists a linear isomorphism $H_0(\a,\Ind_{\d}^{\a}M) \cong H_0(\d,M).$
\item There is an $\a$-module isomorphism
\[ \Ind_{\d}^{\a} (N \otimes M) \xrightarrow{\sim} N \otimes \Ind_{\d}^{\a} M, \quad a \otimes n \otimes m \mapsto \sum a_1 n \otimes a_2 \otimes m,\]
called the \emph{tensor identity}, where $\sum a_1 \otimes a_2$ is the coproduct of $a\in \Uu(\a)$. It restricts to the linear isomorphism
\[ \C1_{\a} \otimes (N \otimes M) \xrightarrow{\sim} N \otimes (\C1_{\a} \otimes M), \quad 1_{\a} \otimes n \otimes m \mapsto n \otimes 1_{\a} \otimes m.\] 
\end{enumerate} 
\end{lem}

\begin{proof}
The first part of the lemma follows directly from the definitions. For the proof of the second part see, e.g., \cite[Proposition 6.5]{Knapp}. 
\end{proof}

We next define ``partial Suzuki functors''. 
Let $l \geq 1$ and $\mu \in \mathcal{C}_l(n)$. 
Suppose that $M \in \Uu_1(\bar{\l}_\mu)\Lmod{}$. The diagonal $\g$-action on 
\[ \mathbf{T}(M):= (\mathbf{V}^*)^{\otimes m} \otimes M\]
restricts to an action of the Lie subalgebra $\l_\mu$. The symmetric group acts on $\mathbf{T}(M)$, as usual, by permuting the factors of the tensor product. We extend this action to an action of $\C[\h^*]^\rtimes$ by letting each $y_i$ act as the operator 
\begin{equation} \label{y-formula partial Suzuki} y_i \mapsto \sum_{1 \leq k \leq n} e_{kk}^{(i)}e_{kk}[1]^{(\infty)}.\end{equation}

\begin{lem} \label{partial Suzuki norm}
The $\l_\mu$-action and the $\C[\h^*]^\rtimes$-action on $\mathbf{T}(M)$ commute. 
\end{lem}

\begin{proof} 
The fact that the $\mathfrak{S}_m$-action commutes with the $\l_\mu$-action follows from Schur-Weyl duality. Therefore we only need to show that the operators \eqref{y-formula partial Suzuki} commute with the $\l_\mu$-action. 
Let $e_{rs} \in \l_\mu$. 
We have an equality of operators on $\mathbf{T}(M)$: 
\begin{align}
y_i \sum_{j =1,\hdots,n,\infty} e_{rs}^{(j)} =& \ \sum_{j \neq i, \infty} \sum_{k=1}^n e_{kk}^{(i)} e_{rs}^{(j)} e_{kk}[1]^{(\infty)} \\ \label{partial Suzuki comm 3} +& \  \sum_{k=1}^n e_{kk}^{(i)} e_{rs}^{(i)} e_{kk}[1]^{(\infty)} + \ \sum_{k=1}^n e_{kk}^{(i)}  e_{kk}[1]^{(\infty)} e_{rs}^{(\infty)}. 
\end{align} 
Consider the first summand in \eqref{partial Suzuki comm 3}:  
\begin{equation} \label{partial Suzuki comm} \sum_{k=1}^n e_{kk}^{(i)} e_{rs}^{(i)} e_{kk}[1]^{(\infty)} =  \sum_{k=1}^n  e_{rs}^{(i)} e_{kk}^{(i)} e_{kk}[1]^{(\infty)} + e_{rs}^{(i)}(e_{rr}[1]^{(\infty)} - e_{ss}[1]^{(\infty)}). \end{equation} 
Since $M$ is an $\bar{\l}_\mu$-module, $e_{rr}[1] - e_{ss}[1] = 0$ as operators on $M$ and the second summand on the RHS of \eqref{partial Suzuki comm} vanishes. Next consider the second summand in \eqref{partial Suzuki comm 3}: 
\begin{equation} \label{partial Suzuki comm 2} \sum_{k=1}^n e_{kk}^{(i)}  e_{kk}[1]^{(\infty)} e_{rs}^{(\infty)} = \sum_{k=1}^n e_{kk}^{(i)} e_{rs}^{(\infty)} e_{kk}[1]^{(\infty)} + (e_{rr}^{(i)} - e_{ss}^{(i)})e_{rs}[1]^{(\infty)}. \end{equation} 
If $r=s$ then the second summand on the RHS of \eqref{partial Suzuki comm 2} vanishes. If $r \neq s$ it vanishes as well since $M$ is an $\bar{\l}_\mu$-module and $e_{rs}[1]$ acts trivially on $M$. 
\end{proof} 

By Lemma \ref{partial Suzuki norm}, there is an induced $\C[\h^*]^\rtimes$-representation on $H_0(\l_\mu, \mathbf{T}(M))$ and, therefore, a functor    
\[ \overline{\mathsf{F}}^\mu \colon \Uu_1(\bar{\l}_\mu)\Lmod{} \to \C[\h^*]^\rtimes\Lmod{}, \quad M \mapsto H_0(\l_\mu, \mathbf{T}(M)), \]
which we call a \emph{partial Suzuki functor}. 
For $\mu = (1^n)$ we also write $\overline{\mathsf{F}} := \overline{\mathsf{F}}^\mu$.     
Set 
\[ \mathcal{H}\mathsf{ind}_t \colon \C[\h^*]^\rtimes\Lmod{} \to \mathcal{H}_{t}\Lmod{}, \quad N \mapsto \mathcal{H}_{t} \otimes_{\C[\h^*]^\rtimes} N.\] 
\begin{pro} \label{diagramHindInd}
The diagram
\[
\begin{tikzcd}
\mathscr{C}_\kappa \arrow[r, "\mathsf{F}_\kappa"] & \mathcal{H}_{t}\Lmod{} \\
U_1(\bar{\l}_\mu)\Lmod{} \arrow[u,"\mathsf{Ind}_{\mu,\kappa}"] \arrow[r,"\overline{\mathsf{F}}^\mu"] & \C[\h^*]^\rtimes\Lmod{} \arrow[u,"\mathcal{H}\mathsf{ind}_{t}",swap]
\end{tikzcd}
\]
commutes, i.e., there exists a natural isomorphism of functors $\mathsf{F}_\kappa \circ \mathsf{Ind}_{\mu,\kappa} \cong \mathcal{H}\mathsf{ind}_{t} \circ \overline{\mathsf{F}}^\mu$. 
Explicitly, for each $M \in \Uu_1(\bar{\l}_{\mu})\Lmod{}$, this isomorphism is given by
\begin{align} \label{skew group ring mod iso} \phi \colon \mathcal{H}\mathsf{ind}_{t}(\overline{\mathsf{F}}^\mu(M)) = \C[\h] \otimes H_0(\l_{\mu}, \mathbf{T}(M)) \xrightarrow{\sim}& \ \mathsf{F}_{\kappa}(\mathsf{Ind}_{\mu,\kappa} (M)) \\
\label{skew group ring mod iso formula} f(x_1, \hdots, x_m) \otimes [v \otimes u] \mapsto& \ [f(x_1, \hdots, x_m) \otimes v \otimes i(u)],
\end{align}
 where $v \in (\mathbf{V}^*)^{\otimes m}, u \in M$ and $i : M \hookrightarrow \Ind_{\hat{\l}^+_{\mu}}^{\hat{\g}_{\kappa}}M$ is the natural inclusion. 
\end{pro}

\begin{proof}
We first show that \eqref{skew group ring mod iso} is an isomorphism of $\C[\h]^\rtimes$-modules. Since the first equality in \eqref{skew group ring mod iso} follows directly from the PBW theorem \eqref{RCA PBW}, we only need to prove the second isomorphism. 
Consider $\Ind_{\hat{\l}_{\mu}^+}^{\hat{\g}_\kappa}M$ as a $\g[t]$-module using the Lie algebra homomorphism 
\begin{equation} \label{t-inverse} \g[t] \xrightarrow{\sim} \g[t^{-1}] \hookrightarrow \hat{\g}_\kappa, \ X[k] \mapsto X[-k].\end{equation} 
The map \eqref{t-inverse} induces 
a $\g[t]$-module isomorphism $\Ind_{\hat{\l}_{\mu}^+}^{\hat{\g}_{\kappa}} M \xrightarrow{\sim} \Ind_{\l_\mu}^{\g[t]} M.$ 
Hence, by Lemma \ref{lemma: tower of ind data}.b), we have a $\g[t]$-module isomorphism
\begin{equation} \label{levi-tensor-id} \Ind_{\l_\mu}^{\g[t]} (\C[\h] \otimes (\mathbf{V}^*)^{\otimes m} \otimes M) \xrightarrow{\sim} \C[\h] \otimes (\mathbf{V}^*)^{\otimes m} \otimes (\Ind_{\l_\mu}^{\g[t]} M),
\end{equation}
where $\g[t]$ acts on the LHS as in \eqref{currentaction}, sending 
\begin{equation} \label{C[h]Sn iso 3} 
1_{\g[t]} \otimes f(x_1,\hdots,x_m) \otimes v \otimes u \mapsto \ f(x_1,\hdots,x_m) \otimes v \otimes 1_{\g[t]} \otimes u.
\end{equation}
Next notice that, by Lemma \ref{lemma: tower of ind data}.a), we have linear isomorphisms
\begin{equation} \label{C[h]Sn iso 1} H_0(\g[t], \Ind_{\l_{\mu}}^{\g[t]}(\C[\h] \otimes \mathbf{T}(M))) \xrightarrow{\sim}  H_0(\l_{\mu}, \C[\h] \otimes \mathbf{T}(M))  
\xrightarrow{\sim} 
\C[\h] \otimes H_0(\l_{\mu},\mathbf{T}(M)).
\end{equation}
Applying $H_0(\g[t],-)$ to the inverse of \eqref{levi-tensor-id} and composing with \eqref{C[h]Sn iso 1}, we obtain an isomorphism
\begin{equation} \label{C[h]Sn iso 2} \mathsf{F}_{\kappa}(\Ind_{\hat{\l}^+_{\mu}}^{\hat{\g}_{\kappa}}M) \xrightarrow{\sim} \C[\h] \otimes H_0(\l_{\mu}, \mathbf{T}(M)). \end{equation} 
It is clear from \eqref{C[h]Sn iso 3} that \eqref{C[h]Sn iso 2} sends the equivalence class $[f(x_1, \hdots, x_m) \otimes v \otimes i(u)]$ to $f(x_1, \hdots, x_m) \otimes [v \otimes u]$. This implies, in particular, that \eqref{C[h]Sn iso 2} is $\C[\h]^\rtimes$-equivariant. 

We next prove that \eqref{skew group ring mod iso} is an isomorphism of $\mathcal{H}_{t}$-modules. 
Since $\mathcal{H}_{t}$ is generated as a $\C$-algebra by $\C[\h]^\rtimes$ and $\C[\h^*]$, it suffices to show that $\phi$ intertwines the $\C[\h^*]$-actions. Moreover, since the subspace $W:=1_{\C[\h]} \otimes H_0(\l_{\mu}, \mathbf{T}(M))$ generates $\mathcal{H}\mathsf{ind}_{t}(\overline{\mathsf{F}}^\mu(M))$ as a $\C[\h]^\rtimes$-module, it is enough to check that $\phi|_W$ intertwines the $\C[\h^*]$-actions.

Consider the subspace $U:=1_{\C[\h]} \otimes (\mathbf{V}^*)^{\otimes m} \otimes M \subset \mathsf{T}_{\kappa}(\mathsf{Ind}_{\mu,\kappa}(M))$ and its image $\overline{U}$ in $\mathsf{F}_{\kappa}(\mathsf{Ind}_{\mu,\kappa}(M))$. By \eqref{skew group ring mod iso formula}, $\phi$ restricts to a linear isomorphism $\phi|_{W} \colon W  \xrightarrow{\sim} \overline{U}$. 
The element $y_i\in \mathcal{H}_{t}$ acts on $\mathsf{F}_{\kappa}(\mathsf{Ind}_{\mu,\kappa}(M))$ as the operator ${}^\kappa\overline{y}_{i}$ (see \eqref{y-formula}). The operators $\partial_{x_i}$ and $(1-\underline{s_{i,j}})$ vanish on the subspace $\overline{U}$. Moreover, $\Omega^{(i,\infty)}_{[p+1]}$ $(p \geq 1)$ and $e_{kl}[1]^{(\infty)}$ $(k \neq l)$ act trivially on all of $\mathsf{F}_{\kappa}(\mathsf{Ind}_{\mu,\kappa}(M))$. Therefore
\begin{equation} \label{kappa-y reg module two} {}^\kappa\overline{y}_{i} = \Omega^{(i,\infty)}_{[1]} = \sum_{1 \leq k \leq n} e_{kk}^{(i)} e_{kk}[1]^{(\infty)}\end{equation}
as operators on $\overline{U}$. On the other hand, the action of $y_i$ on $W$ is given by \eqref{y-formula partial Suzuki}. 
It now follows directly from \eqref{skew group ring mod iso formula} that $\phi$ is $\C[\h^*]$-equivariant. 
\end{proof}

\subsection{Proofs of Theorems \ref{thm: regular module}-\ref{thm: Vermas to standards}} \label{subsection:regular module n=m}

We now prove the theorems from \S \ref{sec: statements comp}. 

\begin{proof}[Proof of Theorem \ref{thm: regular module}]
Combining the left $\C \mathfrak{S}_n$-module isomorphism  
\begin{equation} \label{CSn-V-iso} \C \mathfrak{S}_n \xrightarrow{\sim} ((\mathbf{V}^*)^{\otimes n})_{(-1,\hdots,-1)}, \quad w \mapsto e_{w}^* \end{equation} 
with \eqref{symT-Ch-iso} allows us to identify
\begin{equation} \label{regular module iso Fa} \overline{\Upsilon} \colon \overline{\mathsf{F}}(\mathcal{I}) \cong ((\mathbf{V}^*)^{\otimes n})_{(-1,\hdots,-1)} \otimes \mathcal{I} \cong \C\mathfrak{S}_n \ltimes \C[\h^*] \end{equation} 
as $\C\mathfrak{S}_n$-modules. We claim that \eqref{regular module iso Fa} also intertwines the $\C[\h^*]$-actions. 

Let us prove the claim.  
Consider the subspace $U:= e_{\mathsf{id}}^* \otimes \mathcal{I} \subset (\mathbf{V}^*)^{\otimes n} \otimes \mathcal{I}$ and its image $\overline{U}$ in $\overline{\mathsf{F}}(\mathcal{I})$. The map $\overline{\Upsilon}$ restricts to a linear isomorphism $\overline{\Upsilon}' \colon \overline{U} \cong \C[\h^*]$. Since $\C[\h^*]$ generates $\C\mathfrak{S}_n \ltimes \C[\h^*]$ as an $\mathfrak{S}_n$-module, it suffices to show that $\overline{\Upsilon}'$ is $\C[\h^*]$-equivariant. The action of $y_i$ on $\overline{\mathsf{F}}(\mathcal{I})$ is given by  formula \eqref{y-formula partial Suzuki}. Observe that $e_{kk}^{(i)}.e_{\mathsf{id}}^* = -\delta_{k,i} e_{\mathsf{id}}^*$  
and $e_{kk}[1]$ acts as multiplication by $e_{kk}[1]$ on $\mathcal{I}\cong \Sym(\t[1])$. Hence 
$y_i$ acts on $\overline{U}$ as multiplication by $-e_{ii}[1]$. On the other hand, $y_i$ acts on $\C[\h^*] \subset \C\mathfrak{S}_n \ltimes \C[\h^*]$ as multiplication by $y_i$. It is clear from \eqref{symT-Ch-iso} that $\overline{\Upsilon}'$ intertwines these two actions, which completes the proof of the claim. 

We now prove the theorem. 
By definition, $\mathsf{F}_\kappa(\mathbb{H}_\kappa) = \mathsf{F}_\kappa(\mathsf{Ind}_{\kappa}(\mathcal{I}))$ and, 
by Proposition \ref{diagramHindInd}, $\mathsf{F}_\kappa(\mathsf{Ind}_{\kappa}(\mathcal{I})) \cong \mathcal{H}\mathsf{ind}_{t}(\overline{\mathsf{F}}(\mathcal{I}))$. The claim above implies that $\mathcal{H}\mathsf{ind}_{t}(\overline{\mathsf{F}}(\mathcal{I})) \cong \mathcal{H}\mathsf{ind}_{t}(\C\mathfrak{S}_n \ltimes \C[\h^*]) = \mathcal{H}_{t}$. Formula \eqref{Regular module iso 2} also follows from Proposition \ref{diagramHindInd}. 
\end{proof}

\begin{proof}[Proof of Theorem \ref{thm: standard to Weyl}] 
Set $S_j(\mu) = \{ \mu_{\leq j-1}+1, \hdots, \mu_{\leq j} \}$ so that $\{1, \hdots, n\} = \bigsqcup_{j=1}^l S_j(\mu)$. Write $r \sim s$ if and only if there exists $j$ such that both $r,s \in S_j(\mu)$.  
By Proposition \ref{Schur-Weyl pro}, there is a natural $\C \mathfrak{S}_n$-module isomorphism
\begin{equation} \label{FaL IndSp iso} \overline{\Upsilon}_{\mu,a} \colon \overline{\mathsf{F}}^\mu(L(a,\lambda)) \cong  \C[\h^*]^\rtimes \otimes_{\C \mathfrak{S}_\mu \ltimes \C[\mathbf{\h^*}]}\mathsf{Sp}(a,\lambda)=:\mathsf{Sp}_\mu(a,\lambda).\end{equation} 
We claim that \eqref{FaL IndSp iso} is an isomorphism of $\C[\h^*]^\rtimes$-modules. 

It suffices to show that $\overline{\Upsilon}_{\mu,a}$ is an isomorphism of $\C[\h^*]$-modules. 
Consider the subspace $U:= (\mathbf{V}^*)_{(\mu,\mu)}^{\otimes n} \otimes L(a,\lambda) \subset (\mathbf{V}^*)^{\otimes n} \otimes L(a,\lambda)$ and its image $\overline{U}$ in $\overline{\mathsf{F}}^{\mathsf{\mu}}(L(a,\lambda))$. The map $\overline{\Upsilon}_{\mu,a}$ restricts to a $\C \mathfrak{S}_\mu$-module isomorphism $\overline{\Upsilon}_{\mu,a}' \colon \overline{U} \cong \mathsf{Sp}(a,\lambda)$. Since $\mathsf{Sp}(a,\lambda)$ generates $\mathsf{Sp}_\mu(a,\lambda)$ as an $\mathfrak{S}_n$-module, it suffices to show that $\overline{\Upsilon}_{\mu,a}'$ is $\C[\h^*]$-equivariant. The action of $y_i$ on $\overline{\mathsf{F}}^{\mu}(L(a,\lambda))$ is given by formula \eqref{y-formula partial Suzuki}. 
Let $v = v_1 \otimes \hdots \otimes v_n \in (\mathbf{V}^*)_{(\mu,\mu)}^{\otimes n}$. 
Suppose that $i \in S_j(\mu)$. 
Observe that $e_{kk}^{(i)}.v = 0$ unless $k \sim i$ and $\sum_{k \in S_j(\mu)}e_{kk}^{(i)}.v = -v$. Moreover, the elements $e_{kk}[1]$ $(k \in S_j(\mu))$ act on $L(a,\lambda)$ by the same scalar $-a_i := -a(y_i)$. 
Hence $y_i$ acts on $\overline{U}$ as multiplication by $a_i$. This agrees with the definition of the $y_i$-action on~$\mathsf{Sp}(a,\lambda)$, completing the proof of the claim. 

We now prove the theorem. 
By definition, $ \mathsf{F}_\kappa(\mathbb{W}_{\kappa}(a,\lambda)) = \mathsf{F}_\kappa(\mathsf{Ind}_{\mu,\kappa}(L(a,\lambda)))$ and, by Proposition \ref{diagramHindInd}, $\mathsf{F}_\kappa(\mathsf{Ind}_{\mu,\kappa}(L(a,\lambda))) \cong \mathcal{H}\mathsf{ind}_{t}(\overline{\mathsf{F}}^{\mu}(L(a,\lambda)))$. 
The claim above implies that $\mathcal{H}\mathsf{ind}_{t}(\overline{\mathsf{F}}^{\mu}(L(a,\lambda))) \cong \mathcal{H}\mathsf{ind}_{t}(\mathsf{Sp}_\mu(a,\lambda)) = \Delta_{t}(a,\lambda)$. 
\end{proof}

\begin{proof}[Proof of Theorem \ref{thm: Vermas to standards}] 
In analogy to Proposition \ref{diagramHindInd}, one can show that, for each $\lambda \in \t^*$, there is a $\C[\h]^\rtimes$-module isomorphism 
\begin{align} \label{lem: Borel analogue} \C[\h] \otimes H_0(\b_+, \mathbf{T}(\C_{\lambda})) \xrightarrow{\sim}& \ \mathsf{F}_\kappa(\Ind_{\hat{\b}_+}^{\hat{\g}_\kappa} \C_{\lambda,1}) = \mathsf{F}_{\kappa}(\mathbb{M}_{\kappa}(\lambda)) \\ 
\label{lem: Borel analogue 2}
f(x_1, \hdots, x_m) \otimes [v \otimes u] \mapsto& \ [f(x_1, \hdots, x_m) \otimes v \otimes i(u)],
\end{align}
where $v \in (\mathbf{V}^*)^{\otimes m}, u \in \C_\lambda$ and $i : \C_\lambda \hookrightarrow \Ind_{\hat{\b}_+}^{\hat{\g}_{\kappa}}\C_{\lambda,1}$ is the natural inclusion. 

The first statement of the theorem now follows directly from \eqref{lem: Borel analogue} and Corollary \ref{lem: SW duality form 2}. So consider the second statement. Let $\lambda \in \mathcal{P}_n(m)$. By Corollary \ref{lem: SW duality form 2} and \eqref{RCA PBW}, we can identify $\Delta_t(\lambda) \cong \C[\h] \otimes H_0(\b_+, \mathbf{T}(\C_{\lambda}))$ as $\C[\h]^\rtimes$-modules. Let $\Upsilon_\lambda$ be the composition of this isomorphism with \eqref{lem: Borel analogue}. We need to check that $\Upsilon_\lambda$ intertwines the $\C[\h^*]$-actions. Observe that, by \eqref{lem: Borel analogue 2}, $\Upsilon_\lambda$ restricts to a linear isomorphism $\mathsf{Sp}(\lambda) \to \overline{U}$, where $\overline{U}$ is the image of $U:=1_{\C[\h]} \otimes (\mathbf{V}^*)^{\otimes m} \otimes \C_{\lambda,1}$ in $\mathsf{F}_{\kappa}(\mathbb{M}_{\kappa}(\lambda))$. Since $\mathsf{Sp}(\lambda)$ generates $\Delta_{t}(\lambda)$ as a $\C[\h]^\rtimes$-module, it suffices to show that $\Upsilon_\lambda|_{\mathsf{Sp}(\lambda)}$ intertwines the $\C[\h^*]$-actions. By definition, each $y_i$ acts trivially on $\mathsf{Sp}(\lambda)$. On the other hand, since each $e_{kk}[1]$ acts trivially on $\C_{\lambda,1}$, 
the operator ${}^\kappa\bar{y}_i$ also vanishes on~$\overline{U}$. 
\end{proof}

\section{Relationship between the centres} 

Assume that $n=m$ throughout this section. The fact that the algebras $\UUc$ and $\mathcal{H}_0$ have large centres has many implications for their representation theory. For example, they have uncountably many isomorphism classes of irreducible modules, and Verma-type modules have large endomorphism and extension algebras (see \S \ref{sec: apps} for a more detailed discussion). 
To understand how simple modules or endomorphism rings behave under the Suzuki functor, we must, therefore, understand the relationship between the centres of the categories $\UUc\Lmod{}$ and $\mathcal{H}_0\Lmod{}$. In general, a functor of additive categories does not induce a homomorphism between their centres. In \S \ref{sec: centres of cats}  below we propose two ways to get around this problem. In \S \ref{App to Suzuki section} and \S \ref{sec: CH central cat}, we apply them to the Suzuki functor, and construct a map $\mathfrak{Z} \to \mathcal{Z}$ between the two centres.

\subsection{Centres of categories.} \label{sec: centres of cats} 
Suppose $F \colon \mathcal{A} \to \mathcal{B}$ is an additive functor between additive categories. Recall that the centre $Z(\mathcal{A})$ of $\mathcal{A}$ is the ring of endomorphisms of the identity functor $\id_{\mathcal{A}}$. An element of $z \in Z(\mathcal{A})$ is thus a collection of endomorphisms $\{ z_M \in \End_{\mathcal{A}}(M) \mid M \in \mathcal{A}\}$ such that $f \circ z_M = z_N \circ f$ for all $f \in \Hom_{\mathcal{A}}(M,N)$. 

The functor $F$ does not necessarily induce a ring homomorphism $Z(\mathcal{A}) \to Z(\mathcal{B})$. For example, if $F$ is not essentially surjective, then the collection $\{ F(z_M) \in \End_{\mathcal{B}}(F(M)) \mid M \in \mathcal{A}\}$ does not contain an endomorphism for every object of $\mathcal{B}$. If $F$ is not full, then the endomorphisms $F(z_M)$ may fail to commute with some of the morphisms in $\mathcal{B}$. 
Hence $\{ F(z_M) \in \End_{\mathcal{B}}(F(M)) \mid M \in \mathcal{A}\}$ is not necessarily an endomorphism of the identity functor $\id_{\mathcal{B}}$. We remark that some sufficient conditions for the existence of a canonical homomorphism $Z(\mathcal{A}) \to Z(\mathcal{B})$ are known - for instance $F$ being a Serre quotient functor (see \cite[Lemma~4.3]{Ros}). 

We therefore pursue a different approach to construct a sensible ring homomorphism $Z(\mathcal{A}) \to Z(\mathcal{B})$ encoding information about the functor $F$. 
There are canonical ring homomorphisms
\[ Z(\mathcal{A}) \overset{\alpha}{\longrightarrow} \End(F) \overset{\beta}{\longleftarrow} Z(\mathcal{B}) \]
with $\alpha$ taking $\{ z_M \mid M \in \mathcal{A} \}$ to $\{ F(z_M) \mid M \in \mathcal{A} \}$ and $\beta$ taking $\{ z_K \mid K \in \mathcal{B} \}$ to $\{ z_{F(M)} \mid M \in \mathcal{A}\}$. We assume that $\beta$ is injective, and identify $Z(\mathcal{B})$ with a subring of~$\End(F)$.

\begin{defi}
We call $Z_{F}(\mathcal{A}) := \alpha^{-1}(Z(\mathcal{B})) \subset Z(\mathcal{A})$ the $F$-\emph{centre} of $\mathcal{A}$. If $\mathcal{A} = A\Lmod{}$ is the category of modules over some algebra $A$, we will also write $ Z_F(A):=Z_F(\mathcal{A})$. 
\end{defi}

Restricting $\alpha$ to $Z_{F}(\mathcal{A})$ gives a natural algebra homomorphism from the $F$-centre of $\mathcal{A}$ to the centre of $\mathcal{B}$: 
\begin{equation} \label{first homo btw centres} Z(F) := \alpha|_{Z_{F}(\mathcal{A})} \colon \  Z_F(\mathcal{A}) \longrightarrow Z(\mathcal{B}).\end{equation} 
For any object $M \in \mathcal{A}$, the homomorphism $Z(F)$ fits into the following commutative diagram
\begin{equation} \label{centre endo diagram} 
\begin{tikzcd}
 Z_F(\mathcal{A}) \arrow{r}{Z(F)} \arrow{d}[swap]{can} & Z(\mathcal{B}) \arrow{d}{can} \\
\End_{\mathcal{A}}(M) \arrow{r}{F} & \End_{\mathcal{B}}(F(M))
\end{tikzcd}
\end{equation}
Therefore, $Z(F)$ contains partial information about all the maps between endomorphism rings induced by the functor $F$. 

In general, $Z_F(\mathcal{A}) \neq Z(\mathcal{A})$. In that case, we would like to extend $Z(F)$ to a homomorphism $Z(\mathcal{A}) \to Z(\mathcal{B})$. Of course, there is a price to pay - such a homomorphism cannot make the diagram \eqref{centre endo diagram} commute for all objects $M \in \mathcal{A}$. Instead, we impose the condition that the diagram should commute for all $M$ from some subcategory of $\mathcal{A}$. 

Given a full additive subcategory $\mathcal{A}'$, let $F' \colon \mathcal{A}' \to \mathcal{B}$ be the restricted functor. Restriction to objects in $\mathcal{A}'$ yields canonical homomorphisms $q \colon Z(\mathcal{A}) \to Z(\mathcal{A}')$ and $\End(F) \to \End(F')$. We assume that the canonical map $\beta' \colon Z(\mathcal{B}) \to \End(F')$ is injective, and identify $Z(\mathcal{B})$ with a subring of $\End(F')$. The following commutative diagram illustrates all the maps we have just defined:

\begin{equation}
\begin{tikzcd} \label{cat centres res}
Z(\mathcal{A})  
\arrow{r}{\alpha} \arrow{d}[swap]{q} & \End(F) \arrow{d} \arrow[hookleftarrow]{r}{\beta} & Z(\mathcal{B}) \arrow[equal]{d} \\ 
Z(\mathcal{A}') \arrow{r}{\alpha'} & \End(F') \arrow[hookleftarrow]{r}{\beta'} & Z(\mathcal{B})  
\end{tikzcd}
\end{equation}

\begin{defi} \label{fc central cat defi}
We say that a full subcategory $\mathcal{A}'$ of $\mathcal{A}$ is $F$-\emph{central} if $\Ima(\alpha' \circ q) \subseteq Z(\mathcal{B})$. 
\end{defi}

If $\mathcal{A}'$ is $F$-central, then there is a natural algebra homomorphism 
\[ Z_{\mathcal{A}'}(F) := \alpha' \circ q \colon \ Z(\mathcal{A}) \longrightarrow Z(\mathcal{B}) \]
extending \eqref{first homo btw centres}, and making the diagram 
\begin{equation*} 
\begin{tikzcd}
 Z(\mathcal{A}) \arrow{r}{Z_{\mathcal{A}'}(F)} \arrow{d}[swap]{can} & Z(\mathcal{B}) \arrow{d}{can} \\
\End_{\mathcal{A}}(M) \arrow{r}{F} & \End_{\mathcal{B}}(F(M))
\end{tikzcd}
\end{equation*}
commute for all $M \in \mathcal{A}'$. The homomorphism $Z_{\mathcal{A}'}(F)$ contains partial information about all the maps between endomorphism rings induced by the restricted functor $F'$. 

\subsection{The $\mathsf{F}_{\mathbf{c}}$-centre.} \label{App to Suzuki section}

For the rest of this section, we will use the canonical identifications
\[ \mathfrak{Z} \cong Z(\widehat{\mathbf{U}}_{\mathbf{c}}\Lmod{}), \quad \mathcal{Z} \cong  Z(\mathcal{H}_{0}\Lmod{}), \quad \UU^{op} \cong \End_{\UU}(\UU).\] 
Let us apply the framework developed in \S \ref{sec: centres of cats} to the functor $\mathsf{F}_{\mathbf{c}} \colon \widehat{\mathbf{U}}_{\mathbf{c}}\Lmod{} \to \mathcal{H}_{0}\Lmod{}$.
We have canonical maps 
\begin{equation*} \label{alpha beta Suzuki} \mathfrak{Z} \overset{\alpha}{\longrightarrow} \End(\mathsf{F}_{\mathbf{c}}) \overset{\beta}{\longleftarrow} \mathcal{Z}.\end{equation*}
By Theorem \ref{Regular module thm}, the regular module $\mathcal{H}_0$ is in the image of $\mathsf{F}_{\mathbf{c}}$. The fact that  $\mathcal{Z}$ acts faithfully on $\mathcal{H}_0$ implies that $\beta$ is injective. 

Our first goal is to give a partial description of the $\mathsf{F}_{\mathbf{c}}$-centre of $\widehat{\mathbf{U}}_{\mathbf{c}}\Lmod{}$. 
For any $\kappa \in \C$, define 
\begin{equation} \label{subHeisVir} \mathscr{L}_{\kappa} := \langle {}^{\kappa}\mathbf{L}_{r+1}, \id[r] \mid r \leq 0 \rangle \subset \UU. \end{equation}
When $\kappa = \mathbf{c}$, it follows from Theorem \ref{FF small theorem thm} and \S \ref{sec:va-centre of env} that the generators on the RHS of \eqref{subHeisVir} are algebraically independent. Hence 
\begin{equation} \label{SHV-} \mathscr{L}_{\mathbf{c}} = \C[{}^{\mathbf{c}}\mathbf{L}_{r+1}, \id[r]]_{r \leq 0}. \end{equation} 
We will show that $\mathscr{L}_{\mathbf{c}}$ is a subalgebra of the $\mathsf{F}_{\mathbf{c}}$-centre of $\widehat{\mathbf{U}}_{\mathbf{c}}\Lmod{}$. The proof requires some preparations. 

Let $\kappa$ be arbitrary and set $t=\kappa+n$. Let $1_{\hat{\g}}$ denote the unit in $\widehat{\mathbf{U}}_\kappa$. 
Consider the image $[1 \otimes e_{\mathsf{id}}^* \otimes 1_{\hat{\g}}]$ of $1 \otimes e_{\mathsf{id}}^* \otimes 1_{\hat{\g}} \in \mathsf{T}_{\kappa}(\UU)$ in $\mathsf{F}_{\kappa}(\UU)$. Let $K_{t}$ be the $\mathcal{H}_{t}$-submodule of $\mathsf{F}_{\kappa}(\UU)$ generated by $[1 \otimes e_{\mathsf{id}}^* \otimes 1_{\hat{\g}}]$. 

\begin{lem}
There is an $\mathcal{H}_{t}$-module isomorphism $K_t \cong \mathcal{H}_{t}$. 
\end{lem}

\begin{proof}
Since $\mathsf{F}_{\kappa}$ is right exact, it induces an epimorphism
\[ \mathsf{F}_{\kappa}(\UU) \twoheadrightarrow \mathsf{F}_{\kappa}(\mathbb{H}_{\kappa}) \cong \mathcal{H}_{t}, \quad [1 \otimes e_{\mathsf{id}}^* \otimes 1_{\hat{\g}}] \mapsto [1 \otimes e_{\mathsf{id}}^* \otimes 1_{\mathbb{H}}] = 1_{\mathcal{H}},\]
which restricts to an isomorphism $K_t \cong \mathcal{H}_{t}$. 
\end{proof}

Let $N_{t}$ be the subalgebra of $\End_{\mathcal{H}_{t}}(\mathsf{F}_{\kappa}(\UU))$ consisting of endomorphisms which preserve the submodule $K_t$. Let $\rho_t \colon N_t \to \End_{\mathcal{H}_t}(K_t) \cong \mathcal{H}_t^{op}$ be the map given by restriction of endomorphisms  of $\mathsf{F}_{\kappa}(\UU)$ to those of $K_t$. 

\begin{lem} \label{diagram factors trilemma} The following hold. 
\begin{enumerate}[label=\alph*), font=\textnormal,noitemsep,topsep=3pt,leftmargin=1cm]
\item The image of $\mathscr{L}_{\kappa}^{op}$ under $\End_{\UU}(\UU) \xrightarrow{\mathsf{F}_\kappa}  \End_{\mathcal{H}_t}(\mathsf{F}_{\kappa}(\UU))$ is contained in~$N_t$. 
\item The map $\rho_t \circ \mathsf{F}_\kappa|_{\mathscr{L}_{\kappa}^{op}}$ is given by: 
\begin{align} \label{id[r]} \id[r] \mapsto& \ \sum_{i=1}^n x_i^{-r} \quad (r \leq 0),\\
\label{L_r complete homogeneous} {}^{\kappa}\mathbf{L}_r \mapsto& \ - \frac{1}{2}\sum_{i=1}^n x_i^{1-r}y_i + \sum_{i < j} c_{-r}(x_i,x_j)s_{i,j} + \frac{n(1-r)}{2} \sum_{i=1}^n x_i^{-r} \quad (r \leq 1),\end{align}
where $c_{-r}(x_i,x_j)$ is the complete homogeneous  symmetric polynomial of degree $-r$ in $x_i$ and $x_j$, if $r\leq 0$, and $c_{-1}(x_i,x_j)=0$.   
\item \label{diagram factors trilemma c} When $\kappa = \mathbf{c}$, the image of $\mathsf{F}_{\mathbf{c}}|_{\mathscr{L}_{\mathbf{c}}}$ lies in the image of $\mathcal{Z}$ in $\End_{\mathcal{H}_0}(\mathsf{F}_{\mathbf{c}}(\UUc))$. 
\end{enumerate}
\end{lem}

\begin{proof}
A homomorphism from $K_t$ to $\mathsf{F}_{\kappa}(\UU)$ is determined by where it sends the generator $[1 \otimes e_{\mathsf{id}}^* \otimes 1_{\hat{\g}}]$. Let $z$ be any of our distinguished generators (see \eqref{subHeisVir}) of $\mathscr{L}_{\kappa}$. The corresponding endomorphism of $\mathsf{F}_{\kappa}(\UU)$ sends $[1 \otimes e_{\mathsf{id}}^* \otimes 1_{\hat{\g}}]$ to $[1 \otimes e_{\mathsf{id}}^* \otimes z\cdot1_{\hat{\g}}]$. We are going to use the $\g[t]$-action \eqref{currentaction} to show that $1 \otimes e_{\mathsf{id}}^* \otimes z\cdot1_{\hat{\g}}$ is in the same equivalence class in $\mathsf{F}_{\kappa}(\UU)$ as an element of the form \eqref{id[r]} or \eqref{L_r complete homogeneous}. 
First take $z = \id[r]$ with $r \leq 0$. By \eqref{currentaction}, we have
\[ [1 \otimes e_{\mathsf{id}}^* \otimes \id[r] \cdot 1_{\hat{\g}}] = \sum_{i=1}^n [ x_i^{-r} \otimes e_{\mathsf{id}}^* \otimes 1_{\hat{\g}}]. \]
This yields formula \eqref{id[r]}. 
Secondly, take $z = {}^{\kappa}\mathbf{L}_r$ with $r \leq 1$. By \eqref{currentaction}, we have the following equalities of operators on $\mathsf{F}_{\kappa}(\UU)$ evaluated at $[1 \otimes e_{\mathsf{id}}^* \otimes 1_{\hat{\g}}]$: 
\[ \sum_{s \geq 1}\sum_{k,l} (e_{kl}[r-s]e_{lk}[s])^{(\infty)} = - \sum_{s\geq 1} \sum_i x_i^{s-r} \sum_{k,l} e_{kl}^{(i)} e_{lk}[s]^{(\infty)} = - \sum_i x_i^{1-r} y_i, \]
\[ \sum_{r \leq s \leq 0} \sum_{k,l} (e_{kl}[s]e_{lk}[r-s])^{(\infty)} = \sum_{r \leq s \leq 0} \sum_{i,j} x_i^{-s} x_j^{s-r} \Omega^{(i,j)}= 2 \sum_{i < j} c_{-r}(x_i,x_j)s_{i,j} + n(1-r) \sum_{i=1}^n x_i^{-r},\]
yielding formula \eqref{L_r complete homogeneous}. We have thus shown that the endomorphisms in $\mathsf{F}_\kappa(\mathscr{L}_{\kappa}^{op})$ send the generator $[1 \otimes e_{\mathsf{id}}^* \otimes 1_{\hat{\g}}]$ of $K_t$ to other elements of $K_t$. Hence $\mathsf{F}_\kappa(\mathscr{L}_{\kappa}^{op}) \subseteq N_t$, proving parts a) and b) of the lemma. Part c) can be checked by a direct calculation - it suffices to compute that the elements on the RHS of \eqref{id[r]} and \eqref{L_r complete homogeneous} lie in $\mathcal{Z}$. It also follows from Theorem \ref{thm restricted centres}, which has a more conceptual proof. 
\end{proof}

\begin{thm} \label{pro: question one}
We have 
$\mathscr{L}_{\mathbf{c}} \subseteq Z_{\mathsf{F}_{\mathbf{c}}}(\widehat{\mathbf{U}}_{\mathbf{c}})$. 
Moreover, $Z(\mathsf{F}_{\mathbf{c}})|_{\mathscr{L}_{\mathbf{c}}}$ is given by formulae \eqref{id[r]} and~\eqref{L_r complete homogeneous}. 
\end{thm}

\begin{proof}
We need to check that, for any $M \in \UUc\Lmod{}$ and $z \in \mathscr{L}_{\mathbf{c}}$, the endomorphism $\mathsf{F}_{\mathbf{c}}(z_M)$ lies in the image of $\mathcal{Z}$ in $\End_{\mathcal{H}_0}(\mathsf{F}_{\mathbf{c}}(M))$. 
By Definition \ref{general definition of Suzuki}, $\mathsf{F}_{\mathbf{c}}(M) = \mathsf{F}_{\mathbf{c}}(\widehat{\mathbf{U}}_{\mathbf{c}}) \otimes_{\widehat{\mathbf{U}}_{\mathbf{c}}} M$. The corresponding endomorphism $\mathsf{F}_{\mathbf{c}}(z_M)$ of $\mathsf{F}_{\mathbf{c}}(M)$ sends $r \otimes m \mapsto r \otimes z\cdot m = r \cdot z \otimes m$, for $m \in M$ and $r \in \mathsf{F}_\kappa(\widehat{\mathbf{U}}_{\mathbf{c}})$. Hence $\mathsf{F}_{\mathbf{c}}(z_M) = \mathsf{F}_{\mathbf{c}}(z_{\widehat{\mathbf{U}}_{\mathbf{c}}}) \otimes \id$. But $\mathsf{F}_{\mathbf{c}}(z_{\widehat{\mathbf{U}}_{\mathbf{c}}})$ lies in the image of $\mathcal{Z}$ in $\End_{\mathcal{H}_0}(\mathsf{F}_{\mathbf{c}}(\UUc))$ by part c) of Lemma \ref{diagram factors trilemma}. Hence $\mathsf{F}_{\mathbf{c}}(z_M)$ lies in the image of $\mathcal{Z}$ in $\End_{\mathcal{H}_0}(\mathsf{F}_{\mathbf{c}}(M))$, proving the first statement. The second statement follows directly from part b) of Lemma \ref{diagram factors trilemma}. 
\end{proof}

\subsection{An $\mathsf{F}_{\mathbf{c}}$-central subcategory.} \label{sec: CH central cat}

The following lemma shows that the $\mathsf{F}_{\mathbf{c}}$-centre of $\widehat{\mathbf{U}}_{\mathbf{c}}\Lmod{}$ is a proper subalgebra of $\mathfrak{Z}$.  

\begin{lem}
We have $Z_{\mathsf{F}_{\mathbf{c}}}(\widehat{\mathbf{U}}_{\mathbf{c}}) \neq \mathfrak{Z}$. 
\end{lem}

\begin{proof}
Consider the element $\id[1] \in \mathfrak{Z}$. It follows from \eqref{Regular module iso 2} that $-\alpha(\id[1])_{\mathsf{F}_{\mathbf{c}}(\mathbb{H}_{\mathbf{c}})}$ is the endomorphism of $\mathsf{F}_{\mathbf{c}}(\mathbb{H}_{\mathbf{c}}) \cong \mathcal{H}_0$ given by multiplication with $y_1 + \hdots + y_n$. On the other hand, take, for example, the quotient $M$ of $\Uu_{\mathbf{c}}(\hat{\g})$ by the left ideal generated by $\hat{\g}_{\geq 3}$. One sees easily from \eqref{y-formula} that $-\alpha(\id[1])_{\mathsf{F}_{\mathbf{c}}(M)}$ does not coincide with the endomorphism of $\mathsf{F}_{\mathbf{c}}(M)$ induced by $y_1 + \hdots + y_n$. 
\end{proof}

Our next goal is to find a reasonable $\mathsf{F}_{\mathbf{c}}$-central subcategory of $\widehat{\mathbf{U}}_{\mathbf{c}}\Lmod{}$. 

\begin{defi}
Let $\mathscr{C}_{\mathbb{H}}$ be the full subcategory of $\UUc\Lmod{}$ containing precisely the quotients of direct sums of $\mathbb{H}_{\mathbf{c}}$. Let $\mathsf{F}_{\mathbb{H}}$ be the restriction of $\mathsf{F}_{\mathbf{c}}$ to $\mathscr{C}_{\mathbb{H}}$. 
\end{defi}

As the lemma below shows, category $\mathscr{C}_{\mathbb{H}}$ contains interesting objects such as Verma and Weyl modules. 

\begin{lem} \label{CH Weyls in} The following hold. 
\begin{enumerate}[label=\alph*), font=\textnormal,noitemsep,topsep=3pt,leftmargin=1cm]
\item If $\lambda \in \mathcal{P}_n(n)$, then the Verma module $\mathbb{M}_{\mathbf{c}}(\lambda)$ is an object of $\mathscr{C}_{\mathbb{H}}$. 
\item Let $l \geq 1$, $\mu \in \mathcal{C}_l(n)$, $\lambda \in \mathcal{P}_\mu(\mu)$ and $a \in \C^n$ with $\mathfrak{S}_n(a) = \mathfrak{S}_\mu$. Then the Weyl module $\mathbb{W}_{\mathbf{c}}(a,\lambda)$ is an object of $\mathscr{C}_{\mathbb{H}}$. 
\end{enumerate} 
\end{lem} 

\begin{proof}
Let us prove b). The definition of $\mathbb{H}_{\mathbf{c}}$ implies that 
\[ \Hom_{\UUc}(\mathbb{H}_{\mathbf{c}},\mathbb{W}_{\mathbf{c}}(a,\lambda)) \cong \mathbb{W}_{\mathbf{c}}(a,\lambda)_{(1,\hdots,1)}^{\mathfrak{i}},\] 
where $\mathfrak{i} = \n_-[1] \oplus \n_+[1] \oplus \hat{\g}_{\geq 2}$. 
The subspace $L(a,\lambda) \subset \mathbb{W}_{\mathbf{c}}(a,\lambda)$
is annihilated by $\mathfrak{i}$. It is easy to check that, since $\lambda \in \mathcal{P}_\mu(\mu)$, the difference $\lambda - (1,\hdots,1)$ is a sum of positive roots of $\mathfrak{l}_\mu$. Since $(1,\hdots,1)$ is a dominant weight, it follows that $L(a,\lambda)_{(1,\hdots,1)} \neq \{0\}$. Since $L(a,\lambda)$ is simple as an $\l_{\mu}$-module, any non-zero vector generates $\mathbb{W}_{\mathbf{c}}(a,\lambda)$ as a $\UUc$-module. It follows that there exists an epimorphism $\mathbb{H}_{\mathbf{c}} \twoheadrightarrow \mathbb{W}_{\mathbf{c}}(a,\lambda)$. Hence $\mathbb{W}_{\mathbf{c}}(a,\lambda) \in \mathscr{C}_{\mathbb{H}}$. The proof of a) is analogous.  
\end{proof}

To state the next theorem, we need to introduce some notation: 
\begin{alignat*}{6} \Phi \colon \mathfrak{Z} &\to&&  \End_{\widehat{\mathbf{U}}_{\mathbf{c}}}(\mathbb{H}_{\mathbf{c}}), \quad& z &\mapsto&& \ z_{\mathbb{H}_{\mathbf{c}}}, \\
\Psi \colon  \End_{\widehat{\mathbf{U}}_{\mathbf{c}}}(\mathbb{H}_{\mathbf{c}}) &\to&& \End_{\mathcal{H}_0}(\mathcal{H}_0) \cong \mathcal{H}_0^{op}, \quad& \phi &\mapsto&& \ \mathsf{F}_{\mathbf{c}}(\phi),\end{alignat*}
\[ \Theta := \Psi \circ \Phi.\] 
These maps fit into the following commutative diagram. 
\begin{equation} \label{diagram gamma to theta}
\begin{tikzcd}
\mathfrak{Z} \arrow{r}{q} \arrow[equal]{d} & Z(\mathscr{C}_{\mathbb{H}}) \arrow{r}{\alpha'} \arrow{d} & \End(\mathsf{F}_{\mathbb{H}}) \arrow{d} \arrow[hookleftarrow]{r} & \mathcal{Z} \arrow[equal]{d}  \\ 
\mathfrak{Z} \arrow{r}{\Phi} & \End_{\widehat{\mathbf{U}}_{\mathbf{c}}}(\mathbb{H}_{\mathbf{c}}) \arrow{r}{\Psi} & \mathcal{H}_0^{op} \arrow[hookleftarrow]{r} & \mathcal{Z}
\end{tikzcd}
\end{equation}
where the vertical arrows send an endomorphism of the identity functor (resp.\ $\mathsf{F}_{\mathbb{H}}$) to the corresponding endomorphism of $\mathbb{H}_{\mathbf{c}}$ (resp.\ $\mathcal{H}_0$). 

The following theorem is the main result of this section. 

\begin{thm} \label{thm restricted centres}
The subcategory $\mathscr{C}_{\mathbb{H}}$ is $\mathsf{F}_{\mathbf{c}}$-central and $Z_{\mathscr{C}_{\mathbb{H}}}(\mathsf{F}_{\mathbf{c}}) = \Theta$. 
\end{thm} 

The proof of Theorem \ref{thm restricted centres} will be presented in \S \ref{proof of T restr centr}.  
We note the following corollary, which will be useful later. 

\begin{cor} \label{cor: endos vs functor}
Let $M \in \mathscr{C}_{\mathbb{H}}$ and $z \in \mathfrak{Z}$. Then $\Theta(z)_{\mathsf{F}_{\mathbf{c}}(M)} = \mathsf{F}_{\mathbf{c}}(z_M)$. In particular, $\Theta(\Ann_{\mathfrak{Z}}(M)) \subseteq \Ann_{\mathcal{Z}}(\mathsf{F}_{\mathbf{c}}(M))$. 
\end{cor}

\begin{proof}
By Theorem \ref{thm restricted centres}, we have $\mathsf{F}_{\mathbf{c}}(z_M) = (\alpha' \circ q(z))_M = \Theta(z)_M$. If $z \in \Ann_{\mathfrak{Z}}(M)$, then $z_M = 0$ and so $\Theta(z)_{\mathsf{F}_{\mathbf{c}}(M)} = \mathsf{F}_{\mathbf{c}}(z_M) = 0$. 
\end{proof} 

\subsection{Proof of Theorem \ref{thm restricted centres}.} \label{proof of T restr centr}

The proof of Theorem \ref{thm restricted centres} requires some preperations. We first prove the following lemma. 

\begin{lem} \label{lem centre to endo inj}
The two vertical arrows in \eqref{diagram gamma to theta} are injective. 
\end{lem}

\begin{proof}
Let $M$ be an object of $\mathscr{C}_{\mathbb{H}}$. Since $M$ is a quotient of $\mathbb{H}_{\mathbf{c}}^{I}$ (direct sum over some index set $I$), there exists an epimorphism $p \colon \mathbb{H}_{\mathbf{c}}^{I} \twoheadrightarrow M$. Suppose that $z \in Z(\mathscr{C}_{\mathbb{H}})$. Then $z_M \circ p = p \circ z_{\mathbb{H}_{\mathbf{c}}^{I}}$ and it follows that $z_M$ is uniquely determined by $z_{\mathbb{H}_{\mathbf{c}}^{I}}$. But $z_{\mathbb{H}_{\mathbf{c}}^{I}} = \oplus_{I} z_{\mathbb{H}_{\mathbf{c}}}$, so  $z_M$ is in fact uniquely determined by $z_{\mathbb{H}_{\mathbf{c}}}$. This proves the injectivity of the left vertical arrow.

Now suppose that $\phi \in \End(\mathsf{F}_{\mathbb{H}})$. Let $\phi_M$ be the corresponding endomorphism of $\mathsf{F}_{\mathbf{c}}(M)$. Since $\mathsf{F}_{\mathbf{c}}$ is right exact, $\mathsf{F}_{\mathbf{c}}(p) \colon \mathcal{H}_0^{I} \to \mathsf{F}_{\mathbf{c}}(M)$ is also an epimorphism. Since $\phi$ is a natural transformation, we have $\mathsf{F}_{\mathbf{c}}(p) \circ \phi_{\mathbb{H}_{\mathbf{c}}^I} = \phi_M \circ \mathsf{F}_{\mathbf{c}}(p)$. It follows that $\phi_M$ is determined uniquely by $\phi_{\mathbb{H}_{\mathbf{c}}^I}$. But $\phi_{\mathbb{H}_{\mathbf{c}}^I} = \oplus_I \phi_{\mathbb{H}_{\mathbf{c}}}$, so $\phi_M$ is uniquely determined by $\phi_{\mathbb{H}_{\mathbf{c}}}$. This proves the injectivity of the right vertical arrow.
\end{proof} 

Theorem \ref{thm restricted centres} states that $\mathscr{C}_{\mathbb{H}}$ is $\mathsf{F}_{\mathbf{c}}$-central, i.e., $\Ima \alpha' \circ q \subseteq \mathcal{Z}$. By Lemma \ref{lem centre to endo inj}, this is equivalent to showing that $\Ima \Theta \subseteq \mathcal{Z}$. The rest of this subsection is dedicated to this goal. The main idea is to establish the following two facts: $\Ima \Theta \subseteq Z_{\mathcal{H}_{0}^{op}}(\C[\h^*]^\rtimes)$ and $Z_{\mathcal{H}_{0}^{op}}(\C[\h^*]^\rtimes) = \mathcal{Z}$.

We start by recalling some information about the $G((t))$-action on $\UUc$. 
There is an adjoint action
\[ G((t)) \times \g((t)) \to \g((t)), \quad (g,X) \mapsto g(X) := gXg^{-1}\] 
of $G((t))$ on its Lie algebra $\g((t))$. It extends to an action on $\hat{\g}_{\mathbf{c}}$ if we let $G((t))$ act trivially on $\mathbf{1}$. This action induces an action on the universal enveloping algebra $\mathbf{U}_{\mathbf{c}}(\hat{\g})$ and its completion $\UUc$. 

\begin{pro}[{\cite[Proposition 4.3.8]{Fre}}] \label{pro: centre fixed under G((t))}
The $G((t))$-action on $\mathfrak{Z} \subseteq \UUc$ is trivial. 
\end{pro}

The $G((t))$-action restricts to an $\mathfrak{S}_n$-action on $\UUc$, where
we identify the symmetric group $\mathfrak{S}_n$ with the subgroup of permutation matrices in $G \subset G((t))$. 
The $\mathfrak{S}_n$-action preserves the ideal $\mathfrak{I}_{\mathbf{c}} \subset \mathbf{U}(\hat{\g}_{\mathbf{c}})$ and, hence, induces an action on the module $\mathbb{H}_{\mathbf{c}}$. 

We now define an induced action on $\mathsf{F}_{\mathbf{c}}(\mathbb{H}_{\mathbf{c}})$. 
Let $\mathfrak{S}_n$ act on $(\mathbf{V}^*)^{\otimes n}$ by the rule $e_{i_1}^* \otimes \hdots \otimes e_{i_n}^* \mapsto e_{w(i_{w^{-1}(1)})}^* \otimes \hdots \otimes e_{w(i_{w^{-1}(n)})}^*$. One easily checks that $w \cdot e_{\tau}^*=e_{w \tau w^{-1}}^*$, where $e_{\tau}^*$ is as in \eqref{CSn-V-iso}. 
Combining the $\mathfrak{S}_n$-actions on $\mathbb{H}_{\mathbf{c}}$ and $(\mathbf{V}^*)^{\otimes n}$ defined above with the natural permutation action on $\C[\h]$ we obtain an action 
\begin{equation} \label{Sn action T} \mathfrak{S}_n \times \mathsf{T}_{\mathbf{c}}(\mathbb{H}_{\mathbf{c}}) \to \mathsf{T}_{\mathbf{c}}(\mathbb{H}_{\mathbf{c}}), \quad (w, f \otimes u \otimes h) \mapsto w\cdot f \otimes w \cdot u \otimes w \cdot h. 
\end{equation}
It is easy to check that if $X[k] \in \g[t]$ and $w \in \mathfrak{S}_n$ then $w \circ X[k] = w(X)[k] \circ w$ as operators on $\mathsf{T}_{\mathbf{c}}(\mathbb{H}_{\mathbf{c}})$. Hence the subspace $\g[t].\mathsf{T}_{\mathbf{c}}(\mathbb{H}_{\mathbf{c}})$ is $\mathfrak{S}_n$-stable, and \eqref{Sn action T} descends to an action 
\begin{equation} \label{Sn action F} \star \colon \mathfrak{S}_n \times \mathsf{F}_{\mathbf{c}}(\mathbb{H}_{\mathbf{c}}) \to \mathsf{F}_{\mathbf{c}}(\mathbb{H}_{\mathbf{c}}).
\end{equation} 
Note that this action is different from the $\mathfrak{S}_n$-action defined in \S \ref{subsec: current Lie}. 

There is also a natural conjugation action
\begin{equation} \label{Sn action H} \mathfrak{S}_n \times \mathcal{H}_{0} \to \mathcal{H}_{0}, \quad (w,h) \mapsto w h w^{-1}.
\end{equation} 
In the next lemma we compare the induced actions on endomorphism algebras. 


\begin{lem} \label{cor: Z in centr Sn} 
The map $\Theta$ is $\mathfrak{S}_n$-equivariant. 
\end{lem}

\begin{proof} 
We factor $\Theta$ as a product of 
the maps $\Phi$, $\End_{\widehat{\mathbf{U}}_{\mathbf{c}}}(\mathbb{H}_{\mathbf{c}}) \to \End_{\mathcal{H}_0}(\mathsf{F}_{\mathbf{c}}(\mathbb{H}_{\mathbf{c}}))$ and the isomorphism $\End_{\mathcal{H}_0}(\mathsf{F}_{\mathbf{c}}(\mathbb{H}_{\mathbf{c}})) \cong \End_{\mathcal{H}_0}(\mathcal{H}_0)$ induced by $\Upsilon$ from \eqref{Regular module iso}. 
The first two maps are $\mathfrak{S}_n$-equivariant by construction. So we only need to check that $\Upsilon$ intertwines the two actions \eqref{Sn action F} and \eqref{Sn action H}. Abbreviating $e_{k} := e_{kk}$, 
we have 
\begin{align*}
\Upsilon(w f(x_1,\hdots,x_n)&u g(y_1,\hdots,y_n) w^{-1}) = \Upsilon(f(x_{w(1)},\hdots,x_{w(n)})wu w^{-1}g(y_{w(1)},\hdots,y_{w(n)})) \\
=& \ [f(x_{w(1)},\hdots,x_{w(n)}) \otimes e_{w u w^{-1}}^* \otimes g(-e_{w(1)}[1],\hdots,-e_{w(n)}[1]).1_{\mathbb{H}}]
\end{align*}
 and 
\begin{align*}
w \star \Upsilon(f(x_1,\hdots,x_n&)u g(y_1,\hdots,y_n) ) = \ w \star [f(x_1,\hdots,x_n)\otimes e_{u}^* \otimes g(-e_{1}[1],\hdots,-e_{n}[1]).1_{\mathbb{H}}] \\ 
=& \ [f(x_{w(1)},\hdots,x_{w(n)}) \otimes e_{w u w^{-1}}^* \otimes g(-e_{w(1)}[1],\hdots,-e_{w(n)}[1]).1_{\mathbb{H}}],
\end{align*}
as required. 
\end{proof} 

\begin{pro} \label{pro: centralizer of skgr} 
We have $Z_{\mathcal{H}_{0}}(\C[\h^*]^\rtimes) = \mathcal{Z}$. 
\end{pro} 
\begin{proof}
Write $\mathcal{H}_{\mathsf{reg}} := \C[\h_{\mathsf{reg}} \times \h^*] \rtimes \C \mathfrak{S}_n$. We first prove that $$Z_{\mathcal{H}_{\mathsf{reg}}}(\C[\h^*] \rtimes \C \mathfrak{S}_n) = Z(\mathcal{H}_{\mathsf{reg}}) = \C[\h_{\mathsf{reg}} \times \h^*]^{\mathfrak{S}_n}.$$ We only need to show that $Z_{\mathcal{H}_{\mathsf{reg}}}(\C[\h^*] \rtimes \C \mathfrak{S}_n) \subseteq Z(\mathcal{H}_{\mathsf{reg}})$, the other inclusion being obvious. Let $z \in Z_{\mathcal{H}_{\mathsf{reg}}}(\C[\h^*] \rtimes \C \mathfrak{S}_n)$. We can uniquely write $z = \sum_{w \in \mathfrak{S}_n} f_w w$ with $f_w \in \C[\h_{\mathsf{reg}} \times \h^*]$. Since, by assumption, $z$ commutes with $\C \mathfrak{S}_n$, for any $u \in \mathfrak{S}_n$ we have 
$z = u z u^{-1} = \sum_{w \in \mathfrak{S}_n} u f_w w u^{-1} = \sum_{w \in \mathfrak{S}_n} f_{u^{-1}wu}^{u} w$, where $f^u(a) = f(u^{-1} \cdot a)$. Hence $f_1 = f_1^{u}$ for all $u \in \mathfrak{S}_n$, i.e., $f_1 \in \C[\h_{\mathsf{reg}} \times \h^*]^{\mathfrak{S}_n}$. Next, since $z$ commutes with $\C[\h^*]$, 
$0 = [z,g] = \sum_{w \in \mathfrak{S}_n} f_w (g^{w} - g)w$ for all $g \in \C[\h^*]$. 
But $\mathfrak{S}_n$ acts faithfully on $\h \subset \C[\h^*]$, so for each $w \in \mathfrak{S}_n$ there exists $a \in \h$ such that $w^{-1}(a) \neq a$. This forces $f_w = 0$ for each $w \neq 1$. 

Using the Dunkl embedding (see \eqref{Dunkl embedding}), we view $\mathcal{H}_{0}$ as a subalgebra of $\mathcal{H}_{\mathsf{reg}}$. The following are obvious:
\[ Z_{\mathcal{H}_{0}}(\C[\h^*] \rtimes \C \mathfrak{S}_n) = Z_{\mathcal{H}_{\mathsf{reg}}}(\C[\h^*] \rtimes \C \mathfrak{S}_n) \cap \mathcal{H}_{0}, \quad Z(\mathcal{H}_{\mathsf{reg}}) \cap \mathcal{H}_{0} \subseteq \mathcal{Z}.\] 
Since $\mathcal{H}_{\mathsf{reg}} = \mathcal{H}_{0}[\delta^{-1}]$ and $\delta^{-1}$ is central in $\mathcal{H}_{\mathsf{reg}}$, we also have $\mathcal{Z} \subseteq Z(\mathcal{H}_{\mathsf{reg}}) \cap \mathcal{H}_{0}$. 
\end{proof}

\begin{rem}
Proposition \ref{pro: centralizer of skgr} generalizes to rational Cherednik algebras at $t=0$ associated to any complex reflection group. 
\end{rem}

\begin{pro} \label{centralizer pro}
We have $\Ima \Theta \subseteq \mathcal{Z}$. 
\end{pro}

\begin{proof}  
Lemma \ref{cor: Z in centr Sn} and Proposition \ref{pro: centre fixed under G((t))} imply that $\Ima \Theta \subseteq Z_{\mathcal{H}_{0}^{op}}(\C \mathfrak{S}_n)$. Therefore, it suffices to show that $\Ima \Theta \subseteq Z_{\mathcal{H}_{0}^{op}}(\C[\h^*])$, because then Proposition \ref{pro: centralizer of skgr} implies that $\Ima \Theta \subseteq Z_{\mathcal{H}_{0}^{op}}(\C[\h^*]^\rtimes) = \mathcal{Z}$.

By the definition of $\mathbb{H}_{\mathbf{c}}$, 
there is a natural isomorphism
\begin{equation} \label{EndHrho} \End_{\UUc}(\mathbb{H}_{\mathbf{c}}) \xrightarrow{\sim} (\mathbb{H}_{\mathbf{c}})_{(1,\hdots,1)}^{\mathfrak{i}}.\end{equation}
Observe that $\Sym(\t[1]).1_{\mathbb{H}} \subset (\mathbb{H}_{\mathbf{c}})_{(1,\hdots,1)}^{\mathfrak{i}}$. Indeed, $\Sym(\t[1]).1_{\mathbb{H}}$ has $\t$-weight $(1,\hdots,1)$, and since $\mathfrak{i}$ is an ideal in $\hat{\t}_+$ and $1_{\mathbb{H}}$ is annihilated by $\mathfrak{i}$, so is $\Sym(\t[1]).1_{\mathbb{H}}$. Hence elements of $\Sym(\t[1]).1_{\mathbb{H}}$ define endomorphisms of $\mathbb{H}_{\mathbf{c}}$. 

By construction, $\Ima \Phi \subseteq Z(\End_{\UUc}(\mathbb{H}_{\mathbf{c}}))$, and so $\Ima \Phi$ commutes with the endomorphisms defined by $\Sym(\t[1]).1_{\mathbb{H}}$. Hence $\Ima \Theta = \Psi(\Ima \Phi)$ must commute with the image of these endomorphisms under $\Psi$. But Theorem \ref{Regular module thm} implies that they are mapped to $\C[\h^*] \subset \mathcal{H}^{op}$. 
It follows that $\Ima \Theta \subseteq Z_{\mathcal{H}_{0}^{op}}(\C[\h^*])$, as required. 
\end{proof}

We are now ready to complete the proof of Theorem \ref{thm restricted centres}. 

\begin{proof}[Proof of Theorem \ref{thm restricted centres}]
By Proposition \ref{centralizer pro}, $\Ima \Theta \subseteq \mathcal{Z}$. Lemma \ref{lem centre to endo inj} and the commutativity of diagram \eqref{diagram gamma to theta}, therefore, imply $\Ima (\alpha' \circ q) \subseteq \mathcal{Z}$. The second statement of the theorem also follows directly from the commutativity of the diagram. 
\end{proof}

\section{Filtered and graded versions of the Suzuki functor}

Our next goal is to show that $\Ima \Theta = \mathcal{Z}$. The proof in \S \ref{sec: Theta surj} relies on a filtered version of the Suzuki functor, which we construct in this section. We also introduce a graded version. Assume that $\kappa \in \C$ and $m,n$ are arbitrary unless indicated otherwise.

\subsection{Background from filtered and graded algebra.}

We refer the reader to \cite{vBergh} and \cite{Sjo} for basic definitions from filtered and graded algebra. 
All filtrations we consider are increasing, exhaustive and separated. 
If $M$ is a graded vector space (or module or algebra) we denote the $i$-th graded piece by $M_i$. If $M$ is a filtered vector space (or module or algebra), we denote the $i$-th filtered piece by $M_{\leq i}$. 

Now suppose that $A$ is a filtered algebra and $M,N$ are two filtered $A$-modules. An $A$-module homomorphism $f: M \to N$ is called \emph{filtered} of degree $i$ if $f(M_{\leq r}) \subseteq N_{\leq r+i}$ for all $r \in \Z$. We say that $f$ is a \emph{filtered isomorphism} if $f$ is an isomorphism of $A$-modules and $f(M_{\leq r}) = N_{\leq r}$  for all $r \in \Z$. Let $\Hom_A(M,N)_{\leq i}$ denote the vector space of filtered homomorphisms of degree $i$ and set $\Hom_A^{\mathsf{fil}}(M,N) := \bigcup_{i \in \Z} \Hom_A(M,N)_{\leq i}$. If $M$ is finitely generated as an $A$-module then $\Hom_A(M,N) = \Hom_A^{\mathsf{fil}}(M,N)$. 
Observe that $\Hom_A^{\mathsf{fil}}(M,N)$ is a filtered vector space and $\Hom_A^{\mathsf{fil}}(M,M)$ is also a filtered algebra. 

We next define two categories whose objects are filtered (left) $A$-modules. The first category, denoted $A\Fmod{}$, has Hom-sets of the form $\Hom_A^{\mathsf{fil}}(M,N)$. The second category, denoted  $A\Fmod{}_0$, has Hom-sets of the form $\Hom_A(M,N)_0$. We regard $A\Fmod{}$ as a category enriched in the category $\C\Fmod{}_0$ of filtered vector spaces (where $\C$ is endowed with the trivial filtration). 

Analogous definitions make sense in the graded setting. In particular, if $A$ is a $\Z$-graded algebra then we have two categories of graded modules $A\Gmod{}$ and $A\Gmod{}_0$. We regard $A\Gmod{}$ as a category enriched in the category $\C\Gmod{}_0$ of graded vector spaces. 

If $A$ is a filtered algebra, with associated graded $\gr A$, let $\sigma \colon A \to \gr A$ be the principal symbol map. For $v \in A$, set $\deg v := \deg \sigma(v)$. If $f \colon A \to B$ is a degree zero filtered algebra homomorphism, let $\gr f \colon \gr A \to \gr B$ be the associated graded algebra homomorphism. 

\subsection{Filtrations and gradings.} \label{subsec: U filtr} We consider two  filtrations and a grading on~$\mathbf{U}_{\kappa}(\tilde{\g})$.

\begin{defi} \label{defi U filtr}
Suppose that $l \geq 0$, $X_1, \hdots, X_l \in \g$ and $j_1, \hdots, j_l \in \Z$. An expression of the form $\mathbf{m} = X_1[j_1]\hdots X_l[j_l]\in \mathbf{U}_{\kappa}(\tilde{\g})$ 
is called a \emph{monomial} of \emph{length} $l$, \emph{height} $j_1 + \hdots + j_l$ and  \emph{absolute height} $|j_1| + \hdots + |j_l|$. 
For $r \in \Z$, define: 
\begin{enumerate}[label=(\alph*),topsep=4pt]
\itemsep0em 
\item $\mathbf{U}_{\kappa}(\tilde{\g})_r$ \ \ \ \ $=$  \ $\langle \mbox{ monomials of height } r \ \rangle$,
\item $\mathbf{U}_{\kappa}^{\mathsf{pbw}}(\tilde{\g})_{\leq r}$ $=$ \ $\langle \mbox{ monomials of length} \leq r \ \rangle$, 
\item $\mathbf{U}_{\kappa}^{\mathsf{abs}}(\tilde{\g})_{\leq r}$ \ $=$ \ $\langle \mbox{ monomials of absolute height} \leq r \ \rangle$,
\end{enumerate}
where the brackets denote $\C$-span. 
Observe that (a) defines a grading while (b) and (c) define filtrations on $\mathbf{U}_{\kappa}(\tilde{\g})$. Filtration (b) is the usual PBW filtration. We call filtration~(c) the \emph{absolute height filtration}. Denote by  
$\mathbf{U}_{\kappa}^{\mathsf{pbw}}(\tilde{\g})$ and $\mathbf{U}_{\kappa}^{\mathsf{abs}}(\tilde{\g})$ the corresponding filtered algebras. 
\end{defi}
\begin{defi} \label{filtration Cabsr definition} We define subcategories of graded and filtered smooth modules. 
\begin{enumerate}[label=\alph*), font=\textnormal,noitemsep,topsep=3pt,leftmargin=1cm]
\item Let $\mathscr{C}_{\kappa}^{\mathsf{gr}}$ be the full subcategory of $\mathbf{U}_{\kappa}(\tilde{\g})\Gmod{}$ whose objects are graded modules with the property that the underlying ungraded module is an object of $\mathscr{C}_{\kappa}$. 
\item For $r \geq 0$, let $\mathscr{C}^{\mathsf{abs}}_{\kappa}(r)$ be the full subcategory of $\mathbf{U}_{\kappa}^{\mathsf{abs}}(\tilde{\g})\Fmod{}$ whose objects are filtered modules $M$ such that (i)  the underlying unfiltered module is an object of $\mathscr{C}_{\kappa}(r)$, and (ii) for each $l \geq 0$, we have: $M_{\leq l} = \mathbf{U}_{\kappa}^{\mathsf{abs}}(\tilde{\g})_{\leq l} \cdot M^{I_r}$. 
\end{enumerate}
\end{defi}
 
\begin{rem}
Consider the associated graded algebra $\gr\mathbf{U}_{\kappa}^{\mathsf{abs}}(\tilde{\g})$.  
It is easy to see that the relation 
\begin{equation*} \label{degenerate rel} [\sigma(X\otimes t^r), \sigma(Y \otimes t^l)] = \delta_{|r|+|l|, |r+l|} \sigma([X,Y] \otimes t^{r+l}) \end{equation*} 
holds in $\gr\mathbf{U}_{\kappa}^{\mathsf{abs}}(\tilde{\g})$. Hence 
\[ \gr \mathbf{U}^{\mathsf{abs}}(\tilde{\g}_{\geq 0}) \cong \mathbf{U}^{\mathsf{abs}}(\tilde{\g}_{\geq 0}), \quad \gr \mathbf{U}^{\mathsf{abs}}(\tilde{\g}_{\leq 0}) \cong \mathbf{U}^{\mathsf{abs}}(\tilde{\g}_{\leq 0}).\] 
Moreover, we have $[\gr \mathbf{U}^{\mathsf{abs}}(\tilde{\g}_{\geq 1}), \gr \mathbf{U}^{\mathsf{abs}}(\tilde{\g}_{\leq -1})]=0$.
\end{rem} 

\begin{defi} \label{defi: RCA filt}
We consider the following grading and family of filtrations on the rational Cherednik algebra. 
\begin{enumerate}[label=\alph*), font=\textnormal,noitemsep,topsep=3pt,leftmargin=1cm]
\item 
Setting $\deg x_i =-1$, $\deg y_i =1$ and $\deg \mathfrak{S}_m = 0$ defines a grading on~$\mathcal{H}_t$. We denote the  corresponding graded algebra simply by $\mathcal{H}_t$. 
\item For each $k \geq 1$, setting $\deg x_i = 1$, $\deg y_i =k$ and $\deg \mathfrak{S}_m = 0$ yields a filtration on $\mathcal{H}_t$, and we denote the corresponding filtered algebra by $\mathcal{H}_{t}^{(k)}$. When $k=1$, the resulting filtration is known as the PBW filtration, and we abbreviate $\mathcal{H}_{t}:=\mathcal{H}_{t}^{(1)}$.
\end{enumerate}
We consider $\C[\h]$, $\C[\h]^\rtimes$ and $\C[\h^*]$ as graded (resp.\ filtered) subalgebras of $\mathcal{H}_{t}$.
\end{defi}

\subsection{Filtered lift of the Suzuki functor} 
\label{subsec: filt lift}

Let $M$ be a filtered module in $\mathscr{C}_{\kappa}^{\mathsf{abs}}(r)$. 
We equip $(\mathbf{V}^*)^{\otimes m}$ with the trivial filtration and 
$\mathsf{T}_{\kappa}(M)$ with the tensor product filtration. Explicitly,
\begin{equation} \label{tensor filtration} \mathsf{T}_{\kappa}(M)_{\leq p} = \sum_{k+l=p} \C[\h]_{\leq k} \otimes (\mathbf{V}^*)^{\otimes m} \otimes M_{\leq l}.\end{equation} 
Consider the quotient map 
\begin{equation} \label{tensor to coinv projection} \psi \colon \mathsf{T}_{\kappa}(M) \twoheadrightarrow \mathsf{F}_\kappa(M).
\end{equation}
We endow $\mathsf{F}_\kappa(M)$ with the quotient filtration given by $\mathsf{F}_\kappa(M)_{\leq p} := \psi (\mathsf{T}_{\kappa}(M)_{\leq p})$. 
The following proposition connects the absolute height filtration on $\mathbf{U}_{\kappa}(\tilde{\g})$ with the  filtrations on $\mathcal{H}_{\kappa+n}$. 

\begin{pro} \label{pro: coinv ind filt} \label{pro: filt functor}
For each $r\geq2$, the functor $\mathsf{F}_\kappa$ lifts to a functor 
\begin{equation*} \label{filtered lift}
\mathsf{F}_{\kappa}^{(r)} \colon \mathscr{C}^{\mathsf{abs}}_{\kappa}(r) \to \mathcal{H}_{\kappa+n}^{(2r-3)}\Fmod{}
\end{equation*} 
enriched in $\C\Fmod{}_0$. 
\end{pro}

\begin{proof} 
Let $M \in \mathscr{C}^{\mathsf{abs}}_{\kappa}(r)$. 
We first show that $\mathsf{F}_\kappa(M)$ is a filtered $\mathcal{H}_{\kappa+n}^{(2r-3)}$-module. 
The only non-trivial thing to show is that $y_i\mathsf{F}_{\kappa}(M)_{\leq s} \subseteq \mathsf{F}_{\kappa}(M)_{\leq s+2r-3}$ for $s \in \Z$ and $1 \leq i \leq m$. Recall that the action of $y_i$ is given by \eqref{y-formula}. Clearly each of $\partial_{x_i}$ and $\Omega^{(i,j)}(x_i - x_j)^{-1}(1-\underline{s_{i,j}})$ either vanishes or lowers degree by one. Hence it is enough to show that for each $p \geq 0$, the operator $x_i^p \Omega_{[p+1]}^{(i,\infty)}$ raises degree by at most $2r-3$. Observe that $x_i^p$ raises degree by $p$ and $e_{kl}^{(i)}$ doesn't change degree. 
Therefore it is in fact enough to show that each $e_{lk}[p+1]^{(\infty)}$ changes degree by at most $-p+2r-3$.

If $p\leq r-2$ then the fact that $M$ is a filtered module implies that $e_{lk}[p+1]$ raises degree by at most $r-1$. But $r-1 \leq (r-1) + (r-2-p) = -p+2r-3$. So assume $p > r-2$. 
Let $v \in M$. Because $M \in \mathscr{C}^{\mathsf{abs}}_{\kappa}(r)$, we can assume without loss of generality that $v = X_1[a_1] \hdots X_z[a_z].u$, with $u$ satisfying $\deg u = 0$ and $\hat{\g}_{\geq r}.u = 0$, for  some $X_1,\hdots,X_z \in \g$ and $a_1 \leq \hdots \leq a_z < r$. Hence $\deg v = |a_1| + \hdots + |a_z|$ (by Definition \ref{filtration Cabsr definition}.b)). 

We first prove the inequality
\[ |p+1+a_1| - |a_1| \leq -p+2r-3.\]
First assume $a_1 \leq -p-1$. Then $|p+1+a_1| - |a_1| = -(p+1+a_1) +a_1=-p-1$. But $-p-1 \leq -p+2r-3$ since $r \geq 2$. 
Next assume $-p\leq a_1 \leq -p+r-2<0$. Then $|p+1+a_1| - |a_1| = p+2a_1+1\leq -p+2r-3$.

We argue by induction on $z$ (i.e.\ by induction on the PBW filtration). If $z=1$ then 
\begin{equation} \label{diff in degree} e_{lk}[p+1].v = X_1[a_1]e_{lk}[p+1].u + [e_{lk},X_1][p+1+a_1].u = [e_{lk},X_1][p+1+a_1].u\end{equation} 
modulo $\mathbf{1}$. 
Note that $[e_{lk},X_1][p+1+a_1].u=0$ unless $a_1 \leq -p+r-2$.  
Let us now calculate the difference in degree between $v$ and \eqref{diff in degree}. We have
\[ \deg e_{lk}[p+1].v - \deg v = |p+1+a_1| - |a_1| \leq -p+2r-3. \]
Hence $e_{lk}[p+1]^{(\infty)}$ changes degree by at most $-p+2r-3$, as required. 

Now let $z>1$. We have $$e_{lk}[p+1].v = X_1[a_1]e_{lk}[p+1].v' + [e_{lk},X_1][p+1+a_1].v'$$
modulo $\mathbf{1}$, where $v' = X_2[a_2]\hdots X_z[a_z].u$ and $\deg v' = |a_2| + \hdots + |a_z|$.  
By induction, we know that $e_{lk}[p+1]$ changes the degree of $v'$ by at most $-p+2r-3$. Hence $$\deg X_1[a_1]e_{lk}[p+1].v' \leq (\deg v' + |a_1|) -p+2r -3 = \deg v -p+2r-3.$$  
Moreover, since $M$ is a filtered module, $[e_{lk},X_1][p+1+a_1]$ changes the degree of $v'$ by at most $|p+1+a_1|$. Hence 
\begin{align*}
\deg [e_{lk},X_1][p+1+a_1].v' \leq& \deg v' +|p+1+a_1| \\ =&  \deg v - |a_1| + |p+1+a_1| 
 \leq \deg v -p+2r-3.
\end{align*}
It follows that $e_{lk}[p+1]^{(\infty)}$ changes degree by at most $-p+2r-3$, as required.

We now show that $\mathsf{F}_{\kappa}^{(r)}$ is an enriched functor. 
Suppose that $M$ and $N$ are two filtered modules in $\mathscr{C}_{\kappa}^{\mathsf{abs}}(r)$. Let $h : M \to N$ be a filtered homomorphism of degree $i$. We need to show that $\mathsf{F}_{\kappa}(h)$ is also a filtered homomorphism of degree $i$. 
So let $v \in \mathsf{F}_{\kappa}(M)_{\leq s}$. Recall the projection \eqref{tensor to coinv projection}. 
Since $\mathsf{F}_{\kappa}(M)$ is endowed with the quotient filtration, we can choose $\tilde{v} \in\mathsf{T}_{\kappa}(M)_{\leq s}$ with $\psi(\tilde{v}) = v$. We can assume without loss of generality that $\tilde{v} = f(x_1,\hdots,x_m)\otimes u \otimes z$ with $u \in (\mathbf{V}^*)^{\otimes m}, z \in M$ and $f$ some polynomial. Since $h$ is filtered of degree $i$, we have $\mathsf{T}_{\kappa}(h)(\tilde{v}) = f(x_1,\hdots,x_m) \otimes u \otimes h(z) \in \mathsf{T}_{\kappa}(N)_{\leq s+i}$. However, $\psi' \circ \mathsf{T}_{\kappa}(h)(\tilde{v}) = \mathsf{F}_{\kappa}(h)(v)$, where $\psi'$ is the projection $\psi' : \mathsf{T}_{\kappa}(N) \twoheadrightarrow \mathsf{F}_{\kappa}(N)$. It follows that $\mathsf{F}_{\kappa}(h)(v) \in \mathsf{F}_{\kappa}(N)_{\leq s+i}$, as required. 
\end{proof}

In the following proposition assume that $\kappa = \mathbf{c}$, $m=n$ and 
consider the module $\mathbb{H}_{\mathbf{c}} = \mathbf{U}^{\mathsf{abs}}_{\mathbf{c}}(\tilde{\g})/(\mathfrak{I}_{\mathbf{c}} \cap \mathbf{U}^{\mathsf{abs}}_{\mathbf{c}}(\tilde{\g}))$ as a filtered $\mathbf{U}^{\mathsf{abs}}_{\mathbf{c}}(\tilde{\g})$-module endowed with the quotient filtration. 

\begin{pro} \label{pro: Upsilon is filtered} \label{Psi is filtered} 
The isomorphism $\Upsilon \colon \mathcal{H}_{0} \xrightarrow{\sim} \mathsf{F}_{\mathbf{c}}(\mathbb{H}_{\mathbf{c}})$ from \eqref{Regular module iso} lifts to an isomorphism in the category $\mathcal{H}_0\Fmod{}_0$. Moreover, the map $\Psi \colon \End_{\UUc}(\mathbb{H}_{\mathbf{c}}) \to \End_{\mathcal{H}_{0}}(\mathcal{H}_{0})$ is a filtered algebra homomorphism. 
\end{pro}

\begin{proof} 
Since it is difficult to work with quotient filtrations, we first show that $\mathsf{F}_{\mathbf{c}}(\mathbb{H}_{\mathbf{c}})$ is isomorphic to another module with a more explicit filtration. 
Consider the $\mathcal{H}_0$-module $\mathsf{T}_{\mathbf{c}}(\mathbb{H}_{\mathbf{c}})$. One easily checks that the subspace $M = \C[\h] \otimes ((\mathbf{V}^*)^{\otimes n})_{(-1,\hdots,-1)} \otimes \mathcal{I}$ is a $\mathcal{H}_0$-submodule of $\mathsf{T}_{\mathbf{c}}(\mathbb{H}_{\mathbf{c}})$. Moreover, it follows from Theorem \ref{thm: regular module} that $\mathsf{T}_{\mathbf{c}}(\mathbb{H}_{\mathbf{c}}) = M \oplus \g[t] \cdot \mathsf{T}_{\mathbf{c}}(\mathbb{H}_{\mathbf{c}})$ and $\mathsf{F}_{\mathbf{c}}(\mathbb{H}_{\mathbf{c}}) \cong M$. The latter isomorphism is filtered if we endow $\mathsf{F}_{\mathbf{c}}(\mathbb{H}_{\mathbf{c}})$ with the quotient filtration and $M$ with the subspace filtration. It follows from \eqref{Regular module iso 2} that composing $\Upsilon$ with $\mathsf{F}_{\mathbf{c}}(\mathbb{H}_{\mathbf{c}}) \cong M$ yields an $\mathcal{H}_0$-module isomorphism $\mathcal{H}_0 \cong M$ given by
\[ f(x_1,\hdots,x_n) w g(y_1,\hdots,y_n) \mapsto f(x_1,\hdots,x_n) \otimes e_{w}^* \otimes g(-e_{11}[1],\hdots,-e_{nn}[1])1_{\mathbb{H}}. \]
This formula together with the definition of the filtration on $\mathcal{H}_0$ and \eqref{tensor filtration} imply that the isomorphism $\mathcal{H}_0 \cong M$ is in fact filtered. This proves the first part of the proposition. 

The filtered isomorphism $\Upsilon^{-1}$ induces a filtered isomorphism of endomorphism rings $\End_{\mathcal{H}_{0}}(\mathsf{F}_{\mathbf{c}}(\mathbb{H}_{\mathbf{c}})) \cong \End_{\mathcal{H}_{0}}(\mathcal{H}_{0})$. But $\Psi$ is a composition of the latter with the homomorphism $\mathsf{F}_{\mathbf{c}} \colon \End_{\UUc}(\mathbb{H}_{\mathbf{c}}) \to \End_{\mathcal{H}_{0}}(\mathsf{F}_{\mathbf{c}}(\mathbb{H}_{\mathbf{c}}))$, which is filtered by Proposition \ref{pro: filt functor}. 
\end{proof}

\subsection{Graded lift of the Suzuki functor.} 
\label{subsec: graded lift} 

Suppose that $M$ is a graded module in $\mathscr{C}_{\kappa}^{\mathsf{gr}}$. Consider $(\mathbf{V}^*)^{\otimes m}$ as a graded vector space concentrated in degree zero. 
Endow $\mathsf{T}_{\kappa}(M)$ with the tensor product grading in analogy to \eqref{tensor filtration}. It follows immediately from \eqref{currentaction} that $\mathsf{F}_\kappa(M)$ is a quotient of $\mathsf{T}_{\kappa}(M)$ by a graded subspace. Hence the grading on $\mathsf{T}_{\kappa}(M)$ descends to a grading on $\mathsf{F}_\kappa(M)$. 

\begin{pro} \label{pro: coinv ind grad}
The functor $\mathsf{F}_\kappa$ lifts to a functor 
\begin{equation*} \label{filtered lift}
\mathsf{F}_{\kappa}^{\mathsf{gr}} \colon \mathscr{C}^{\mathsf{gr}}_{\kappa} \to \mathcal{H}_{\kappa+n}\Gmod{}
\end{equation*} 
enriched in $\C\Gmod{}_0$. 
\end{pro}

\begin{proof}
Let $M \in \mathscr{C}^{\mathsf{gr}}_{\kappa}$. 
We first prove that $\mathsf{F}_\kappa(M)$ is a graded $\mathcal{H}_{\kappa+n}$-module. 
It suffices to show that $y_i\mathsf{F}_{\kappa}(M)_s \subseteq \mathsf{F}_{\kappa}(M)_{s+1}$ for $s \in \Z$ and $1 \leq i \leq m$. Recall that the action of $y_i$ is given by \eqref{y-formula}. 
Clearly $\partial_{x_i}$ and $\Omega^{(i,j)}(x_i - x_j)^{-1}(1-\underline{s_{i,j}})$ either vanish or raise degree by one. 
Since $M$ is a graded $\mathbf{U}_\kappa(\tilde{\g})$-module, the same holds for $x_i^p\Omega_{[p+1]}^{(i,\infty)}$ for each $p \geq 0$, as required. 
The proof of the fact that $\mathsf{F}_{\kappa}^{\mathsf{gr}}$ is an enriched functor is analogous to the proof of Proposition \ref{pro: filt functor}. 
\end{proof}

\section{Surjectivity of $\Theta$} \label{sec: Theta surj}

In this section we show that $\Ima \Theta = \mathcal{Z}$. Assume that $n=m$ and $\kappa = \mathbf{c}$ throughout.

\subsection{The associated graded map.} \label{assoc graded sec}

Consider the following commutative diagram in the category of vector spaces. 
\begin{equation} \label{all filtr diagram}
\begin{tikzcd}[column sep=small]
\mathbb{H}_{\mathbf{c}}^{\mathfrak{I}_{\mathbf{c}}} \arrow[hookrightarrow,r] & \mathbb{H}_{\mathbf{c}}  \arrow[hookrightarrow]{r}{v \mapsto 1_{\C[\h]} \otimes e_{\mathsf{id}}^* \otimes v} &[+45pt] \mathsf{T}_{\mathbf{c}}(\mathbb{H}_{\mathbf{c}}) \arrow[twoheadrightarrow,rrr] & &[-19pt]  &[-30pt] \mathsf{F}_{\mathbf{c}}(\mathbb{H}_{\mathbf{c}}) \arrow{d}{\Upsilon^{-1}}[swap]{\wr} \\
 \End_{\UUc}(\mathbb{H}_{\mathbf{c}}) \arrow{u}{\wr} \arrow{rrr}{\Psi} & & & \End_{\mathcal{H}_{0}}(\mathcal{H}_{0}) &[-19pt] \cong &[-30pt] \mathcal{H}_0 \\
 \Ima \Phi  \arrow[hookrightarrow,u] \arrow{rrrrr}{\Psi} & & & &[-19pt]  &[-19pt] \mathcal{Z} \arrow[hookrightarrow,u]
\end{tikzcd}
\end{equation}
Note that the fact that $\Psi(\Ima \Phi) \subseteq \mathcal{Z}$ follows from Proposition \ref{centralizer pro}. 
We endow each of the vector spaces above with a filtration: 
\begin{itemize}[leftmargin=1cm]
\item $\mathbb{H}_{\mathbf{c}} = \mathbf{U}^{\mathsf{abs}}_{\mathbf{c}}(\tilde{\g})/(\mathfrak{I}_{\mathbf{c}} \cap \mathbf{U}^{\mathsf{abs}}_{\mathbf{c}}(\tilde{\g}))$ carries the quotient filtration 
and $\mathbb{H}_{\mathbf{c}}^{\mathfrak{I}_{\mathbf{c}}} \subset \mathbb{H}_{\mathbf{c}}$ has the subspace filtration,
\item $\End_{\UUc}(\mathbb{H}_{\mathbf{c}})$ carries the filtration induced by the one on $\mathbb{H}_{\mathbf{c}}$ and $\Ima \Phi \subset \End_{\UUc}(\mathbb{H}_{\mathbf{c}})$ has the subspace filtration,
\item $\mathsf{T}_{\mathbf{c}}(\mathbb{H}_{\mathbf{c}})$ has the filtration from \eqref{tensor filtration} and $\mathsf{F}_{\mathbf{c}}(\mathbb{H}_{\mathbf{c}})$ has the corresponding quotient filtration,  
\item $\mathcal{H}_0$ has the PBW filtration, $\End_{\mathcal{H}_{0}}(\mathcal{H}_{0})$ carries the induced filtration and $\mathcal{Z} \subset \mathcal{H}_0$ the subspace filtration.  
\end{itemize}

\begin{lem}
Each map in the diagram \eqref{all filtr diagram} is filtered. 
\end{lem}

\begin{proof}
Every map is filtered by definition except for $\Psi$ and $\Upsilon^{-1}$. The fact that the latter two are filtered follows from Proposition \ref{Psi is filtered}. 
\end{proof}

We will show that $\Ima \Theta = \mathcal{Z}$ by computing the associated graded algebra homomorphism
\begin{equation} \label{eq: gr Psi} 
\gr\Psi \colon \gr\Ima \Phi \to \gr\mathcal{Z}.
\end{equation} 
 
We split the task of computing \eqref{eq: gr Psi} into two parts. We first compute the principal symbols of the images of Segal-Sugawara operators in $\mathbb{H}_{\mathbf{c}}$. We then compute the images of these principal symbols under the associated graded of the map $\mathbb{H}_{\mathbf{c}} \to \mathcal{H}_0$ arising from the upper right corner of the diagram \eqref{all filtr diagram}. 

\subsection{Calculation of principal symbols.} 

The ideal $(\mathbf{U}(\hat{\g}_-)(\n_{+}\otimes t^{-1}\C[t^{-1}])) \cap \mathbf{U}(\hat{\g}_-)^{\ad \t}$ in $\mathbf{U}(\hat{\g}_-)^{\ad \t}$ is two-sided (see e.g.\ \cite{MolMuk}). Hence the corresponding projection 
\begin{equation*} \label{Affine HC homo} \mathsf{AHC} \colon \mathbf{U}(\hat{\g}_-)^{\ad \t} \twoheadrightarrow \mathbf{U}(\t\otimes t^{-1}\C[t^{-1}]) \end{equation*}
is an algebra homomorphism, often called the \emph{affine Harish-Chandra homomorphism}. Note that $\mathsf{AHC}$ is, moreover, a filtered homomorphism with respect to the PBW filtrations. 

\begin{lem} \label{cor:Tk decomposition} 
Let $1 \leq k \leq n$. The Segal-Sugawara vector $\mathbf{T}_k$ from Example \ref{exa: complete set T} can be written as
\begin{equation*} \mathbf{T}_k = \mathbf{P}_k + Q_k +Q_k', \end{equation*} 
where $\mathbf{P}_k := (e_{11}[-1])^k + \hdots + (e_{nn}[-1])^k$,  
\[ Q_k \in (\mathbf{U}(\tilde{\g}_-)_{-k}\cap\mathbf{U}^{\mathsf{pbw}}(\tilde{\g}_-)_{\leq k-1})^{\ad \t}, \quad Q'_k \in (\mathbf{U}(\tilde{\g}_-)_{-k}\cap\mathbf{U}^{\mathsf{pbw}}(\tilde{\g}_-)_{\leq k})^{\ad \t}\] and $Q'_k \in \ker \mathsf{AHC}$. 
\end{lem}

\begin{proof} 
Consider the algebra $\Uu(\doublehat{\g}_-)$ from Example \ref{exa: complete set T} equipped with a modified PBW filtration in which $\tau$ has degree zero. One easily sees that the principal symbol of $\Tr (E_\tau^k)$ equals $\Tr ((E^{(-1)})^k)$, where $E^{(-1)} := (e_{ij}[-1])_{i,j=1}^n$ is a matrix with coefficients in $S(\hat{\g}_-)$. But $\gr \mathsf{AHC}(\Tr ((E^{(-1)})^k)) = \mathbf{P}_k$. 
\end{proof} 

\begin{defi}
Suppose that $A \in \mathbf{U}(\hat{\g}_-)$. We write $A_{l} := A_{\langle -l-1 \rangle}$ so that $\mathbb{Y}\langle A,z\rangle = \sum_{l \in \Z} A_{l} z^l$ (note that the same notation was used with a different meaning in \eqref{kappa-L coeff}). In particular, for $1 \leq k \leq n$, we write $\mathbf{T}_{k,l} := \mathbf{T}_{k,\langle -l-1 \rangle}$ (not to be confused with $\mathbf{T}_{k;l}$ from Example \ref{exa: complete set T} ). We also write
\[ \widehat{A}_{l} := \widehat{\Phi}(A_{l}), \quad \overline{A}_{l}:=\sigma^{\mathsf{abs}}(\widehat{A}_{l}),\] 
where $\sigma^{\mathsf{abs}} \colon \mathbb{H}_{\mathbf{c}} \to \gr \mathbb{H}_{\mathbf{c}}$ is the principal symbol map with respect to the absolute height filtration and $$\widehat{\Phi} : \UUc \twoheadrightarrow \UUc/\UUc.\mathfrak{I}_{\mathbf{c}} = \mathbb{H}_{\mathbf{c}}$$ is the canonical map. If $v \in \mathbb{H}_{\mathbf{c}}$, set $\deg v := \deg \sigma^{\mathsf{abs}}(v)$.
\end{defi}

The proof of the following key proposition is rather technical and has been relegated to the appendix. 

\begin{pro} \label{S=E,T=P}
Let $1 \leq k \leq n$. Then: 
\[ \widehat{\mathbf{T}}_{k,l} = 0 \quad (l < -2k), \quad \quad  
\widehat{\mathbf{T}}_{k,-2k} = \widehat{\mathbf{P}}_{k,-2k} = \sum_{i=1}^n (e_{ii}[1])^k.1_{\mathbb{H}},\]
\[ \overline{\mathbf{T}}_{k, -2k+2+b} = \overline{\mathbf{P}}_{k,- 2k+2+b} = k\sum_{i=1}^ne_{ii}[-b-1](e_{ii}[1])^{k-1}.1_{\mathbb{H}} \ + \ (\mathbb{H}_{\mathbf{c}})_{\leq k+b-1} \quad (b \geq 0) .\] 
\end{pro}

\subsection{The main result.} 

Recall from Theorem \ref{pro: centre facts} that $\gr \mathcal{Z} = \C[\h \oplus \h^*]^{\mathfrak{S}_n}$. The latter is known as the ring of diagonal invariants or multisymmetric polynomials. Given $a,b \in \Z_{\geq 0}$, the multisymmetric power-sum polynomial of degree $(a,b)$ is defined as $\mathsf{p}_{a,b}:= x_1^ay_1^b + \hdots + x_n^ay_n^b$. We call $a+b$ the total degree of $\mathsf{p}_{a,b}$. 

\begin{pro} \label{diag inv gens}
The polynomials $\mathsf{p}_{a,b}$ with $a+b \leq n$ generate $\C[\h \oplus \h^*]^{\mathfrak{S}_n}$. 
\end{pro}

\begin{proof}
See, e.g., \cite[Corollary 8.4]{Rydh}. 
\end{proof}

We are ready to prove our main result: the surjectivity of $\Theta$. We also partially describe the kernel of $\Theta$, compute~$\Theta$ on Segal-Sugawara operators corresponding to $\mathbf{T}_1$ and $\mathbf{T}_2$, and compute the principal symbols of the images of ``higher-order" Segal-Sugawara operators under $\Theta$.

\begin{thm} \label{thm: Psi is surj} 
The map $\Theta \colon \mathfrak{Z} \to \mathcal{Z}$ is surjective with 
\begin{enumerate}[leftmargin=1cm,label=(\roman*)]
\item $\Theta(\mathbf{T}_{k,l}) = 0 \quad (l < -2k)$, 
\item $\Theta(\mathbf{T}_{1,l}) = \mathsf{p}_{l+1,0} \quad (l \geq 0)$, 
\item $\Theta(\mathbf{T}_{2,l}) = - \mathsf{p}_{l+3,1} + \sum_{i < j} 2c_{l+2}(x_i,x_j)s_{i,j} + ((n+1)l+ 3n +1) \sum_{i=1}^n x_i^{l+2} \quad (l \geq -2)$, 
\item $\Theta(\mathbf{T}_{k, -2k}) = (-1)^k \mathsf{p}_{0,k}$,
\item $\sigma(\Theta(\mathbf{T}_{k,-2k+2+b})) = (-1)^{k-1}k\mathsf{p}_{b+1,k-1} \quad (b \geq 0)$, \end{enumerate} 
where $1 \leq k \leq n$, $c_{r}(x_i,x_j)$ is the complete homogeneous symmetric polynomial of degree $r$ in $x_i$ and $x_j$, and $\sigma \colon \mathcal{Z} \to \gr \mathcal{Z}$ is the principal symbol map. 
\end{thm}

\begin{proof}
Part (i) follows directly from Proposition \ref{S=E,T=P}, while (ii)-(iii) follow from Lemma \ref{diagram factors trilemma} and the fact that $\mathbf{T}_2 = 2\cdot {}^{\mathbf{c}}\mathbf{L} + \id[-2]$.   
Proposition \ref{S=E,T=P} together with \eqref{Regular module iso 2} implies that $\Upsilon^{-1}$ sends 
\[ [1_{\C[\h]} \otimes e_{\mathsf{id}}^* \otimes \widehat{\mathbf{T}}_{k,-2k}] = [1_{\C[\h]} \otimes e_{\mathsf{id}}^* \otimes \sum_{i=1}^n (e_{ii}[1])^k.1_{\mathbb{H}}] \ \mapsto \ (-1)^k \mathsf{p}_{0,k}, \]
which proves (iv). 
Moreover, Proposition \ref{S=E,T=P} together with \eqref{currentaction} and \eqref{Regular module iso 2} implies that $\gr \Upsilon^{-1}$ sends
\[ [1_{\C[\h]} \otimes e_{\mathsf{id}}^* \otimes \overline{\mathbf{T}}_{k,-2k+2+b}] = k\sum_{i=1}^n x_i^{b+1} \otimes e_{\mathsf{id}}^* \otimes (e_{ii}[1])^{k-1}.1_{\mathbb{H}} \ \mapsto \ (-1)^{k-1}k\mathsf{p}_{b+1,k-1}, \]
which proves (v) because $\gr \Psi (\overline{\mathbf{T}}_{k,r}) = \sigma(\Theta(\mathbf{T}_{k,r}))$ for $r \geq -2k+2$. 

It follows from (iv) and (v) that the multisymmetric power-sum polynomials of total degree $\leq n$ all lie in the image of $\gr \Psi$. But, by Proposition \ref{diag inv gens}, these polynomials generate $ \C[\h \oplus \h^*]^{\mathfrak{S}_n} = \gr \mathcal{Z}$. Hence the map $\gr\Psi \colon \gr\Ima \Phi \to \gr\mathcal{Z}$ is surjective. By \cite[Lemma 1(e)]{Sjo}, the map $\Psi \colon \Ima \Theta \to \mathcal{Z}$ is surjective as well because the filtration on $\mathcal{Z}$ is exhaustive and discrete. The surjectivity of $\Theta = \Psi \circ \Phi$ follows. 
\end{proof}

\section{Applications and connections to other topics} \label{sec: apps}

We present several applications of Theorem \ref{thm: Psi is surj}. Assume that $n=m$ throughout. 

\subsection{Endomorphism rings and simple modules.} 

We prove that the homomorphisms between endomorphism rings of Weyl and Verma modules induced by the Suzuki functor are surjective and use this fact to show that every simple $\mathcal{H}_0$-module is in the image of~$\mathsf{F}_{\mathbf{c}}$. 

\begin{cor} \label{cor: surjective on endo rings} 
The functor $\mathsf{F}_{\mathbf{c}}$ induces surjective ring homomorphisms: 
\begin{equation} \label{eq: End standard Weyl} \mathsf{F}_{\mathbf{c}} \colon \End_{\UUc} (\mathbb{W}_{\mathbf{c}}(a,\lambda)) \to \End_{\mathcal{H}_{0}}(\Delta_{0}(a,\lambda)), 
\end{equation} 
for $l \geq 1$, $\nu \in \mathcal{C}_l(n)$, $\lambda \in \mathcal{P}_n(\nu)$ and $a \in \h^*$ with $\mathfrak{S}_n(a) = \mathfrak{S}_\nu$; and 
\begin{equation} \label{eq: End standard Verma} \mathsf{F}_{\mathbf{c}} \colon \End_{\UUc} (\mathbb{M}_{\mathbf{c}}(\lambda)) \to \End_{\mathcal{H}_{0}}(\Delta_{0}(\lambda)), 
\end{equation} 
for $\lambda \in \mathcal{P}(n)$. Moreover, the homomorphisms \eqref{eq: End standard Verma} are graded. 
\end{cor} 

\begin{proof}  
The existence of the ring homomorphisms \eqref{eq: End standard Weyl} and \eqref{eq: End standard Verma} follows from the fact that  $\mathsf{F}_{\mathbf{c}}(\mathbb{W}_{\mathbf{c}}(a,\lambda)) \cong \Delta_{0}(a,\lambda)$ (Theorem \ref{thm: standard to Weyl}) and $\mathsf{F}_{\mathbf{c}}(\mathbb{M}_{\mathbf{c}}(\lambda)) \cong \Delta_{0}(\lambda)$ (Theorem \ref{thm: Vermas to standards}). Let us prove their surjectivity. 
Corollary \ref{cor: endos vs functor} implies that we have a commutative diagram 
\[
\begin{tikzcd}
\mathfrak{Z} \arrow{r}{\Theta} \arrow{d}[swap]{can} & \mathcal{Z} \arrow{d}{can} \\
\End_{\UUc} (\mathbb{W}_{\mathbf{c}}(a,\lambda)) \arrow{r}{\mathsf{F}_{\mathbf{c}}} & \End_{\mathcal{H}_{0}}(\Delta_{0}(a,\lambda))
\end{tikzcd}
\]
By Theorem \ref{thm: Psi is surj}, $\Theta$ is surjective, and, by Theorem \ref{thm 2 bellamy}.b), the right vertical map is surjective as well. Hence the lower horizontal map must be surjective, too. The proof in the case of the Verma modules $\mathbb{M}_{\mathbf{c}}(\lambda)$ is analogous. The fact that \eqref{eq: End standard Verma} is a graded homomorphism follows from Proposition \ref{pro: coinv ind grad}. 
\end{proof}

We need the following lemma. 

\begin{lem} \label{lem: quotients by endos}
Let $M$ be a $\UUc$-module and $A \subseteq \End_{\UUc}(M)$ be a vector subspace. Then 
\[ \mathsf{F}_{\mathbf{c}}(M/A M) = \mathsf{F}_{\mathbf{c}}(M)/\mathsf{F}_{\mathbf{c}}(A)\mathsf{F}_{\mathbf{c}}(M). \]
\end{lem}

\begin{proof}
Let $\mathcal{B}$ be a basis of $A$. By definition, 
$M/AM = M/ \sum_{f \in \mathcal{B}} \Ima f$. 
Consider the exact sequence
\[ \bigoplus_{f \in \mathcal{B}} M \xrightarrow{\oplus_{f \in \mathcal{B}} f} M \to M/ \sum_{f \in \mathcal{B}} \Ima f \to 0.\]
By Remark \ref{remark colimits Suzuki}, 
the functor $\mathsf{F}_{\mathbf{c}}$ preserves colimits. In particular, it preserves (possibly infinite) direct sums and cokernels. 
Hence
\begin{align*} \textstyle  \mathsf{F}_{\mathbf{c}}(M/ \sum_{f \in \mathcal{B}} \Ima f) =& \ \mathsf{F}_{\mathbf{c}}(\coker (\oplus_{f \in \mathcal{B}} f))  \\ =& \ \coker (\oplus_{f \in \mathcal{B}} \mathsf{F}_{\mathbf{c}}(f)) = \textstyle \mathsf{F}_{\mathbf{c}}(M)/\sum_{f \in \mathcal{B}} \Ima \mathsf{F}_{\mathbf{c}}(f). \end{align*} 
But $\sum_{f \in \mathcal{B}} \Ima \mathsf{F}_{\mathbf{c}}(f) = \mathsf{F}_{\mathbf{c}}(A) \mathsf{F}_{\mathbf{c}}(M)$. 
\end{proof}

\begin{cor} \label{all simples}
Every simple $\mathcal{H}_0$-module is in the image of the functor $\mathsf{F}_{\mathbf{c}}$. 
\end{cor}

\begin{proof}
Let $L$ be a simple $\mathcal{H}_0$-module. By Lemma \ref{lem: quotients of vermas}, there exists a generalized Verma module $\Delta_0(a,\lambda)$ such that $L \cong \Delta_0(a,\lambda)/I\cdot \Delta_0(a,\lambda)$ for some ideal $I \subset \End_{\mathcal{H}_0}(\Delta_0(a,\lambda))$. Let $J := \mathsf{F}_{\mathbf{c}}^{-1}(I) \subset \End_{\UUc} (\mathbb{W}_{\mathbf{c}}(a,\lambda))$. 
Corollary \ref{cor: surjective on endo rings} implies that $\mathsf{F}_{\mathbf{c}}(J) = I$. Hence, by Lemma \ref{lem: quotients by endos}, 
\[ \mathsf{F}_{\mathbf{c}}(\mathbb{W}_{\mathbf{c}}(a,\lambda)/J\cdot\mathbb{W}_{\mathbf{c}}(a,\lambda)) = \Delta_0(a,\lambda)/I\cdot \Delta_0(a,\lambda) \cong L. \qedhere \]
\end{proof}

\begin{rem}
When $\kappa \neq \mathbf{c}$, it has been shown (see \cite[Theorem 4.3]{Suz} and \cite[Theorem A.5.1]{VV}) that, under some mild assumptions, every simple $\mathcal{H}_{\kappa+n}$-module in category $\O(\mathcal{H}_{\kappa+n})$ is in the image of $\mathsf{F}_\kappa$. It is noteworthy that the proofs in \cite{Suz} and \cite{VV} employ very different techniques from those used by us in the $\kappa = \mathbf{c}$ case. 
\end{rem}

\subsection{Restricted Verma and Weyl modules.}

We are going to compute the Suzuki functor on restricted Verma and Weyl modules as well as their simple quotients. 

Consider the algebra $\mathscr{Z}$ from \eqref{centre A} equipped with the natural $\Z$-grading induced from $\mathbf{U}_{\mathbf{c}}(\hat{\g})$. In \cite[\S 3.2]{AF2}, Arakawa and Fiebig consider the \emph{restriction functor} 
\begin{equation} \label{restriction functor AF} \mathscr{C}_{\mathbf{c}} \to \mathscr{C}_{\mathbf{c}}, \quad M \mapsto \overline{M} := M / \sum_{0 \neq i \in \Z} \mathscr{Z}_i \cdot M. \end{equation}
This functor is right exact because it is left adjoint to the invariants functor $M \mapsto \underline{M} := \{ m \in M \mid z \cdot m = 0 \mbox{ for all } z \in \mathscr{Z}_i, i \neq 0\}$. Given $\mu \in \t^*$, in \cite[\S 3.5]{AF2}, Arakawa and Fiebig define the corresponding \emph{restricted Verma module}~as $\overline{\mathbb{M}}_{\mathbf{c}}(\mu)$. By \cite[Lemma 3.5]{AF2}, 
\[ \overline{\mathbb{M}}_{\mathbf{c}}(\mu) = \mathbb{M}_{\mathbf{c}}(\mu)/\mathscr{Z}_{-} \cdot \mathbb{M}_{\mathbf{c}}(\mu), \]
where $\mathscr{Z}_{-} = \bigoplus_{i < 0} \mathscr{Z}_i$. 

Consider $\mathbb{M}_{\mathbf{c}}(\mu)$ as a graded $\hat{\g}_{\mathbf{c}}$-module with the subspace $\C_{\lambda,1} \subset  \Ind^{\hat{\g}_{\kappa}}_{\hat{\b}_+}\C_{\lambda,1}$ lying in degree zero (or, equivalently, as a module over the Kac-Moody algebra $\hat{\g}_{\mathbf{c}} \rtimes \C\mathbf{d}$, where $[\mathbf{d},X[n]] = n \cdot X[n]$ for $X \in \g$, with $\mathbf{d}$ acting by zero on $\C_{\lambda,1}$). It is known (see, e.g., \cite[Proposition 9.2.c)]{Kac} that $\mathbb{M}_{\mathbf{c}}(\mu)$ has a unique graded simple quotient $\mathbb{L}(\mu)$. 

\begin{lem} \label{lemma non dom simples}
If $\mu \notin \mathcal{P}(n) \subset \t^*$ then $\mathsf{F}_{\mathbf{c}}(\mathbb{L}(\mu)) = 0$.  
\end{lem} 

\begin{proof}
By Theorem \ref{thm: Vermas to standards}, the module $\mathbb{M}_{\mathbf{c}}(\mu)$ is killed by $\mathsf{F}_{\mathbf{c}}$. Since $\mathsf{F}_{\mathbf{c}}$ is right exact, its quotient $\mathbb{L}(\mu)$ is killed as well.
\end{proof}

We also consider $\Delta_0(\lambda)$, for $\lambda \in \mathcal{P}(n)$, as a graded $\mathcal{H}_0$-module. It follows from \cite[Proposition 4.3]{Gor2} that $\Delta_0(\lambda)$ has a unique graded simple quotient $L_\lambda$ (not to be confused with $L(\lambda)$ from \S \ref{subsec: SW duality}). 
 
\begin{cor} \label{cor: res Vermas}
Let $\lambda \in \mathcal{P}(n)$. Then $\mathsf{F}_{\mathbf{c}}(\overline{\mathbb{M}}_{\mathbf{c}}(\lambda)) \cong \mathsf{F}_{\mathbf{c}}(\mathbb{L}(\lambda)) \cong L_\lambda$. 
\end{cor} 

\begin{proof}
Consider the short exact sequence 
\begin{equation} \label{K-filtration} 0 \to K \to \overline{\mathbb{M}}_{\mathbf{c}}(\lambda) \to \mathbb{L}(\lambda) \to 0. \end{equation} 
By \cite[Proposition 3.1]{DGK}, $K$ has a (possibly infinite) filtration 
\[ 0 = K_0 \subset K_1 \subset K_2 \subset ... \]
by submodules $K_i$ such that $K = \colim K_i$, and each $K_{i+1}/K_i$ is a graded shift of a highest weight module of some weight $\mu_i$. Next, it follows from \cite[Theorem 4.7(4)]{AF1} that none of the weights $\mu_i$ are equal to $\lambda$. Moreover, \cite[Lemma 4.2(5)]{AF1} implies that each $\mu_i$ is equal to $w \cdot \lambda= w(\lambda + \rho) - \rho$ with $e \neq w \in \mathfrak{S}_n$. In particular, none of the weights $\mu_i$ are dominant. 

Since $K_i$ is a graded shift of a highest weight module of weight $\mu_i$, there exists a surjection $\mathbb{M}_{\mathbf{c}}(\mu_i)[k] \twoheadrightarrow K_i$ for some $k \in \Z$. Because $\mu_i$ is not dominant, Theorem \ref{thm: Vermas to standards} implies that $\mathbb{M}_{\mathbf{c}}(\mu_i)[k]$ is killed by $\mathsf{F}_{\mathbf{c}}$. The right exactness of $\mathsf{F}_{\mathbf{c}}$, therefore, implies that $K_i$ is killed as well. 

It follows that every submodule in the filtration \eqref{K-filtration} is annihilated by $\mathsf{F}_{\mathbf{c}}$. However, Definition \ref{general definition of Suzuki} implies that $\mathsf{F}_{\mathbf{c}}$ preserves colimits. Therefore, $\mathsf{F}_{\mathbf{c}}(K) = \colim \mathsf{F}_{\mathbf{c}}(K_i) = \colim 0 = 0$. Hence, by another application of right exactness, we get that $\mathsf{F}_{\mathbf{c}}(\overline{\mathbb{M}}_{\mathbf{c}}(\lambda)) \cong \mathsf{F}_{\mathbf{c}}(\mathbb{L}(\lambda))$.

We next prove that $\mathsf{F}_{\mathbf{c}}(\overline{\mathbb{M}}_{\mathbf{c}}(\lambda)) \cong L_\lambda$.
Abbreviate $\mathbb{E}_\lambda := \Ima \mathscr{Z} \subset \End_{\UUc}(\mathbb{M}_{\mathbf{c}}(\lambda))$ and $E_\lambda := \End_{\mathcal{H}_0}(\Delta_0(\lambda))$. These rings are $\Z_{\leq0}$-graded. 
Let $\mathbb{E}_\lambda^- \lhd \mathbb{E}_\lambda$ and $E_\lambda^- \lhd E_\lambda$ denote their maximal graded ideals. 
It follows from the proof of Corollary \ref{cor: surjective on endo rings} that the restriction of \eqref{eq: End standard Verma} to $\mathbb{E}_\lambda$ is surjective (in fact, by \cite[Theorem 9.5.3]{Fre}, $\mathbb{E}_\lambda = \End_{\UUc}(\mathbb{M}_{\mathbf{c}}(\lambda))$, but we do not need to use this fact). Since \eqref{eq: End standard Verma} is a graded homomorphism, it follows that $\mathsf{F}_{\mathbf{c}}(\mathbb{E}_\lambda^-) = E_\lambda^-$. Therefore,  Lemma \ref{lem: quotients by endos} implies that 
\[ \mathsf{F}_{\mathbf{c}}(\overline{\mathbb{M}}_{\mathbf{c}}(\lambda)) = \mathsf{F}_{\mathbf{c}}(\mathbb{M}_{\mathbf{c}}(\lambda)/\mathbb{E}_\lambda^- \cdot \mathbb{M}_{\mathbf{c}}(\lambda)) = \Delta_0(\lambda)/E_\lambda^- \cdot \Delta_0(\lambda). \] 
Arguing as in the proof of Lemma \ref{lem: quotients of vermas}, one concludes that $\Delta_0(\lambda)/E_\lambda \cdot \Delta_0(\lambda) = L_\lambda$. 
\end{proof}

Given $\lambda \in \mathcal{P}(n)$, we define the corresponding \emph{restricted Weyl module} to be $\overline{\mathbb{W}}_{\mathbf{c}}(\lambda)$. Since $\mathscr{Z}_{+} = \bigoplus_{i > 0} \mathscr{Z}_i$ annihilates $\mathbb{W}_{\mathbf{c}}(\lambda)$, we have 
\[ \overline{\mathbb{W}}_{\mathbf{c}}(\lambda) = \mathbb{W}_{\mathbf{c}}(\lambda)/\mathscr{Z}_{-} \cdot \mathbb{W}_{\mathbf{c}}(\lambda). \] 

\begin{cor} \label{cor: restricted Weyl calc}
Let $\lambda \in \mathcal{P}(n)$. Then $\mathsf{F}_{\mathbf{c}}(\overline{\mathbb{W}}_{\mathbf{c}}(\lambda)) \cong L_\lambda$. 
\end{cor} 

\begin{proof}
Let $M(\lambda)$ denote the Verma module over $\g$ with highest weight $\lambda$. The canonical surjection $M(\lambda) \twoheadrightarrow L(\lambda)$ induces a surjection $\mathbb{M}_{\mathbf{c}}(\lambda) =  \Ind^{\hat{\g}_{\kappa}}_{\hat{\g}_+} M(\lambda) \twoheadrightarrow \Ind^{\hat{\g}_{\kappa}}_{\hat{\g}_+} L(\lambda) = \mathbb{W}_{\mathbf{c}}(\lambda)$. Let $K$ denote its kernel. The functor $\mathsf{F}_{\mathbf{c}}$ sends the exact sequence $0 \to K \to \mathbb{M}_{\mathbf{c}}(\lambda) \to \mathbb{W}_{\mathbf{c}}(\lambda) \to 0$ to the exact sequence $\mathsf{F}_{\mathbf{c}}(K) \to \Delta_0(\lambda) \xrightarrow{f} \Delta_0(\lambda) \to 0$. But $\Delta_0(\lambda)$ is a cyclic $\mathcal{H}_0$-module, so $f$ must be an isomorphism. It follows that $\mathsf{F}_{\mathbf{c}}(K) = 0$. Moreover, $\mathsf{F}_{\mathbf{c}}(\overline{K}) = 0$ because $\overline{K}$ is a quotient of $K$. 

Since the restriction functor \eqref{restriction functor AF} is right exact, we also have an exact sequence $\overline{K} \to \overline{\mathbb{M}}_{\mathbf{c}}(\lambda) \to \overline{\mathbb{W}}_{\mathbf{c}}(\lambda) \to 0$. The functor $\mathsf{F}_{\mathbf{c}}$ sends it to the exact sequence $0=\mathsf{F}_{\mathbf{c}}(\overline{K}) \to L_\lambda \to \mathsf{F}_{\mathbf{c}}(\overline{\mathbb{W}}_{\mathbf{c}}(\lambda)) \to 0$ because $\mathsf{F}_{\mathbf{c}}(\overline{\mathbb{M}}_{\mathbf{c}}(\lambda)) \cong L_\lambda$, by Corollary \ref{cor: res Vermas}. It follows that $\mathsf{F}_{\mathbf{c}}(\overline{\mathbb{W}}_{\mathbf{c}}(\lambda)) \cong L_\lambda$. 
\end{proof}

\subsection{Poisson brackets.} 

Suppose that $A$ is an algebraic deformation of an associative algebra $A_0$, i.e., $A$ is a free $\C[\hbar]$-algebra such that $A/\hbar A = A_0$. Then there is a canonical Poisson bracket on $Z(A_0)$, called the Hayashi bracket, given by
\[ \{a,b\} := \frac{1}{\hbar}[\tilde{a},\tilde{b}] \mod \hbar,\] 
where $\tilde{a},\tilde{b}$ are arbitrary lifts of $a$ and $b$, respectively. This Poisson bracket was introduced by Hayashi in  \cite{Hay}. 
Applying this construction to $\UUc$ and $\mathcal{H}_t$, we get Poisson brackets on $\mathfrak{Z}$ and $\mathcal{Z}$. 

\begin{lem} \label{lemma Heis Virasoro}
The vector space spanned by $1$, $\id[r]$ and ${}^{\mathbf{c}}\mathbf{L}_r$ is, under the Poisson bracket, a Lie subalgebra of $\mathfrak{Z}$ isomorphic to the semidirect product of the Heisenberg algebra with the Virasoro algebra. 
Moreover, the subspace spanned by $\id[r]$ and ${}^{\mathbf{c}}\mathbf{L}_{r+1}$ $(r \leq 0)$ is a Lie subalgebra. 
\end{lem}

\begin{proof}
This follows from, e.g., \cite[(3.1.3)]{Fre}. 
\end{proof}

By Lemma \ref{lemma Heis Virasoro}, the algebra $\mathscr{L}_{\mathbf{c}}$ from \eqref{subHeisVir} is a Poisson subalgebra of $\mathfrak{Z}$. 
Since the generators ${}^{\kappa}\mathbf{L}_{r+1}, \ \id[r]$ $(r \leq 0)$ of $\mathscr{L}_\kappa$ are defined for any $\kappa$, they have canonical lifts to $\widehat{\mathbf{U}}_{\C[t]}$. Let $\mathscr{L}_{\C[t]}$ be the $\C[t]$-subalgebra of $\widehat{\mathbf{U}}_{\C[t]}$ generated by them. The map $\rho_t \circ \mathsf{F}_\kappa|_{\mathscr{L}_{\kappa}^{op}}$ from Lemma \ref{diagram factors trilemma} also lifts to a map $\rho_{\C[t]} \circ \mathsf{F}_{\C[t]}|_{\mathscr{L}_{\C[t]}^{op}} \colon \mathscr{L}_{\C[t]}^{op} \to \mathcal{H}_{\C[t]}^{op}$. 

\begin{thm} \label{pro: Poisson homo}
The map $\Theta \colon \mathscr{L}_{\mathbf{c}} \to \mathcal{Z}$ is a homomorphism of Poisson algebras. 
\end{thm}

\begin{proof} 
It follows from Lemma \ref{diagram factors trilemma}.c), 
Theorem \ref{pro: question one} and Theorem \ref{thm restricted centres} that we can identify $\Theta|_{\mathscr{L}_{\mathbf{c}}}$ with $\rho_0 \circ \mathsf{F}_{\mathbf{c}}|_{\mathscr{L}_{\mathbf{c}}}$. Since $\Theta$ is an algebra homomorphism, it suffices to check that $\Theta$ preserves the Poisson bracket on multiplicative generators of $ \mathscr{L}_{\mathbf{c}}$. Let $a_{\mathbf{c}}, b_{\mathbf{c}}$ be any two of the generators ${}^{\mathbf{c}}\mathbf{L}_{r+1},\ \id[r]$ $(r \leq 0)$ and let $a$ and $b$ be their canonical lifts to $\mathscr{L}_{\C[t]}^{op}$. Let us interpret $a_{\mathbf{c}}$ and $b_{\mathbf{c}}$ as endomorphisms of $\UUc$. Then
\begin{align*} 
\Theta(\{a_{\mathbf{c}},b_{\mathbf{c}}\}) =& \ 
\rho_0 \circ \mathsf{F}_{\mathbf{c}}(\{a_{\mathbf{c}},b_{\mathbf{c}}\}) \\
=& - \ \rho_0 \circ \mathsf{F}_{\mathbf{c}}\left(\mathsf{spec}_{t=0}\left(\frac{1}{t}[a,b]\right)\right) \\
=& \ - \mathsf{spec}_{t=0}\left(\frac{1}{t}[\rho_{\C[t]} \circ \mathsf{F}_{\C[t]}(a), \rho_{\C[t]} \circ \mathsf{F}_{\C[t]}(b)]\right) \\ 
=& \    \{ \rho_0 \circ \mathsf{F}_{\mathbf{c}}(a_{\mathbf{c}}),\rho_0 \circ \mathsf{F}_{\mathbf{c}}(b_{\mathbf{c}})\} =  \{ \Theta(a_{\mathbf{c}}), \Theta(b_{\mathbf{c}}) \}
\end{align*} 

The second equality follows from the definition of the Poisson bracket. The third equality follows from the easily verifiable fact that $\mathsf{spec}_{t=0} \circ \rho_{\C[t]} \circ \mathsf{F}_{\C[t]} = \rho_0 \circ \mathsf{F}_{\mathbf{c}} \circ \mathsf{spec}_{t=0}$. The fourth equality follows from part b) of Lemma \ref{diagram factors trilemma}, which implies that $\rho_{\C[t]} \circ \mathsf{F}_{\C[t]}(a)$ and $\rho_{\C[t]} \circ \mathsf{F}_{\C[t]}(b)$ are, respectively, lifts of $\rho_0 \circ \mathsf{F}_{\mathbf{c}}(a)$ and $\rho_0 \circ \mathsf{F}_{\mathbf{c}}(b)$ to $\mathcal{H}_{\C[t]}^{op}$. The minus signs in the second and third lines arise because we work with lifts in the opposite algebras. 
\end{proof}

\begin{rem}
It would be interesting to know whether there exists a bigger subalgebra $\mathscr{L}_{\mathbf{c}} \subset A \subset \mathfrak{Z}$ such that $\Theta|_{A}$ is a homomorphism of Poisson algebras. 
\end{rem}

\begin{rem}
The image of the ``grading element'' ${}^{\mathbf{c}} \mathbf{L}_0$ under $\Theta$ is the so-called Euler element $\mathbf{eu}$ in $\mathcal{Z}$. Moreover, since ${}^{\mathbf{c}} \mathbf{L}_{1}, -2{}^{\mathbf{c}} \mathbf{L}_0, -{}^{\mathbf{c}} \mathbf{L}_{-1}$ form an $\mathfrak{sl}_2$-triple under the Poisson bracket, we obtain an $\mathfrak{sl}_2$-action on $\mathcal{Z}$. This action is not integrable, in contrast to the well-studied (\cite{BGin, BEG}) action of the $\mathfrak{sl}_2$-triple $\sum_i x_i^2$, $\mathbf{eu}$, $\sum_i y_i^2$. For example, the subspace of $\mathcal{Z}$ spanned by $\sum_i x_i^r$ $(r \geq 0)$ is isomorphic to the contragredient Verma module of weight zero while the subspace spanned by $\Theta(^{\mathbf{c}}\mathbf{L}_r)$ $(r \leq 1)$ is isomorphic to the contragredient Verma module of weight two. It would be interesting to know in more detail how $\mathcal{Z}$ decomposes under our $\mathfrak{sl}_2$-action. 
\end{rem}

\subsection{A description of $\Theta$ in terms of opers} \label{section: opers} 

We are going to show that $\Theta$ induces an embedding of the Calogero-Moser space into the space of opers on the punctured disc and describe some of its properties. Let us first introduce some notation. 
Set $\mathbb{D} := \Spec \C[[t]]$ and $\mathbb{D}^\times := \Spec \C((t))$. Let $B \subset G$ be the standard Borel subgroup and $N := [B,B]$.
 
The notion of a $G$-oper on $\mathbb{D}^\times$ was introduced by Drinfeld and Sokolov in \cite{DS}. It was later generalized by Beilinson and Drinfeld in \cite{BD} for arbitrary smooth curves. Roughly speaking, a $G$-oper is a triple consisting of a principal $G$-bundle, a connection as well as a reduction of the structure group to $B$, satisfying a certain transversality condition. 

We will work with  an explicit description of $G$-opers on $\mathbb{D}^\times$ from \cite[\S 3]{DS} in terms of certain operators (see also \cite[\S 4.2.2]{Fre}), which we now recall. 
Let $\Loc'_G(\mathbb{D}^\times)$ be the space of operators of the form 
\[ \nabla = \partial_t + u(t), \quad u(t) \in \g((t)).\] 
There is an action of $G((t))$ on $\Loc_G(\mathbb{D}^\times)$ by the rule $g \cdot (\partial_t + A(t)) = \partial_t + gA(t)g^{-1} - g^{-1}\partial_tg$. Elements of the orbit space $\Loc_G(\mathbb{D}^\times) = \Loc'_G(\mathbb{D}^\times)/G((t))$ are called $G$\emph{-local systems} on $\mathbb{D}^\times$. Let $\Op_G(\mathbb{D}^\times)$ be the space of $N((t))$-equivalence classes of operators of the form
\[ \nabla = \partial_t + p_{-1} + v(t), \quad v(t) \in \b((t)),\]
where $p_{-1} = e_{2,1} + \hdots + e_{n,n-1} \in \g$. 
Elements of $\Op_G(\mathbb{D}^\times)$ are called $G$\emph{-opers} on $\mathbb{D}^\times$. 
There is a natural map $\Op_G(\mathbb{D}^\times) \to \Loc_G(\mathbb{D}^\times)$ sending an $N((t))$-equivalence class to a $G((t))$-equivalence class. 
An oper has \emph{trivial monodromy} if it is in the $G((t))$-orbit of the local system $\partial_t$. Let $\Op_G(\mathbb{D}^\times)^{0}$ denote the subspace of opers with trivial monodromy. 

A $G$-oper on $\mathbb{D}$ with \emph{singularity of order} at most $r$ (see \cite[\S 3.8.8]{BD}), where $r \geq 1$, is an $N[[t]]$-equivalence class of operators of the form
\begin{equation} \label{opers singular} \nabla = \partial_t + t^{-r}(p_{-1} + v(t)), \quad v(t) \in \b[[t]]. \end{equation} 
Let $\Op^{\leqslant r}_G(\mathbb{D})$ be the space of all such $G$-opers. 
By \cite[Proposition 3.8.9]{BD}, the natural map $\Op^{\leqslant r}_G(\mathbb{D}) \to \Op_G(\mathbb{D}^\times)$ sending an $N[[t]]$-equivalence class of operators to their $N((t))$-equivalence class is injective. The space $\Op^{\leqslant r}_G(\mathbb{D})$ can be endowed with the structure of a scheme and $\Op_G(\mathbb{D}^\times)$ with the structure of an ind-scheme (see, e.g., \cite[\S3.1.11]{BD}). 

For an operator \eqref{opers singular}, its $r$-th residue ($r \geq 1$) is defined in \cite[\S 4.3]{FFTL} as $\Res_r(\nabla) := p_{-1} + v(0)$. Under conjugation by an element $A(t) \in N[[t]]$, $\Res_r(\nabla)$ is conjugated by $A(0)$. Hence the projection of $\Res_r(\nabla)$ onto $\g/G \cong \t/\mathfrak{S}_n$ (identified via the Chevalley isomorphism) is well defined, and we have a map 
\[ \Res_r \colon \Op^{\leqslant r}_G(\mathbb{D}) \to \t/\mathfrak{S}_n. \] 
For each $z \in \t/\mathfrak{S}_n$, let $\Op^{\leqslant r}_G(\mathbb{D})_z := \Res_r^{-1}(z).$ 

Let $\check{G}$ denote the Langlands dual of $G$. Let $\Op_{\check{G}}(\mathbb{D}^\times)$ be the space obtained by replacing all the algebraic groups and Lie algebras by their Langlands duals in the definitions above. Noting that $\check{\t} = \t^*$, let 
\[ \varpi \colon \t^* \to  \t^*/\mathfrak{S}_n = \check{\t}/\mathfrak{S}_n, \quad \vartheta \colon \g^* \to \g^*/G \cong  \t^*/\mathfrak{S}_n = \check{\t}/\mathfrak{S}_n\]
be the canonical projections. For $\lambda \in \Pi^+$, we abbreviate
\[ \Op_{\check{G}}^\lambda(\mathbb{D}) := \Op^{\leqslant 1}_{\check{G}}(\mathbb{D})_{\varpi(-\lambda-\rho)}^0.\]

We are next going to recall the connection between opers and the algebra $\mathfrak{Z}$. 
Consider $\mathfrak{Z}$ as a graded algebra, with the grading induced by the grading on $\UUc$, and, moreover, as a filtered algebra, with the filtration induced by the PBW filtration on $\UUc$. 
Let $\mathfrak{Z}^{\leqslant r}(\hat{\g})$ be the quotient of $\mathfrak{Z}$ by the ideal topologically generated by elements of graded degree $i$ and PBW degree $j$, satisfying $-i < j(1-r)$. 

\begin{thm} \label{opers FF theorem order Weyls} The following hold. 
\begin{enumerate}[label=\alph*), font=\textnormal,noitemsep,topsep=3pt,leftmargin=1cm]
\item There is a canonical algebra isomorphism
\begin{equation} \label{FF iso opers} \mathfrak{Z} \cong \C[\Op_{\check{G}}(\mathbb{D}^\times)]. \end{equation} 
\item The isomorphism \eqref{FF iso opers} induces, for each $r \geq 0$, isomorphisms
\[ \mathfrak{Z}^{\leqslant r}(\hat{\g}) \cong \C[\Op_{\check{G}}^{\leqslant r}(\mathbb{D})] \]
\item For each $\lambda \in \Pi^+$, the canonical map $\mathfrak{Z} \to \End_{\UUc} (\mathbb{W}_{\mathbf{c}}(\lambda))$ is surjective. Moreover, 
\[ \End_{\UUc} (\mathbb{W}_{\mathbf{c}}(\lambda)) \cong \C[ \Op_{\check{G}}^\lambda(\mathbb{D})].\]
\end{enumerate} 
\end{thm}

\begin{proof}
Part a) is \cite[Theorem 4.3.6]{Fre}, part b) is \cite[Proposition 3.8.6]{BD} and part c) is \cite[Theorem 9.6.1]{Fre}. 
\end{proof} 

\begin{defi} \label{Ur-module}
For $\chi \in \g^* \cong \g[r-1]^*$, let $\mathbb{I}_{r,\chi} := \Ind_{\hat{\g}_{\geq r-1} \oplus \C \mathbf{1}}^{\hat{\g}_{\mathbf{c}}} \C_\chi$, with $\hat{\g}_{\geq r-1}$ acting on $\C_\chi$ via 
$\hat{\g}_{\geq r-1} \twoheadrightarrow \g[r-1] \xrightarrow{\chi} \C$ 
and $\mathbf{1}$ acting as the identity. Set $\mathbb{U}_{r} := \mathbb{I}_{r+1,0}$. 
\end{defi}

\begin{thm}[{\cite[Theorem 5.6.(1)-(2)]{FFTL}}] \label{thm U I support} 
We have \[ \supp_{\mathfrak{Z}} \mathbb{U}_{r} \subseteq \Op_{\check{G}}^{\leqslant r}(\mathbb{D}), \quad  \supp_{\mathfrak{Z}} \mathbb{I}_{r,\chi} \subseteq \Op_{\check{G}}^{\leqslant r}(\mathbb{D})_{\vartheta(\chi)}.\] 
\end{thm}

Let us identify $\t^* \cong \h^*$ via \eqref{symT-Ch-iso} and $\t \cong \t[1], \ z \mapsto z[1]$. Recall the map $\pi$ and the varieties $\Omega_{\mathbf{a},\lambda}$ from~\S\ref{sec: supp Verma modules}. The following corollary gives a partial description of $\Theta$ in terms of opers. 

\begin{cor} \label{big cor on opers} The following hold. 
\begin{enumerate}[label=\alph*), font=\textnormal,noitemsep,topsep=3pt,leftmargin=1cm]
\item The map $\Theta \colon \mathfrak{Z} \to \mathcal{Z}$ induces a closed embedding 
\[ \Theta^* \colon \Spec \mathcal{Z} \hookrightarrow \Op_{\check{G}}(\mathbb{D})^{\leqslant 2}. \] 
\item Let $l \geq 1$, $\nu \in \mathcal{C}_l(n)$, $\lambda \in \mathcal{P}_n(\nu)$, $a \in \h^*$ with $\mathfrak{S}_n(a) = \mathfrak{S}_\nu$ and $\mathbf{a} = \varpi(a)$. 
We have
\begin{equation*} \label{Schubert cells vs opers} \Theta^*(\Omega_{\mathbf{a},\lambda}) \subseteq \Op_{\check{G}}^{\leqslant 2}(\mathbb{D})_{\mathbf{a}}.\end{equation*} 
Hence the following diagram commutes: 
\begin{equation} \label{opers CM diagram} 
\begin{tikzcd}
\Spec \mathcal{Z} \arrow[swap, "\pi"]{d} \arrow[hookrightarrow, "\Theta^*"]{r} & \Op_{\check{G}}(\mathbb{D})^{\leqslant 2} \arrow["\Res_2"]{d} \\ 
\h^*/\mathfrak{S}_n \arrow["\sim"]{r} & \t^*/\mathfrak{S}_n
\end{tikzcd}
\end{equation} 
\item If $\mathbf{a} = 0$ then 
\begin{equation} \label{Schubert cells vs opers 2} \Theta^*(\Omega_{\lambda}) \subseteq   \Op_{\check{G}}^\lambda(\mathbb{D}). \end{equation} 
\end{enumerate} 
\end{cor}

\begin{proof}
By Theorem \eqref{thm: Psi is surj}, $\Theta$ is surjective, so it induces a closed embedding $\Theta^* \colon \Spec \mathcal{Z} \hookrightarrow \Op_{\check{G}}(\mathbb{D}^\times)$. Corollary \ref{cor: endos vs functor} implies that \[ \Theta^*(\Spec \mathcal{Z}) = \Theta^*(\supp_{\mathcal{Z}}(\mathcal{H}_0)) \subseteq \supp_{\mathfrak{Z}} \mathbb{H}_{\mathbf{c}}.\]
Since $\mathbb{H}_{\mathbf{c}}$ is a quotient of $\mathbb{U}_{2}$, it follows from Theorem \ref{thm U I support} that 
\[ \supp_{\mathfrak{Z}} \mathbb{H}_{\mathbf{c}} \subseteq \supp_{\mathfrak{Z}} \mathbb{U}_{2} \subseteq \Op_{\check{G}}^{\leqslant 2}(\mathbb{D}).\] 
This proves part a). Let us prove part b). 
Corollary \ref{cor: endos vs functor} implies that 
\begin{equation} \label{supp Delta vs Weyl} \Theta^*(\supp_{\mathcal{Z}}(\Delta_{0}(a,\lambda)) \subseteq \supp_{\mathfrak{Z}}\mathbb{W}_{\mathbf{c}}(a,\lambda). \end{equation} 
If we take $\chi \in \g[1]^*$ with $\chi|_{\n_-[1]\oplus \n_+[1]} = 0$ and $\chi|_{\t[1]} = a$ then $\mathbb{W}_{\mathbf{c}}(a,\lambda)$ is a quotient of $\mathbb{I}_{2,\chi}$. Hence Theorem \ref{thm U I support} implies that 
\[ \supp_{\mathfrak{Z}} \mathbb{W}_{\mathbf{c}}(a,\lambda) \subseteq \supp_{\mathfrak{Z}} \mathbb{I}_{2,\chi} \subseteq \Op_{\check{G}}^{\leqslant 2}(\mathbb{D})_{\mathbf{a}}. \]
The commutativity of the diagram \eqref{opers CM diagram} now follows directly from Proposition \ref{pro pi fibre Omega cells}. Let us next prove part c). As a special case of \eqref{supp Delta vs Weyl}, we have $ \Theta^*(\supp_{\mathcal{Z}}(\Delta_{0}(\lambda)) \subseteq \supp_{\mathfrak{Z}}\mathbb{W}_{\mathbf{c}}(\lambda)$. Theorem \ref{opers FF theorem order Weyls}.c) implies that 
\[ \supp_{\mathfrak{Z}}\mathbb{W}_{\mathbf{c}}(\lambda) = \Op_{\check{G}}^\lambda(\mathbb{D}),\]
completing the proof. 
\end{proof}

\subsection{Extensions and differential forms.} \label{subsec:ext and diff forms}

Let $\kappa \in \C$. We are going to show that the first derived functor of $\mathsf{F}_{\kappa}$ vanishes on modules which admit a filtration by Weyl modules. We also formulate a conjecture that $\mathsf{F}_{\mathbf{c}}$ induces a map between certain extension algebras. 

We say that a $\UU$-module has a $\Delta$\emph{-filtration} if it has a finite filtration with each subquotient isomorphic to $\mathbb{W}_{\kappa}(\lambda)$ for some $\lambda \in \mathcal{P}(n)$. Let $\UU\Lmod{}_{\Delta}$ be the full subcategory of $\UU\Lmod{}$ consisting of modules with a $\Delta$-filtration. 

\begin{pro} \label{pro: F exact on delta}
We have $L^1\mathsf{F}_{\kappa}(M) = 0$ for all $M \in \UU\Lmod{}_{\Delta}$. 
Hence $\mathsf{F}_{\kappa}$ is exact on $\UU\Lmod{}_{\Delta}$.  
\end{pro} 

\begin{proof} 
Consider the augmentation map $\varepsilon \colon \mathbf{U}(\g) \to \C$. Tensoring with $\C$ over $\mathbf{U}(\mathfrak{sl}_n)$ we obtain a map $\varepsilon' \colon \mathbf{U}(\g) \otimes_{\mathbf{U}(\mathfrak{sl}_n)} \C \to \C$. Let $K := \ker \varepsilon'$. By \cite[Proposition VI.16.1]{HS}, adapted to the Lie algebra homology setting, we have a long exact sequence
\begin{equation} \label{les gl vs sl} H_1(\mathfrak{sl}_n,N) \xrightarrow{\mathsf{cores}} H_1(\g,N) \to N \otimes_{\mathbf{U}(\g)}K \to H_0(\mathfrak{sl}_n,N) \xrightarrow{\mathsf{cores}} H_0(\g,N) \to 0 \end{equation}
for any $U(\g)$-module $N$, where $\mathsf{cores}$ is the corestriction map. If $N$ is finite-dimensional then, by Whitehead's first lemma (see e.g.\  \cite[Proposition VII.6.1]{HS}), $H_1(\mathfrak{sl}_n,N) = 0$. If, moreover, the corestriction map $H_0(\mathfrak{sl}_n,N) \to H_0(\g,N)$ is an isomorphism, the long exact sequence \eqref{les gl vs sl} forces $H_1(\g,N) \cong N \otimes_{\mathbf{U}(\g)}K$. 

Now let $\lambda \in \mathcal{P}(n)$ and take $N = (\mathbf{V}^*)^{\otimes n} \otimes L(\lambda)$. We claim that the corestriction map is an isomorphism. We need to show that $\mathfrak{sl}_n \cdot N = \g \cdot N$, which is equivalent to showing that any trivial $\mathfrak{sl}_n$-submodule of $N$ is also trivial as a $\g$-module. If $\mu = \sum_i a_i \epsilon_i$ is a weight of $(\mathbf{V}^*)^{\otimes n}$ then $\phi(\mu) :=\sum_i a_i = - n$. Similarly, if $\mu$ is a weight of $L(\lambda)$, then $\phi(\mu)= n$. Hence, for any weight $\mu$ of $N$, we must have $\phi(\mu) = 0$. But a non-trivial $\g$-module which is trivial when restricted to $\mathfrak{sl}_n$ must have weights of the form $\chi = a \sum_i \epsilon_i$ for $0 \neq a \in \Z$, which implies that $\phi(\chi) \neq 0$. This proves the claim. 

It follows that
\begin{equation} \label{H1 is zero 2} H_1(\g, (\mathbf{V}^*)^{\otimes n} \otimes L(\lambda)) \cong (\mathbf{V}^*)^{\otimes n} \otimes L(\lambda)) \otimes_{\mathbf{U}(\g)} K.\end{equation} 
We can identify $K = \id\cdot\C[\id]$. Since the identity matrix $\id$ acts on $L(\lambda)$ by the scalar $n$, and on $(\mathbf{V}^*)^{\otimes n}$ by the scalar $-n$, it acts by zero on the tensor product $(\mathbf{V}^*)^{\otimes n} \otimes L(\lambda)$. Hence the RHS of \eqref{H1 is zero 2} is zero. 
It follows that
\begin{equation} \label{H1 is zero} H_1(\g, (\mathbf{V}^*)^{\otimes n} \otimes L(\lambda)) = 0.\end{equation} 

Since homology commutes with induction, using the tensor identity and arguing as in the proof of Proposition \ref{diagramHindInd}, one shows that 
\[ L^i\mathsf{F}_{\kappa}(\mathbb{W}_{\kappa}(\lambda)) = H_i(\g[t], \mathsf{T}_\kappa(\mathbb{W}_{\kappa}(\lambda))) = \C[\h] \otimes H_i(\g, (\mathbf{V}^*)^{\otimes n} \otimes L(\lambda)).\]  
Together with \eqref{H1 is zero}, this implies that $L^1\mathsf{F}_{\kappa}(\mathbb{W}_{\kappa}(\lambda)) = 0$. One shows that $L^1\mathsf{F}_{\kappa}(M) = 0$ for all $M \in \UU\Lmod{}_{\Delta}$ by induction on the length of the $\Delta$-filtration. 
\end{proof}

\begin{cor} \label{cor: ext to ext}
The functor $\mathsf{F}_{\kappa}$ induces a linear map 
\[ \Ext_{\UU}^1(M,M) \to \Ext_{\mathcal{H}_{\kappa+n}}^1(\mathsf{F}_{\kappa}(M),\mathsf{F}_{\kappa}(M))\]
for all $M$ in $\UU\Lmod{}_{\Delta}$.  
\end{cor}

\begin{proof}
This follows from Proposition \ref{pro: F exact on delta} because the category $\UU\Lmod{}_{\Delta}$ is closed under one-step extensions. 
\end{proof}

Corollary \ref{cor: ext to ext} admits, at least conjecturally, a geometric interpretation when $\kappa = \mathbf{c}$. Frenkel and Teleman consider in \cite{FT} the category of $(\UUc,G[[t]])$-bimodules. They conjecture, for $\mu \in \Pi^+$ (and prove for $\mu = 0$), that $\Ext_{\UUc,G[[t]]}^\bullet(\mathbb{W}_{\mathbf{c}}(\mu),\mathbb{W}_{\mathbf{c}}(\mu))$ is isomorphic to the algebra of differential forms on $\Op_{\check{G}}^\mu(\mathbb{D})$. Note that if this conjecture holds, the algebra of self-extensions is generated by $\Ext^1$. An analogous result for rational Cherednik algebras is proven in \cite[Corollary 4.2]{Bel2}, stating that $\Ext_{\mathcal{H}_0}^\bullet(\Delta_0(\lambda),\Delta_0(\lambda))$ is isomorphic to the algebra of differential forms on $\Omega_\lambda$, for $\lambda \in \mathcal{P}(n)$.

\begin{con} \label{the conjecture}
Let $\lambda \in \mathcal{P}(n)$. 
The functor $\mathsf{F}_{\mathbf{c}}$ induces a surjective algebra homomorphism
\[ \Ext_{\UUc,G[[t]]}^\bullet(\mathbb{W}_{\mathbf{c}}(\lambda),\mathbb{W}_{\mathbf{c}}(\lambda)) \to \Ext_{\mathcal{H}_{0}}^\bullet(\Delta_0(\lambda),\Delta_0(\lambda)), \]
which is given by the restriction of differential forms via the inclusion \eqref{Schubert cells vs opers 2}. 
\end{con}

\appendix

\section{Proof of Proposition \ref{S=E,T=P}.} \label{degree estimation chapter} 

We work in the following setup. 
Let $1 \leq a \leq k$, $j_1,\hdots,j_a \geq 1$ and $j_1 + \hdots + j_a = k$. Consider an element $C=X_1[-j_1] \hdots X_a[-j_a] \in \mathbf{U}(\hat{\g}_-)$, where $X_i \in \{ e_{rs} \mid 1 \leq r,s \leq n\}$. 
\begin{lem} \label{Estimation lemma} 
The following estimates hold. 
\begin{enumerate}[label=\alph*), font=\textnormal,noitemsep,topsep=3pt,leftmargin=1cm]
\item For arbitrary $C$ as above: 
\begin{itemize}
\item $\widehat{C}_l = 0$ \  if \ $l < - (k +a)$,
\item $\deg \widehat{C}_{ - (k +a)} \leq a$,
\item $\deg \widehat{C}_{ - (k +a)+1 } \leq a-1$,
\item $\deg \widehat{C}_{ - (k +a)+2+p } \leq a+p$ \ $(p \geq 0)$. 
\end{itemize}
\item
Moreover, if $C \in \ker \mathsf{AHC} \subset \mathbf{U}(\hat{\g}_-)^{\ad \t}$ then:
\begin{itemize}
\item $\widehat{C}_{ - (k +a)} = 0$, 
\item $\deg \widehat{C}_{ - (k +a)+2+p } \leq a+p -2$  \ $(p \geq 0)$. 
\end{itemize} 
\end{enumerate}
\end{lem}

\begin{proof} 
Recall the module $\mathbb{U}_2 = \widehat{\mathbf{U}}_{\mathbf{c}}(\hat{\g}) / I_2$ from Definition \ref{Ur-module}, where $I_2$ is the left ideal in $\widehat{\mathbf{U}}_{\mathbf{c}}(\hat{\g})$ generated by $\hat{\g}_{\geq 2}$. 
We will often make use of the fact that $[\hat{\g}_{\geq 0}, I_2] \subseteq I_2$. 
Let $\widetilde{\Phi} \colon \widehat{\mathbf{U}}_{\mathbf{c}}(\hat{\g}) \twoheadrightarrow \mathbb{U}_2$ be the canonical map and $\widetilde{C}_l = \widetilde{\Phi}(C_l)$. Below in steps 1-4 we will show, by induction on $a$ that part a) of the lemma holds with $\widehat{C}_l$ replaced by $\widetilde{C}_l$ $(l \in \Z)$. Since~$\mathbb{H}_{\mathbf{c}}$ is a quotient of $\mathbb{U}_2$, the estimates in part a) must then also hold for $\widehat{C}_l$. In steps 5-6 we will prove part b). 

\noindent
\textbf{1. The base case.} Let us first tackle the base case $a=1$. Then $C = X_1[-k]$ and, by definition, 
\[ \mathbb{Y} \langle C,z \rangle = \frac{1}{(k-1)!}\partial_z^{k-1}\mathbb{Y}\langle X_1[-1],z\rangle = \sum_{i \in \Z} \frac{(i+1)\dotsi(i+k-1)}{(k-1)!}X_1[-i-k] z^i.\]
Hence 
\begin{equation} \label{eq:base case} C_i = 
\frac{(i+1)\dotsi(i+k-1)}{(k-1)!}X_1[-i-k].
\end{equation}
In particular,
\begin{equation}
\label{eq: base case vanishing} C_i = 0 \quad \mbox{if} \quad i = -1,\hdots,-k+1.
\end{equation}
We now consider the four cases in the lemma. First suppose that $i < -(k+1)$. Since $-i-k>1$ and $X_1[b].1_{\mathbb{H}}=0$ for $b > 1$, formula \eqref{eq:base case} implies that $\widetilde{C}_i = C_i.1_{\mathbb{H}} = 0$. 

In the second and third cases we have $C_{-(k+1)} = (-1)^{k-1} kX_1[1]$ and $C_{-k} = (-1)^{k-1}X_1$. Hence $\deg \widetilde{C}_{-(k+1)} \leq 1$ and $\deg  \widetilde{C}_{-k} \leq 0$. Finally suppose that $i = -k + p +1$ with $p \geq 0$. Formula \eqref{eq:base case} implies that $C_i$ is a multiple of $X_1[-p-1]$ and so $\deg \widetilde{C}_i \leq p+1$.\\

\noindent
\textbf{2. The inductive case - notation.} Assume $a \geq 2$. Let us set $k' = j_2 + \hdots + j_a$ and $a'=a-1$. Set $A = X_1[-j_1]$ and $B = X_2[-j_2]\hdots X_a[-j_a]$. 
By definition of the normally ordered product we have
\begin{equation} \label{ABC equation} C_l = \sum_{\substack{r+s=l,\\ r \geq 0}} A_rB_s + \sum_{\substack{r+s=l, \\ s<0}} B_rA_s. \end{equation} 
Set $C_l^+ = \sum_{\substack{r+s=l\\ r \geq 0}} A_rB_s$ and $C_l^- = \sum_{\substack{r+s=l\\ s < 0}} B_rA_s$ so that $C_l = C_l^+ + C_l^-$. Also set $\widetilde{C}_l^+:= \widetilde{\Phi}(C_l^+)$ and $\widetilde{C}_l^-:= \widetilde{\Phi}(C_l^-)$. \\

\noindent
\textbf{3. The inductive case - $C_l^+$.} First suppose that $l < -(k+a)+2$. Consider any monomial $A_rB_s$ in $C_l^+$. Since $r \geq 0$, we have $s = l-r < -(k+a)+2 \leq -(k'+a')$. Therefore, by induction, $\widetilde{B}_s = 0$. Hence $\widetilde{C}_l^+ = 0$. This takes care of the first three cases. 

Now assume that $l = -(k+a)+2+p$ with $p \geq 0$. 
Since $r, j_1 \geq 0$, we get from \eqref{eq:base case} that $A_r$ is a scalar multiple of $X_1[-r-j_1]$. Hence $\deg A_r \leq |r+j_1| = r + j_1$. 

We now estimate the degree of $\widetilde{B}_s$. We have $s = l-r = -(k'+a') +2 + p - (r + j_1 + 1)$. There are four situations to consider. Firstly, suppose that $p \geq r + j_1 + 1$. Then, by induction (the fourth case), we conclude that $\deg \widetilde{B}_s \leq a' + p - (r + j_1 + 1)$. Hence $\deg \widetilde{\Phi}(A_rB_s) \leq \deg A_r + \deg \widetilde{B}_s \leq (r+j_1) + (a' + p - (r + j_1 + 1)) = a'+p-1 = a+p-2$. Secondly, suppose that $p = r + j_1$. Then $s= -(k'+a')+1$ and so, by induction (the third case), we have $\deg \widetilde{B}_s \leq a' - 1$. Hence $\deg \widetilde{\Phi}(A_rB_s) \leq \deg A_r + \deg \widetilde{B}_s \leq (r+j_1) + a' - 1 = a+p -2$. Thirdly, suppose that $p = r + j_1 - 1$. Then $s=-(k'+a')$ and so, by induction (the second case), we have $\deg \widetilde{B}_s \leq a'$. Hence $\deg \widetilde{\Phi}(A_rB_s) \leq \deg A_r + \deg \widetilde{B}_s \leq (r+j_1) + a' = a+p$. Finally, if $p < r+j_1 - 1$ then $s < -(k'+a')$. Hence $\widetilde{B}_s = 0$ and $\widetilde{\Phi}(A_rB_s) = 0$. Overall we conclude that $\deg \widetilde{C}_l \leq a+p$. \\

\noindent
\textbf{4. The inductive case - $C_l^-$.}
Regard $C_l^-$ as a sum of monomials $B_rA_s$ as in \eqref{ABC equation}. If $s \leq -j_1 - 2$ then, by \eqref{eq:base case}, $A_s$ is a scalar multiple of $X_1[b]$ with $b \geq 2$. Hence in both of these cases $B_rA_s \in I_2$. Therefore, it is enough to consider the cases $s = -j_1$ and $s = -j_1 - 1$. 

Suppose that $s = -j_1$. Then $A_s = (-1)^{j_1-1}X_1$. In particular, $\deg \widetilde{A}_s = 0$ and $[A_s, I_2] \subseteq I_2$. 
Firstly, assume that $l \leq -(k+a)$. Then $r = l-s \leq -(k+a+s) = -(k'+a')-1$. Hence, by induction, $B_r \in I_2$, and so $[B_r, A_s] \in I_2$. It follows that $B_rA_s = A_sB_r - [B_r,A_s] \in I_2$. 

For the remaining cases, note that 
\[ \widetilde{\Phi}(B_rA_s) = \widetilde{\Phi}(A_sB_r) - \widetilde{\Phi}([B_r,A_s]) = A_s \cdot \widetilde{B}_r - \widetilde{\Phi}([B_r,A_s]). \] 
We can write $B_r = y + z$ with $y \in I_2$, $z \in \mathbf{U}(\hat{\g}_{-}) \otimes \mathbf{U}(\n_{-}) \otimes \mathbf{U}(\n_+) \otimes \mathbf{U}(\t[1])$ and $\deg z = \deg \widetilde{B}_r$. Then $\widetilde{\Phi}([B_r,A_s]) = \widetilde{\Phi}([z,A_s])$ and $\deg \widetilde{\Phi}([z,A_s]) \leq 
\deg z = \deg \widetilde{B}_r$. Therefore, $\deg \widetilde{\Phi}(B_rA_s) \leq \widetilde{B}_r$. 

Secondly, assume that $l = -(k+a)+1$. Then $r = l-s = -(k'+a')$. Hence, by induction, $\deg \widetilde{B}_r \leq a' = a - 1$ and so we can conclude that $\deg \widetilde{\Phi}(B_rA_s) \leq \deg \widetilde{B}_r \leq a - 1$. 
Thirdly, assume that $l = -(k+a)+2$. Then $r = l-s = -(k'+a')+1$. Hence, by induction, $\deg \widetilde{B}_r \leq a'-1= a - 2$ and so $\deg \widetilde{\Phi}(B_rA_s) \leq \deg \widetilde{B}_r \leq a - 2 < a$. 
Fourthly, assume that $l = -(k+a)+2+p$ with $p>0$. Then $r = l-s = -(k'+a')+2+(p-1)$. Hence, by induction, $\deg \widetilde{B}_r \leq a' +p-1= a + p -2$ and so $\deg \widetilde{\Phi}(B_rA_s) \leq \deg \widetilde{B}_r \leq a + p -2 < a + p$.

Now suppose that $s = -j_1 - 1$. Then $A_s = (-1)^{j_1-1}j_1X_1[1]$, $\deg \widetilde{A}_s = 1$ and $[A_s, I_2] \subseteq I_2$. The proof of the case $l < - (k+a)$ is the same as in the second paragraph of this step. Moreover, the same argument as in the third paragraph shows that $\deg \widetilde{\Phi}(B_rA_s) \leq \widetilde{B}_r+1$. 

Firstly, assume that 
$l=-(k+a)$. Then $r = l-s = -(k'+a')$. Hence, by induction, $\deg \widetilde{B}_r \leq a' = a - 1$ and so $\deg \widetilde{\Phi}(B_rA_s) \leq \deg \widetilde{B}_r + 1 \leq a$. 
Secondly, assume that $l = -(k+a)+1$. Then $r = l-s = -(k'+a')+1$. Hence, by induction, $\deg \widetilde{B}_r \leq a'-1= a - 2$ and so $\deg \widetilde{\Phi}(B_rA_s) \leq \deg \widetilde{B}_r + 1 \leq a - 1$. 
Thirdly, assume that $l = -(k+a)+2+p$ with $p\geq0$. Then $r = l-s = -(k'+a')+2+p$. Hence, by induction, $\deg \widetilde{B}_r \leq a' +p= a + p - 1$ and so $\deg \widetilde{\Phi}(B_rA_s) \leq \deg \widetilde{B}_r + 1 \leq a + p$.    
This proves that $\widetilde{C}_l^-$ satisfies the required constraints and completes the proof of the first part of the lemma. \\

\noindent
\textbf{5. An auxiliary induction.} We claim that 
\begin{description}
\item [(C)] If $X_i \in \n_+ \oplus \n_-$, for some $1 \leq i \leq a$, then $C_{-(k+a)} \in \mathfrak{I}_{\mathbf{c}}$. 
\end{description}
If $a = 1$ then $C_{-(k+1)} = (-1)^{k-1} kX_1[1] \in \mathfrak{I}_{\mathbf{c}}$ since $X_1 \in \n_+ \oplus \n_-$. So suppose $a > 1$. Then, by part 3 of the proof, $C_{-(k+a)}^+ \in \mathfrak{I}_{\mathbf{c}}$. Let us show that $\widehat{C}_{-(k+a)}^-\in \mathfrak{I}_{\mathbf{c}}$ as well. Part~5 implies that it suffices to consider the monomial $B_rA_s$ in $C_{-(k+a)}^-$ with $s = -j_1 - 1$. Since $A_s = (-1)^{j_1-1}j_1X_1[1]$, we have $B_rA_s\in \mathfrak{I}_{\mathbf{c}}$ if $X_1 \in \n_+ \oplus \n_-$. 

Otherwise, $X_1 \in \t$ and $X_i \in \n_+ \oplus \n_-$ for some $2 \leq i \leq a$. 
Since $r = -(k'+a')$, induction gives $B_{r} \in \mathfrak{I}_{\mathbf{c}}$ and $B_r$ can be written as a (finite) sum $\sum_p Z_pY_p$ with $Z_p \in \mathbf{U}_{\mathbf{c}}(\tilde{\g})$ and $Y_p \in \mathfrak{i}$ or $Y_p = e_{qq} - 1$ for some $1 \leq q \leq n$. In the first case, we use the fact that, by Lemma \ref{lem:leviideal}, $\mathfrak{i}$ is an ideal in $\hat{\t}_+$. Since $A_s \in \t[1]$, we get $[Y_p,A_s] \in \mathfrak{i}$. In the second case, $[Y_p,A_s]=0$. It follows that $[B_r,A_s] \in \mathfrak{I}_{\mathbf{c}}$. Hence $B_rA_s = A_sB_r - [B_r,A_s] \in \mathfrak{I}_{\mathbf{c}}$ as well. \\

\noindent
\textbf{6. Part b) of the lemma.} 
We now prove part b) the lemma. First observe that in many parts of the proof so far we have already established the stronger inequalities in the second statement of the lemma without even using the assumption that $C \in \ker \mathsf{AHC}$. Let us consider all the remaining cases. The first such case appears in part~3 of the proof: $l = -(k+a)+2+p$ with $p = r + j_1 - 1$. In that case $s=-(k'+a')$. Since $C \in \ker \mathsf{AHC}$, the claim \textbf{(C)} implies that $B_{s} \in \mathfrak{I}_{\mathbf{c}}$ and so $\widehat{\Phi}(A_rB_s) = 0$. 

The second case appears in part 4 of the proof: $s = - j_1 - 1$ and $l = -(k+a)$. It follows directly from \textbf{(C)} that $\widehat{\Phi}(B_rA_s) = 0$. The third case also appears in part 4 of the proof: $s = - j_1 - 1$ and $l = -(k+a)+2+p$ with $p \geq 0$. In that case $A_s = (-1)^{j_1-1}j_1X_1[1]$. There are two possibilities. Either $X_1 \in \n_+ \oplus \n_-$ or $B_r \in \ker \mathsf{AHC}$. In the first case $\widehat{\Phi}(B_rA_s) = 0$. In the second case, by induction, $\deg \widehat{B}_r \leq a' +p-2= a + p - 3$. Part 4 implies that $\deg \widehat{\Phi}(B_rA_s) \leq \deg \widehat{B}_r + 1$. Hence $\deg \widehat{\Phi}(B_rA_s) \leq a + p - 2$, as required. 
This was the last case to consider. We have therefore completed the proof of the lemma.  
\end{proof}

Lemma \ref{Estimation lemma} directly implies the following. 

\begin{cor} \label{Estimation corollary} 
Suppose that either (i) $C \in\mathbf{U}(\tilde{\g}_-)_{-k}\cap \mathbf{U}^{\mathsf{pbw}}(\tilde{\g}_-)_{\leq k-1}$ or (ii) $C \in (\mathbf{U}(\tilde{\g}_-)_{-k}\cap\mathbf{U}^{\mathsf{pbw}}(\tilde{\g}_-)_{\leq k})^{\ad \t}$ and $C \in \ker \mathsf{AHC}$. 
Then: 
\[ \widehat{C}_l = 0 \quad (l \leq -2k), \quad \quad \deg \widehat{C}_{ - 2k+2+p } \leq k-2+p \quad (p \geq 0). \]
\end{cor}

\begin{lem} \label{Estimation lemma 3} 
We have: 
\begin{itemize}
\item $\widehat{\mathbf{P}}_{k,l} = 0$ if $l < - 2k$, \quad \quad \quad \quad \quad \quad \ \  $\bullet$ $\widehat{\mathbf{P}}_{k,- 2k} = \sum_{i=1}^n (e_{ii}[1])^k.1_{\mathbb{H}},$ 
\item 
$\widehat{\mathbf{P}}_{k,- 2k+1} = k\sum_{i=1}^n(e_{ii}[1])^{k-1}.1_{\mathbb{H}},$
$\quad \bullet$ if $b \geq 0$ then:  
\[ \overline{\mathbf{P}}_{k,- 2k+2+b} = k\sum_{i=1}^ne_{ii}[-b-1](e_{ii}[1])^{k-1}.1_{\mathbb{H}} \ + \ (\mathbb{H}_{\mathbf{c}})_{\leq k+b-1}.\] 
\end{itemize}
\end{lem}

\begin{proof}
The first case follows directly from Lemma \ref{Estimation lemma}. So consider the remaining three cases. 
Fix $1 \leq i \leq n$. Let $A = e_{ii}[-1]$, $B = (e_{ii}[-1])^{k-1}$ and $C = AB$. 
By Lemma \ref{Estimation lemma}, we have $B_s.1_{\mathbb{H}} = 0$ for $s < -2k +2$ and $A_s.1_{\mathbb{H}} = 0$ for $s < -2$. Hence \eqref{ABC equation} implies that
\[ \widehat{C}_{-2k} = B_{-2k+2}A_{-2}.1_{\mathbb{H}}, \quad \widehat{C}_{-2k+1} = B_{-2k+2}A_{-1}.1_{\mathbb{H}} +  B_{-2k+3}A_{-2}.1_{\mathbb{H}}.\] 
By induction, we know that 
$B_{-2k+2} = (e_{ii}[1])^{k-1}$ and $B_{-2k+3} = (k-1)(e_{ii}[1])^{k-2}$ 
modulo~$\mathfrak{I}_{\mathbf{c}}$. Hence $\widehat{C}_{-2k} = A_{-2}B_{-2k+2}.1_{\mathbb{H}} = (e_{ii}[1])^{k}.1_{\mathbb{H}}$ and 
\begin{align*} \widehat{C}_{-2k+1} =& \ B_{-2k+2}A_{-1}.1_{\mathbb{H}} +  B_{-2k+3}A_{-2}.1_{\mathbb{H}} \\ =& \ B_{-2k+2}.1_{\mathbb{H}} + A_{-2}B_{-2k+3}.1_{\mathbb{H}} 
= k (e_{ii}[1])^{k-1}.1_{\mathbb{H}}. 
\end{align*} 
This proves the second and third cases. Finally consider the fourth case. We have 
\[ \widehat{C}_{-2k+2+b} = \sum_{0 \leq s \leq b} A_s B_{-2k+2+b-s}.1_{\mathbb{H}} +  B_{-2k+3+b}A_{-1}.1_{\mathbb{H}} +  B_{-2k+4+b}A_{-2}.1_{\mathbb{H}}.\]
Lemma \ref{Estimation lemma} implies that $A_b B_{-2k+2}.1_{\mathbb{H}} + B_{-2k+4+b}A_{-2}.1_{\mathbb{H}}$ is the leading term of $ \widehat{C}_{-2k+2+b}$. 
By induction, we know that 
$\sigma^{\mathsf{abs}}(\widehat{B}_{-2k+4+b}) = (k-1)e_{ii}[-b-1](e_{ii}[1])^{k-2}.1_{\mathbb{H}}$ and $\widehat{B}_{-2k+2} = (e_{ii}[1])^{k-1}.1_{\mathbb{H}}.$ Hence $\sigma^{\mathsf{abs}}(\widehat{C}_{-2k+2+b}) = ke_{ii}[-b-1](e_{ii}[1])^{k-1}.1_{\mathbb{H}}$. 
Summing over $i=1,\hdots,n$ yields the lemma. 
\end{proof}

We can now prove Proposition \ref{S=E,T=P}.

\begin{proof}[Proof of Proposition \ref{S=E,T=P}] 
By Lemma \ref{cor:Tk decomposition}, we can write
\[ \widehat{\mathbf{T}}_{k,l} = \widehat{Q}_{k,l} + \widehat{Q}'_{k,l} + \widehat{\mathbf{P}}_{k,l}, \]
where $Q_k \in (\mathbf{U}(\tilde{\g}_-)_{-k}\cap\mathbf{U}^{\mathsf{pbw}}(\tilde{\g}_-)_{\leq k-1})^{\ad \t}$, $Q'_k \in (\mathbf{U}(\tilde{\g}_-)_{-k}\cap\mathbf{U}^{\mathsf{pbw}}(\tilde{\g}_-)_{\leq k})^{\ad \t}$ and \linebreak $\mathsf{AHC}(Q'_k)=0$. Hence
Corollary \ref{Estimation corollary} implies that 
$\widehat{Q}_{k,l} = \widehat{Q}'_{k,l} = 0$ for $l \leq -2k$ and 
$$\deg \widehat{Q}_{k,-2k + 2 + p} = \deg \widehat{Q}'_{k,-2k + 2 + p} \leq k+p-2$$ for $p \geq 0.$  
On the other hand, we know from Lemma \ref{Estimation lemma 3} that $\widehat{\mathbf{P}}_{k,l} = 0$ for $l < -2k$, $\deg \widehat{\mathbf{P}}_{k,- 2k} = k$ and $\deg \widehat{\mathbf{P}}_{k,-2k+2+p} = k+p$ for $p \geq 0$. It follows that $\widehat{\mathbf{T}}_{k,l} = 0$ if $l < -2k$, $\widehat{\mathbf{T}}_{k,-2k} = \widehat{\mathbf{P}}_{k,-2k}$ and that $\widehat{\mathbf{P}}_{k,l}$ is the leading term of $\widehat{\mathbf{T}}_{k,l}$ if $l \geq -2k+2$, as required. 
\end{proof}

\end{document}